\documentclass[11pt]{amsart}
\title[]{Fractional Brauer configuration algebras III: fractional Brauer graph algebras of type MS}
\author{Nengqun Li and Yuming Liu*}

\address{Nengqun Li
\newline School of Mathematics
\newline Liaoning Normal University
\newline Dalian 116029
\newline P.R.China}
\email{linengqun@lnnu.edu.cn}

\address{Yuming Liu
\newline School of Mathematical Sciences
\newline Laboratory of Mathematics and Complex Systems
\newline Beijing Normal University
\newline Beijing 100875
\newline P.R.China}
\email{ymliu@bnu.edu.cn}

\date{version of \today}
\usepackage{amsmath}
\usepackage{amssymb}
\usepackage{amsxtra}
\usepackage{abstract}
\usepackage{mathrsfs}
\usepackage{mathtools}
\usepackage[all]{xy}
\usepackage{amscd}
\usepackage{amsthm}
\usepackage[dvips]{graphicx}
\usepackage{ulem}
\usepackage{color}
\usepackage{appendix}
\usepackage{tikz}

\newtheorem{Thm}{Theorem}[section]
\newtheorem{Lem}[Thm]{Lemma}
\newtheorem{Def}[Thm]{Definition}

\newtheorem{Prop}[Thm]{Proposition}
\newtheorem{Ex1}[Thm]{Example}
\newtheorem{Rem1}[Thm]{Remark}

\newcommand{\lra}{\longrightarrow}

\newcommand{\ra}{\rightarrow}
\newcommand{\sdp}{\times\kern-.2em\vrule height1.1ex depth-.05ex}
\newcommand{\epi}{\lra \kern-.8em\ra}

 \normalbaselines
\begin{document}
\renewcommand{\thefootnote}{\alph{footnote}}
\setcounter{footnote}{-1} \footnote{* Corresponding author.}
\setcounter{footnote}{-1} 
\renewcommand{\thefootnote}{\alph{footnote}}
\setcounter{footnote}{-1} 

\maketitle

\begin{abstract}
In previous two papers, we defined fractional Brauer configuration algebras and developed their covering theory. In this paper, we study the representation theory of fractional Brauer graph algebras of type MS, a special class of fractional Brauer configuration algebras that properly generalizes Brauer graph algebras. We first introduce the notion of Brauer $G$-set, which is a generalization of fractional Brauer graph of type MS. Then we develop a covering theory for Brauer $G$-sets and use it to characterize the representation types of fractional Brauer graph algebras of type MS. Moreover, we describe the AR-components of representation-finite and domestic fractional Brauer graph algebras of type MS respectively.
\end{abstract}

\section{introduction}

Aiming at a generalization of Brauer configuration algebras (abbr. BCAs) we introduced in \cite{LL} a class of locally bounded quiver algebras called fractional Brauer configuration algebras (abbr. f-BCAs). It was shown that f-BCAs of type S (abbr. $f_{s}$-BCAs) are locally bounded Frobenius algebras, and over an algebraically closed field, the representation-finite $f_{s}$-BCAs coincide with standard representation-finite basic self-injective algebras. We also developed in \cite{LL2} a covering theory for f-BCAs. In the present paper we concentrate on a subclass of $f_{s}$-BCAs, the fractional Brauer graph algebras of type MS (abbr. $f_{ms}$-BGAs), which are a proper generalization of Brauer graph algebras (abbr. BGAs). According to \cite[Proposition 6.6]{LL}, $f_{ms}$-BGAs are self-injective special biserial and therefore have tame representation type.

It is well-known that over an algebraically closed field, Brauer graph algebras coincide with symmetric special biserial algebras and whose representation types are classified in terms of the defining Brauer graphs (abbr. BGs) as follows (here we view a BG as an f-BC, see \cite[Section 3]{LL} or Section 2 below).

\begin{Thm}{\rm(cf. \cite{BS} and \cite{S})} \label{domestic-BGA}
Suppose that the base field is algebraically closed. Let $E$ be a finite connected Brauer graph and $A_E$ be the corresponding BGA. Then
\begin{itemize}
\item [$(1)$] $A_E$ is representation-finite if and only if $E$ is a Brauer tree (abbr. BT).
\item [$(2)$] $A_E$ is $1$-domestic if and only if one of the following holds
\item The underlying graph of $E$ is a tree, with two vertices f-degree $2$ and others f-degree $1$.
\item $E$ is f-degree-free and the underlying graph of $E$ has a unique cycle of odd length.
\item [$(3)$] $A_E$ is $2$-domestic if and only if $E$ is f-degree-free and the underlying graph of $E$ has a unique cycle of even length.
\item [$(4)$] $A_E$ is not of polynomial growth in other cases.
\end{itemize}
\end{Thm}

Moreover, recently Duffield determined explicitly in \cite{D} the Auslander-Reiten components (abbr. AR-components) of BGAs in terms of the defining BGs. (Note that about thirty years ago Erdmann and Skowro\'nski had already obtained a general description in \cite{ES} on the representation types and the AR-components of self-injective special biserial algebras.) Since
$f_{ms}$-BGAs ($\subseteq$ self-injective special biserial algebras) are generalization of BGAs ($\subseteq$ symmetric special biserial algebras), it is natural to ask how to characterize the representation types and AR-components of $f_{ms}$-BGAs.

To achieve this, our first idea is to use the covering theory for $f_{ms}$-BGs developed in \cite{LL2} to reduce an $f_{ms}$-BG to a smaller $f_{ms}$-BG (and hopefully to a BG). However, since a quotient of an $f_{ms}$-BG by a group of automorphisms may not again be an $f_{ms}$-BG, we need a variation of the notion $f_{ms}$-BG. Therefore we define Brauer $G$-sets, which is generalization of $f_{ms}$-BGs. We further develop a covering theory for Brauer $G$-sets so that the natural projection of any Brauer $G$-set to its quotient becomes a covering of Brauer $G$-sets, and we define the fundamental groups of Brauer $G$-sets. It should be noted that the quotient of an $f_{ms}$-BG $E$ by a group $\Pi$ of automorphisms can be viewed as an $f_{ms}$-BC, but if the action of $\Pi$ on $E$ is not admissible, then the natural quotient $E\rightarrow E/\Pi$ is not a covering of f-BCs in the sense of \cite{LL2}. In our reduction process in this paper, we frequently use nonadmissible actions, this is why we need a more flexible covering theory in the situation of Brauer $G$-sets.

We then apply the covering theory for Brauer $G$-sets to study the representation theory of $f_{ms}$-BGAs. Our first application is to determine the representation type of an $f_{ms}$-BGA. For a finite connected $f_{ms}$-BG $E$, we consider the quotient $E/\langle\sigma\rangle$, which is a special kind of Brauer $G$-sets called modified Brauer graphs (abbr. modified BGs), where $\langle\sigma\rangle$ is the group of automorphisms of $E$ generated by the Nakayama automorphism $\sigma$. Using $E/\langle\sigma\rangle$ we define the reduced form $R_E$ of $E$, which is a BG. So both the $f_{ms}$-BGAs $A_E$ and $A_{R_E}$ are special biserial algebras and their representation types are determined by the numbers of equivalence classes of bands over them respectively. Therefore we define lines and bands in a Brauer $G$-set and connect the bands of an $f_{ms}$-BG $E$ with the bands of associated $f_{ms}$-BGA $A_E$. By applying the covering theory of Brauer $G$-sets we show that the $f_{ms}$-BGA $A_E$ and the BGA $A_{R_E}$ have the same representation type:

\begin{Thm} {\rm(see Theorem \ref{rep type of f-BCA and its reduced form})}
Let the base field be algebraically closed. Let $E$ be a finite connected $f_{ms}$-BG. Then $A_E$ is representation-finite (resp. domestic) if and only if $A_{R_E}$ is representation-finite (resp. domestic).
\end{Thm}

However, some new phenomenon appears here, an $n$-domestic BGA must have $n\leq 2$, but an $n$-domestic $f_{ms}$-BGA can have arbitrarily large $n$ (see Remark \ref{description-in-terms-of-brauer-graph}).

We also characterize the representation type of an $f_{ms}$-BGA in terms of the fundamental group of its defining $f_{ms}$-BG.

\begin{Thm}{\rm(see Theorem \ref{fundamental group of rep-finite f-BGA} and Theorem \ref{Theorem-fundamental-group-of-domestic-fms-BGA})}
Let the base field be algebraically closed. Let $E$ be a finite connected $f_{ms}$-BG. Then
\begin{itemize}
\item[(1)] $A_E$ is representation-finite if and only if $\Pi(E)\cong\mathbb{Z}$;
\item[(2)] $A_E$ is domestic if and only if $\Pi(E)\cong F\langle a,b\rangle/\langle a^2=b^2\rangle$ or $\Pi(E)\cong\mathbb{Z}\oplus\mathbb{Z}$, where $F\langle a,b\rangle$ denotes the free group on the set $\{a,b\}$.
\end{itemize}
\end{Thm}

The proofs of the above theorems are largely based on the computation of fundamental groups of modified BGs (see Proposition \ref{modified fundamental group of modified f-BG}), which are again based on an analogy of the Van Kampen theorem of  Brauer $G$-sets (see Proposition \ref{modified Van-Kampen}).

Our second application is to determine the AR-quivers of representation-finite $f_{ms}$-BGAs in terms of defining $f_{ms}$-BGs. Assume that $E$ is a finite connected $f_{ms}$-BG with $A_E$ representation-finite, then we can describe the AR-quiver of $A_E$ by using the modified BG $E/\langle\sigma\rangle$ and the image of the homomorphism $\Pi(E)\rightarrow\Pi_{m}(E/\langle\sigma\rangle)$ between the fundamental groups of $E$ and $E/\langle\sigma\rangle$ respectively (see Theorem \ref{case (a): AR-quiver of A_E} and Theorem \ref{case (b): AR-quiver of A_E}).

Our third application is to construct the defining $f_{ms}$-BGs of domestic $f_{ms}$-BGAs and describe the stable AR-components of these algebras. Let $E$ be a finite connected $f_{ms}$-BG and let $B$ be the modified BG $E/\langle\sigma\rangle$. We first describe the shape of $B$ when $A_E$ is domestic.

\begin{Lem} {\rm(see Lemma \ref{B})} \label{B-introduction}
Let the base field be algebraically closed. Suppose that the modified BG $B$ has $k$ edges, $l$ double half-edges, and $n$ vertices $v_1$, $\cdots$, $v_n$ of f-degree $d_1$, $\cdots$, $d_n$ respectively.  Then $A_E$ is domestic if and only if one of the following conditions holds:
\begin{itemize}
\item[$(1)$] $l=2$, $k-n+1=0$, $d_i=1$ for $1\leq i\leq n$;
\item[$(2)$] $l=0$, $k-n+1=0$, $d_i=2$ for exactly two numbers $i=i_0$, $i_1$, and $d_i=1$ for $i\neq i_0$, $i_1$;
\item[$(3)$] $l=0$, $k-n+1=1$, $d_i=1$ for $1\leq i\leq n$.
\end{itemize}
\end{Lem}

Note that in Cases $(2)$ and $(3)$ $B$ is a BG, where in Case $(2)$ the underlying graph of $B$ is a tree with two vertices f-degree $2$ and others f-degree $1$, and in Case $(3)$ the underlying graph of $B$ has a unique cycle such that the f-degree of each vertex of $B$ is $1$.

In Case $(1)$ (resp. $(2)$), we show that the order of the Nakayama automorphism $\sigma$ of $E$ is even (resp. odd), and $E$ is determined by $B$ and the order of $\sigma$ up to isomorphism (see Lemma \ref{determined-up-to-isomorphism-1} and Lemma \ref{determined-up-to-isomorphism-2}). In Case $(3)$, we suppose the order of the Nakayama automorphism $\sigma$ of $E$ is $r$; we construct $r$ $f_{ms}$-BGs $E_{r1},\cdots,E_{rr}$, and show that $E$ is isomorphic to one of them (see Proposition \ref{construction-of-E-in-case-(3)}).

We describe the stable AR-components of domestic $f_{ms}$-BGAs as follows.

\begin{Thm} {\rm(see Propositions \ref{stable-AR-component-case-(1)}, \ref{stable-AR-component-case-(2)}, \ref{stable-AR-component-case-(3)})}
Let the base field be algebraically closed. Let $E$ be a finite connected $f_{ms}$-BG with $A_E$ domestic. Denote by $B=E/\langle\sigma\rangle$, where $\sigma$ is the Nakayama automorphism of $E$, and denote by $\prescript{}{s}{\Gamma}_{A_E}$ the stable Auslander-Reiten quiver of $A_E$. Then
\begin{itemize}
\item if $B$ is as in case $(1)$ of Lemma \ref{B-introduction}, and suppose that the order of $\sigma$ is $2r$. Then $\prescript{}{s}{\Gamma}_{A_E}$ is a disjoint union of $4r$ components of the form $\mathbb{Z}A_{\infty}/\langle\tau^{n}\rangle$, $2r$ components of the form $\mathbb{Z}\widetilde{A}_{n,n}$, and infinitely many components of the form $\mathbb{Z}A_{\infty}/\langle\tau\rangle$;
\item if $B$ is as in case $(2)$ of Lemma \ref{B-introduction}, and suppose that the order of the Nakayama automorphism of $E$ is $2r-1$. Then $\prescript{}{s}{\Gamma}_{A_E}$ is a disjoint union of $4r-2$ components of the form $\mathbb{Z}A_{\infty}/\langle\tau^{n-1}\rangle$, $2r-1$ components of the form $\mathbb{Z}\widetilde{A}_{n-1,n-1}$, and infinitely many components of the form $\mathbb{Z}A_{\infty}/\langle\tau\rangle$;
\item if $B$ is as in case $(3)$ of Lemma \ref{B-introduction}, suppose that the order of $\sigma$ is $r$ and suppose that $E\cong E_{rl}$ for some $1\leq l\leq r$. Denote by $m$ the length of the unique cycle of the underlying diagram of $B$, $p$ (resp. $q$) the number of edges of the underlying diagram of $B$ outside (resp. inside) this cycle. If $m$ is odd, then $\prescript{}{s}{\Gamma}_{A_E}$ is a disjoint union of $(r,m+2l)$ components of the form $\mathbb{Z}A_{\infty}/\langle\tau^{\frac{r(m+2p)}{(r,m+2l)}}\rangle$, $(r,m+2l)$ components of the form $\mathbb{Z}A_{\infty}/\langle\tau^{\frac{r(m+2q)}{(r,m+2l)}}\rangle$, $(r,m+2l)$ components of the form $\mathbb{Z}\widetilde{A}_{\frac{r(m+2p)}{(r,m+2l)},\frac{r(m+2q)}{(r,m+2l)}}$, and infinitely many components of the form $\mathbb{Z}A_{\infty}/\langle\tau\rangle$. If $m$ is even, then $\prescript{}{s}{\Gamma}_{A_E}$ is a disjoint union of $(2r,m+2l)$ components of the form $\mathbb{Z}A_{\infty}/\langle\tau^{\frac{r(m+2p)}{(2r,m+2l)}}\rangle$, $(2r,m+2l)$ components of the form $\mathbb{Z}A_{\infty}/\langle\tau^{\frac{r(m+2q)}{(2r,m+2l)}}\rangle$, $(2r,m+2l)$ components of the form $\mathbb{Z}\widetilde{A}_{\frac{r(m+2p)}{(2r,m+2l)},\frac{r(m+2q)}{(2r,m+2l)}}$, and infinitely many
components of the form $\mathbb{Z}A_{\infty}/\langle\tau\rangle$, where $(a,b)$ denotes the greatest common divisor of $a$ and $b$.
\end{itemize}
\end{Thm}

In a forthcoming paper \cite{LL3}, we will further use covering theory of f-BCAs to study fractional Brauer configuration algebras of type MS (which are proper generalization of BCAs) and determine the tameness and wildness of these algebras.

This paper is organized as follows.

In Section 2 we introduce Brauer $G$-sets and study their covering theory; we define lines and bands for a Brauer $G$-set, and discuss the relations between the bands of an $f_{ms}$-BG and the bands of associated $f_{ms}$-BGA; we show that an $f_{ms}$-BGA $A_E$ is representation-finite (resp. domestic) if and only if the BGA $A_{R_E}$ of the reduced form $R_E$ of $E$ is representation-finite (resp. domestic). As a byproduct, we obtain some unexpected example of weakly symmetric $f_{ms}$-BGA which is not a BGA (Example \ref{weakly-symmetric-algebra-example}).

In Section 3 we first calculate the fundamental groups of modified BGs using an analogy of the Van Kampen theorem, and then together with covering theory for Brauer $G$-sets we calculate the fundamental groups of $f_{ms}$-BGs $E$ with $A_E$ representation-finite or domestic.

In Section 4 we describe the Auslander-Reiten quivers (abbr. AR-quivers) of representation-finite $f_{ms}$-BGAs in terms of defining $f_{ms}$-BGs and show that these algebras coincide with basic representation-finite self-injective algebras of class $A_n$.

In final section we construct the defining  $f_{ms}$-BGs of domestic $f_{ms}$-BGAs and describe their stable AR-components.

\section*{Data availability} The datasets generated during the current study are available from the corresponding author on reasonable request.

\section*{Acknowledgements} The first author is supported by the Doctoral Research Startup Fund of Liaoning Normal University (No.603260070011). We would like to thank Drew Duffield for answering some questions on the results in his paper \cite{D}. We would also like to thank Ibrahim Assem, Karin Erdmann and Sibylle Schroll for some correspondence on AR-components of self-injective special biserial algebras.

\section{Brauer $G$-sets and their covering theory}

Throughout this paper we assume that $k$ is a field and fix $G$ to be an infinite cyclic group generated by $g$, all categories will be locally bounded $k$-categories which are given by locally finite quivers with some relations, and all algebras will be finite dimensional $k$-algebras which are given by finite quivers with some relations. We write a path $p$ in a quiver from right to left and denote by $s(p)$ and $t(p)$ the source and the terminal of $p$, respectively.

\subsection{Review on fractional Brauer graph of type MS}

\begin{Def}\label{fms-BC} {\rm(cf. \cite[Section 3]{LL})}
A fractional Brauer configuration of type MS (abbr. $f_{ms}$-BC) is a quadruple $E=(E,P,L,d)$, where $E$ is a $G$-set, $P$, $L$ are two partitions of $E$ such that each class $P(e)$ is a finite set and each class $L(e)=\{e\}$ (for each $e$ in $E$, denote by $P(e)$ and $L(e)$ the equivalence classes of $e$ under the partitions $P$ and $L$, respectively) and $d: E\rightarrow \mathbb{Z}_{>0}$ is a function, such that
\begin{itemize}
    \item if $e_1$, $e_2$ belong to same $G$-orbit, then $d(e_1)=d(e_2)$;
    \item $P(e_1)=P(e_2)$ if and only if $P(g^{d(e_1)}\cdot e_1)=P(g^{d(e_2)}\cdot e_2)$.
\end{itemize}
Moreover, if each class $P(e)$ contains exactly two elements, then $E$ is called a fractional Brauer graph of type MS (abbr. $f_{ms}$-BG).
\end{Def}

Follows from \cite[Remark 3.4]{LL}, the elements of $E$ are called {\it angles}, the $G$-orbits of $E$ are called {\it vertices}, the subsets of $E$ of the form $P(e)$ are called {\it polygons} (if the cardinality of $P(e)$ is $n$ then we call $P(e)$ an {\it $n$-gon}. If $P(e)$ is a $2$-gon, then we also call $P(e)$ an {\it edge}, and call an angle in $P(e)$ a {\it half-edge}), and the function $d: E\rightarrow \mathbb{Z}_{>0}$ is called the {\it degree function}. If $v$ is a vertex of $E$ which is a finite set, define the {\it fractional-degree (abbr. f-degree)} $d_f(v)$ of $v$ as $\frac{d(v)}{\mid v\mid}$; $E$ is said to have {\it integral f-degree} if each vertex of $E$ is a finite set and the f-degree of each vertex of $E$ is an integer. $E$ is called {\it f-degree-free} if $d_f(v)\equiv 1$. The permutation $\sigma: E\rightarrow E$, $e\rightarrow g^{d(e)}\cdot e$ on $E$ is called the {\it Nakayama automorphism} of $E$.

Note that a finite $f_{ms}$-BC $E$ (that is, the $G$-set $E$ is finite) with integral f-degree and containing no $1$-gons is exactly a Brauer configuration (abbr. BC) in the sense of Green and Schroll \cite{GS}, and a finite $f_{ms}$-BG with integral f-degree is exactly a Brauer graph (abbr. BG). Note also that a Brauer tree (abbr. BT) is a connected BG with $n$ vertices and $n-1$ edges such that the f-degrees of all but at most one vertex are equal to $1$.

\begin{Def} \label{fms-BC algebra} {\rm(\cite[Definition 4.1 and Definition 4.4]{LL})}
For an $f_{ms}$-BC $E=(E,P,L,d)$, the fractional Brauer configuration category of type MS (abbr. $f_{ms}$-BCC) associated with $E$ is a $k$-category $\Lambda_{E}=kQ_{E}/I_{E}$, where $Q_{E}$ is a quiver defined as follows: $(Q_{E})_{0}=\{P(e)\mid e\in E\}$, $(Q_{E})_{1}=\{L(e)\mid e\in E\}$ with $s(L(e))=P(e)$ and $t(L(e))=P(g\cdot e)$, and $I_{E}$ is the ideal of path category $kQ_{E}$ generated by the following relations:
\begin{itemize}
\item $L(g^{d(e)-1}\cdot e)\cdots L(g\cdot e)L(e)-L(g^{d(h)-1}\cdot h)\cdots L(g\cdot h)L(h)$, where $P(e)=P(h)$;
\item Paths of the form $L(e_2)L(e_1)$ with $g\cdot e_1\neq e_2$;
\item Paths of the form $L(g^{n-1}\cdot e)\cdots L(g\cdot e)L(e)$ for $n>d(e)$.
\end{itemize}
Moreover, if $E$ is a finite $f_{ms}$-BC, then we define $A_{E}=\bigoplus_{x,y\in (Q_E)_0}\Lambda_{E}(x,y)$ (which is a finite dimensional Frobenius $k$-algebra) and call $A_E$ a fractional Brauer configuration algebra of type MS (abbr. $f_{ms}$-BCA).
\end{Def}

Accordingly, if $E$ is a BC (resp. BG), then we call $A_E$ a BCA (resp. BGA).

Note that in \cite{LL,LL2} we defined $A_{E}$ as the opposite algebra of $kQ_E/I_E$ and where the Nakayama automorphism $\sigma$ of $E$ induces the Nakayama automorphism of the Frobenius algebra $kQ_{E}^{op}/I_{E}^{op}$. However, since it is more direct to connect $E$ with the quiver $Q_E$, we define $A_{E}$ as the algebra $kQ_E/I_E$ in the present paper, and therefore the Nakayama automorphism $\sigma$ of $E$ induces the inverse of the Nakayama automorphism of the Frobenius algebra $A_{E}$.

According to \cite[Proposition 6.6 and Corollary 6.7]{LL}, if $E$ is an $f_{ms}$-BC, then $\Lambda_E$ is a locally bounded special multiserial Frobenius category, and if $E$ is an $f_{ms}$-BG, then $\Lambda_E$ is a locally bounded special biserial Frobenius category. It follows that $f_{ms}$-BCAs and $f_{ms}$-BGAs are proper generalizations of BCAs and BGAs respectively.

For the definitions of morphisms (coverings), walks and fundamental groups (groupoids) of $f_{ms}$-BCs, we refer to \cite[Section 3]{LL2}.

\subsection{Brauer $G$-sets and their fundamental groups}
\

In this subsection, we define Brauer $G$-sets and their fundamental groups.

Let $E=(E,P,L,d)$ be an $f_{ms}$-BG. Since each edge $P(e)$ of $E$ contains $2$ half-edges, we can define an involution $\tau$ on $E$ such that $P(e)=\{e,\tau(e)\}$ for every $e\in E$. Therefore an $f_{ms}$-BG can be considered as a triple $E=(E,\tau,d)$, where $E$ is a $G$-set ($G=\langle g\rangle\cong\mathbb{Z}$), $\tau$ is an involution on $E$ without fixed points, and $d:E\rightarrow \mathbb{Z}_{>0}$ is a function on $E$, such that
\begin{itemize}
\item if $e_1,e_2\in E$ belong to the same $G$-orbit, then $d(e_1)=d(e_2)$;
\item $g^{d(\tau(e))}\cdot \tau(e)=\tau(g^{d(e)}\cdot e)$ for every $e\in E$.
\end{itemize}
Moreover, if $E=(E,P,L,d)$ is an $f_{ms}$-BC such that each polygon of $E$ contains at most $2$ elements, then we may regard $E$ as a quadruple $(E,U,\tau,d)$, where $U=\{e\in E\mid |P(e)|=2\}$ is a subset of $E$, and $\tau$ is an involution on $U$ without fixed points such that $P(e)=\{e,\tau(e)\}$ for every $e\in U$.

The above discussion motivates us to introduce the following notion.

\begin{Def}\label{modified f-BG}
 A Brauer $G$-set is a quadruple $E=(E,U,\tau,d)$, where $E$ is a $G$-set ($G=\langle g\rangle\cong\mathbb{Z}$), $U$ is a subset of $E$, $\tau$ is an involution on $U$ ($\tau$ may have fixed points), and $d:E\rightarrow \mathbb{Z}_{>0}$ is a function on $E$, such that
\begin{itemize}
\item [$(mf1)$] $d(e_1)=d(e_2)$ if $e_1$, $e_2$ belong to the same $G$-orbit;
\item [$(mf2)$] $\sigma(U)=U$ and $\tau\sigma(e)=\sigma\tau(e)$ for every $e\in U$, where $\sigma:E\rightarrow E$, $e\mapsto g^{d(e)}\cdot e$.
\end{itemize}
\end{Def}

The elements of $E$ are called {\it half-edges} of $E$; an element $e\in U$ with $\tau(e)=e$ is called a {\it double half-edge} of $E$; the $G$-orbits of $E$ are called {\it vertices} of $E$; if $e\in U$ and $e\neq \tau(e)$, then the subset $\{e,\tau(e)\}$ of $E$ is called an {\it edge} of $E$; the function $d$ is called the {\it degree function}; $E$ is said to have {\it integral f-degree} (resp. to be {\it f-degree-free}) if each vertex of $E$ is a finite set and the f-degree of each vertex of $E$ is an integer (resp. is equal to $1$); the permutation $\sigma:E\rightarrow E$, $e\mapsto g^{d(e)}\cdot e$ on $E$ is called the {\it Nakayama automorphism} of $E$.

Every Brauer $G$-set $E=(E,U,\tau,d)$ can be visualized as some diagram $\Gamma(E)$, which consists of edges and half-edges (we also view an edge in $\Gamma(E)$ as a union of two half-edges). Each edge (of the form $\{e,\tau(e)\}$) of $E$ corresponds to an edge in $\Gamma(E)$, where $e$ and $\tau(e)$ correspond to the two half-edges of the edge; and each $e\in E$ which is not contained in any edge of $E$ corresponds to a (single) half-edge in $\Gamma(E)$. Each vertex (of the form $G\cdot e$) of $E$ corresponds to an vertex in $\Gamma(E)$, and a half-edge in $\Gamma(E)$ corresponding to $e\in E$ is connected to a vertex in $\Gamma(E)$ corresponding to the $G$-orbit $G\cdot e'$ of $E$ if and only if $e\in G\cdot e'$. For each vertex $v$ of $\Gamma(E)$ corresponding to the $G$-orbit $G\cdot e$ of $E$, the $G$-set structure of $G\cdot e$ gives an order (which is always taken to be clockwise) to half-edges in $\Gamma(E)$ which are connected to $v$.

\begin{Ex1}\label{modified-fms-BG-Ex-1}
Let $E=\{e,e'\}$ be a $G$-set with $g\cdot e=e'$, $g\cdot e'=e$. Let $U=E$, $\tau=id_{U}$ be an involution on $U$, and $d:E\rightarrow \mathbb{Z}_{>0}$ be the function given by $d(e)=d(e')=2$. Then $(E,U,\tau,d)$ is a Brauer $G$-set, which is given by the diagram
$$\begin{tikzpicture}
\fill (0,0) circle (0.5ex);
\draw    (0,0)--(1,0);
\draw    (0,0)--(-1,0);
\node at(0.6,0.2) {$e'$};
\node at(-0.6,0.15) {$e$};
\node at(1.2,-0.1) {.};
\end{tikzpicture}$$
\end{Ex1}

\begin{Ex1}\label{modified-fms-BG-Ex-2}
Let $E=\{e,e',e_1,e_2\}$ be a $G$-set with $g\cdot e=e_1$, $g\cdot e_1=e'$, $g\cdot e'=e_2$, $g\cdot e_2=e$. Let $U=\{e,e',e_1\}$, $\tau$ be the involution on $U$ given by $\tau(e)=e'$, $\tau(e')=e$, $\tau(e_1)=e_1$, and $d:E\rightarrow \mathbb{Z}_{>0}$ be the function given by $d(e)=d(e')=d(e_1)=d(e_2)=4$. Then $(E,U,\tau,d)$ is a Brauer $G$-set, which is given by the diagram
$$
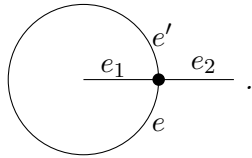
\begin{tikzpicture}
\fill (0,0) circle (0.5ex);
\draw    (0,0)--(1,0);
\draw    (0,0)--(-1,0);
\draw (-1,0) circle (1);
\node at(0.6,0.2) {$e_2$};
\node at(-0.6,0.2) {$e_1$};
\node at(0,-0.6) {$e$};
\node at(0.05,0.6) {$e'$};
\node at(1.2,-0.1) {.};
\end{tikzpicture}$$
\end{Ex1}

\begin{Rem1} \label{view-f_ms-BC-as-Brauer-G-set}
According to the remarks before Definition \ref{modified f-BG}, an $f_{ms}$-BG is identified with a Brauer $G$-set $E=(E,U,\tau,d)$ such that $U=E$ and $\tau$ has no fixed points, and an $f_{ms}$-BC such that each polygon of it contains at most two elements is identified with a Brauer $G$-set $E=(E,U,\tau,d)$ such that $\tau$ has no fixed points. We shall frequently using these identifications.
\end{Rem1}

\begin{Def}\label{modified BG}
A Brauer $G$-set $E=(E,U,\tau,d)$ of integral f-degree (or equivalenly, the Nakayama automorphism $\sigma$ of $E$ is an identity) with $U=E$ is said to be a modified Brauer graph (abbr. modified BG).
\end{Def}

In the sense of Remark \ref{view-f_ms-BC-as-Brauer-G-set}, an $f_{ms}$-BC such that each polygon of it contains at most two elements is a modified BG if and only if it is a BG. Note that the Brauer $G$-set in Example \ref{modified-fms-BG-Ex-1} is a modified BG but the Brauer $G$-set in Example \ref{modified-fms-BG-Ex-2} is not a modified BG.

\medskip
Let $E=(E,U,\tau,d)$ be a Brauer $G$-set. A {\it walk} of $E$ is a sequence of the form $$w=e_{n}\frac{\delta_{n}}{}e_{n-1}\frac{\delta_{n-1}}{}\cdots\frac{\delta_{3}}{}e_{2}\frac{\delta_{2}}{}e_{1}\frac{\delta_{1}}{}e_{0},$$
where $e_i\in E$ for $0\leq i\leq n$ and $\delta_j\in\{g,g^{-1},\tau\}$ for $1\leq j\leq n$, such that $e_{i-1},e_i\in U$ if $\delta_i=\tau$ and
\begin{equation*}
e_i=\begin{cases}
g\cdot e_{i-1}, \text{ if } \delta_i=g; \\
g^{-1}\cdot e_{i-1}, \text{ if } \delta_i=g^{-1}; \\
\tau(e_{i-1}), \text{ if } \delta_i=\tau.
\end{cases}
\end{equation*}
We may write $w=(e_{n}|\delta_{n}\cdots\delta_{1}|e_{0})$ or $w=\delta_{n}\cdots\delta_{1}$ if there is no confusion. For such a walk $w$ as above, define $s(w)=e_{0}$ and $t(w)=e_{n}$, where $s(w)$ and $t(w)$ are the source and the terminal of $w$, respectively. A Brauer $G$-set $E$ is said to be {\it connected} if every two half-edges $e_1$, $e_2$ of it can be connected by a walk.

\begin{Rem1}
For an $f_{ms}$-BC $E$, we have already defined the walks and the special walks of $E$ in \cite[Section 3]{LL2}. Suppose that $E$ is $f_{ms}$-BC such that each polygon of $E$ contains at most $2$ angles, by Remark \ref{view-f_ms-BC-as-Brauer-G-set}, we can also view $E$ as a Brauer $G$-set. Then every walk of the Brauer $G$-set $E$ is a walk of the $f_{ms}$-BC $E$, and every special walk of the $f_{ms}$-BC $E$ is a walk of $E$ as a Brauer $G$-set since a special walk contains no subwalk of the form $(e| \tau | e)$. In particular, $E$ is connected as an f-BC if and only if $E$ is connected as a Brauer $G$-set.
\end{Rem1}

\begin{Def}\label{homotopy of modified walks}
Let $E=(E,U,\tau,d)$ be a Brauer $G$-set. Define the homotopy relation $\approx$ on the set of walks of $E$ as the equivalence relation generated by
\begin{itemize}
\item [$(mh1)$] $(e|g^{-1}g|e)\approx(e|gg^{-1}|e)\approx(e|\tau^{2}|e)\approx(e||e)$ for every $e\in E$.
\item [$(mh2)$] $(g^{d(\tau(e))}\cdot \tau(e)|\tau|g^{d(e)}\cdot e)(g^{d(e)}\cdot e|g^{d(e)}|e)\approx (g^{d(\tau(e))}\cdot \tau(e)|g^{d(\tau(e))}|\tau(e))(\tau(e)|\tau|e)$ for every $e\in U$.
\item [$(mh3)$] If $w_{1}\approx w_{2}$, then $uw_{1}\approx uw_{2}$ and $w_{1}v\approx w_{2}v$ whenever the compositions make sense.
\end{itemize}
\end{Def}

For a walk $w$ of a Brauer $G$-set $E$, denote by $[w]$ the homotopy class of $w$.

In \cite[Definition 3.9]{LL2} we have already defined the homotopy relation $\sim$ on the set of walks of an f-BC $E$. Note that if $E$ is an $f_{ms}$-BC $E$ such that each polygon of $E$ contains at most $2$ angles, then for walks $w_1$, $w_2$ of $E$ as a Brauer $G$-set, $w_1\approx w_2$ implies $w_1\sim w_2$ (in fact, we have $w_1\approx w_2$ if and only if $w_1\sim w_2$, see the proof of Lemma \ref{two definitions of fundamental group are equal}).

Recall that in \cite[Definition 3.11]{LL2} we have defined the fundamental group (groupoid) of an f-BC $E$, using the homotopy relation $\sim$. Similarly, we can define the {\it fundamental group} $\Pi_m(E,e)$ (resp. {\it fundamental groupoid} $\Pi_m(E,A)$) of a Brauer $G$-set $E$ at $e\in E$ (resp. on a subset $A\subseteq E$), using the homotopy relation $\approx$. For the definition of the fundamental groupoid of a quiver, we refer to \cite[Section 4]{LL2}.

\begin{Ex1}
Let $E=(E,U,\tau,d)$ be the Brauer $G$-set in Example \ref{modified-fms-BG-Ex-1}, and let $A=\{e,e'\}=E$. Then $\Pi_m(E,A)$ is isomorphic to $\mathscr{F}/\langle c^2=1_x,d^2=1_y,bac=cba,abd=dab\rangle$, where $\mathscr{F}$ is the fundamental groupoid of the quiver
$$
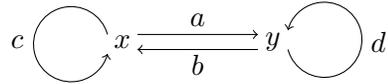
\begin{tikzpicture}
\draw[->] (0.2,0.1) -- (1.8,0.1);
\draw[->] (1.8,-0.1) -- (0.2,-0.1);
\draw[->] (-0.2,0.1) arc (15:345:0.5);
\draw[->] (2.2,-0.1) arc (-165:165:0.5);
\node at(0,0) {$x$};
\node at(2,0) {$y$};
\node at(-1.4,0) {$c$};
\node at(3.4,0) {$d$};
\node at(1,0.3) {$a$};
\node at(1,-0.3) {$b$};
\end{tikzpicture}.$$
\end{Ex1}

\begin{Ex1}
Let $E=(E,U,\tau,d)$ be the Brauer $G$-set in Example \ref{modified-fms-BG-Ex-2}. Then by Lemma \ref{a calculation of modified fundamental group} and Lemma \ref{isomorphism of modified fundamental groupoids} below, $\Pi_m(E,e)\cong F\langle x,y,z\rangle/\langle xy=yx,xz=zx, z^{2}=1\rangle$, where $F\langle x,y,z\rangle$ denotes the free group on the set $\{x,y,z\}$.
\end{Ex1}

If $E$ is an $f_{ms}$-BC such that each polygon of $E$ contains at most $2$ angles, then by Remark \ref{view-f_ms-BC-as-Brauer-G-set} we can view it as a Brauer $G$-set. The following lemma shows that the fundamental group of $E$ as an $f_{ms}$-BC is isomorphic to the fundamental group of $E$ as a Brauer $G$-set.

\begin{Lem}\label{two definitions of fundamental group are equal}
Let $E=(E,U,\tau,d)$ be a Brauer $G$-set such that $\tau$ has no fixed points (equivalently, $E$ is an $f_{ms}$-BC such that each polygon of $E$ contains at most $2$ angles), and let $A$ be a subset of $E$. Denote by $\Pi_m(E,A)$ and $\Pi(E,A)$ the fundamental groupoid of $E$ on $A$ as a Brauer $G$-set and as an $f_{ms}$-BC respectively. Then $\Pi_m(E,A)$ is isomorphic to $\Pi(E,A)$.
\end{Lem}

\begin{proof}
In the following, when we say a walk of $E$, we always view $E$ as a Brauer $G$-set unless otherwise stated. For each walk $w$ of $E$ as an $f_{ms}$-BC, we denote by $\overline{w}$ the homotopy class of $w$ under the homotopy relation $\sim$. First we need to show the following fact: If $w_1,w_2$ are two walks of $E$ with $w_1\sim w_2$, then $w_1\approx w_2$. Using the same idea as in \cite[Proposition 3.46]{LL2} we can show that for each walk $w$ of $E$, there exist some special walk $v$ and some integer $n$ such that $w\approx(t(w)|g^{n d(t(w))}|t(v))v$. For two walks $w_1,w_2$ of $E$ with $w_1\sim w_2$, let $w_i\approx(t(w_i)|g^{n_i d(t(w_i))}|t(v_i))v_i$ ($i=1,2$), where $v_i$ is a special walk of $E$ and $n_i$ is an integer. Then $(t(w_1)|g^{n_1 d(t(w_1))}|t(v_1))v_1\sim (t(w_2)|g^{n_2 d(t(w_2))}|t(v_2))v_2$, and by \cite[Proposition 3.46]{LL2} we have $(v_1,n_1)=(v_2,n_2)$. Then $w_1\approx (t(w_1)|g^{n_1 d(t(w_1))}|t(v_1))v_1=(t(w_2)|g^{n_2 d(t(w_2))}|t(v_2))v_2\approx w_2$ since $t(w_1)=t(w_2)$.

For each walk $w$ of $E$ as an $f_{ms}$-BC, choose a walk $w'$ of $E$ such that $w\sim w'$ (by \cite[Proposition 3.46]{LL2} this can be done). For every $a$, $b\in A$, define a map $F:\Pi(E,A)(a,b)\rightarrow\Pi_{m}(E,A)(a,b)$, $\overline{w}\mapsto [w']$. If $w_1$, $w_2$ are two walks of $E$ (as an $f_{ms}$-BC) with $w_1\sim w_2$, then $w'_1\sim w'_2$ and $w'_1\approx w'_2$. So $F$ is well-defined. The same argument shows that $F$ does not depend on the choice of the walk $w'$ for each walk $w$ of $E$ as an $f_{ms}$-BC.

For $\overline{w_1}\in\Pi(E,A)(a,b)$ and $\overline{w_2}\in\Pi(E,A)(b,c)$, since $w'_2 w'_1$ is a walk of $E$ such that $w_2 w_1\sim w'_2 w'_1$, $F(\overline{w_2 w_1})=[w'_2 w'_1]=F(\overline{w_2})F(\overline{w_1})$. Therefore $F$ becomes a functor from $\Pi(E,A)$ to $\Pi_{m}(E,A)$. If $F(\overline{w_1})=F(\overline{w_2})$, then $w'_1\approx w'_2$. Therefore $w_1\sim w_2$ and $\overline{w_1}=\overline{w_2}$, so $F$ is faithful. For every morphism $[w]\in\Pi_{m}(E,A)(a,b)$, since $w$ and $w'$ are walks with $w\sim w'$, we have $w\approx w'$. Therefore $F(\overline{w})=[w']=[w]$ and $F$ is dense. Since $F$ induces identity map on the set of objects, $F$ is an isomorphism.
\end{proof}

\subsection{Covering theory for Brauer $G$-sets}
\

In this subsection, we define morphisms and coverings between Brauer $G$-sets, which are analogies of the notions of morphisms and coverings between f-BCs in \cite{LL2}; we compare different coverings between Brauer $G$-sets using their fundamental groups; we define special walks on Brauer $G$-sets and construct the universal cover of Brauer $G$-sets.

\begin{Def}\label{covering of modified f-BG}
Let $E=(E,U,\tau,d)$ and $E'=(E',U',\tau',d')$ be Brauer $G$-sets. A morphism (resp. covering) $f:E\rightarrow E'$ of Brauer $G$-sets is a morphism of $G$-sets ($G=\langle g\rangle$) satisfying the following conditions $(1)$, $(2)$ and $(3)$ (resp. $(1')$, $(2)$ and $(3)$):
\begin{itemize}
\item [$(1)$] For every $e\in U$, we have $f(e)\in U'$;
\item [$(1')$] For every $e\in E$, $e\in U$ if and only if $f(e)\in U'$;
\item [$(2)$] $f(\tau(e))=\tau'(f(e))$ for every $e\in U$;
\item [$(3)$] $d'(f(e))=d(e)$ for every $e\in E$.
\end{itemize}
\end{Def}

Note that a covering of f-BCs always induces bijections on polygons, but here a covering of Brauer $G$-sets $f$ may maps both of the two half-edges $e,\tau(e)$ in an edge to a double half-edge $\tau'(f(e))=f(e)$. This indicates that the covering of Brauer $G$-sets is more flexible than the covering of f-BCs.

\begin{Rem1} \label{lifting-walks}
\begin{itemize}
\item [(1)] Let $E$ and $E'$ as above. If $E$ and $E'$ are $f_{ms}$-BCs such that each polygon contains at most $2$ angles (that is, $\tau$ and $\tau'$ have no fixed points), then for any map $f:E\rightarrow E'$, $f$ is a covering of Brauer $G$-sets if and only if $f$ is a covering of f-BCs.

\item [(2)] A morphism $f:E\rightarrow E'$ of Brauer $G$-sets maps each walk of $E$ to a  walk of $E'$. Moreover, if $f:E\rightarrow E'$ is a covering of Brauer $G$-sets, then for every walk $w'=(h'|\delta_n\cdots\delta_1|e')$ of $E'$ and for every $e\in E$ (resp. $h\in E$) which lies over $e'$ (resp. $h'$), there exists a unique walk $w$ of $E$ which lies over $w'$ whose source (resp. terminal) is $e$ (resp. $h$).
\end{itemize}
\end{Rem1}

The covering theory for Brauer $G$-sets is similar to that of f-BCs. Here we only state several propositions that we need without proof.

The following proposition is an analogy to \cite[Proposition 3.16]{LL2}.

\begin{Prop}\label{modified homotopy lifting}
Let $f:E\rightarrow E'$ be a covering of Brauer $G$-sets, $u$, $v$ be two walks of $E$ with $s(u)=s(v)$ or $t(u)=t(v)$. Then $u\approx v$ if and only if $f(u)\approx f(v)$.
\end{Prop}

Therefore each covering $f:E\rightarrow E'$ of Brauer $G$-sets induces an injective map $f_{*}:\Pi_{m}(E,e)\rightarrow \Pi_{m}(E',f(e))$ of associated fundamental groups for each $e\in E$.

\begin{Rem1} \label{general-method-to-compute-the-fundamental-group}
Proposition \ref{modified homotopy lifting} suggests a general method to calculate the fundamental group of a given Brauer $G$-set $E$. Let $f:E\rightarrow E'$ be a covering of Brauer $G$-sets and $e'\in E'$. Then $f^{-1}(e')$ becomes a $\Pi_m(E',e')$-set: for any $e\in f^{-1}(e')$ and $[w']\in\Pi_m(E',e')$, $[w']\cdot e$ is defined as the terminal of $w$, where $w$ is the walk of $E$ lying over $w'$ with $s(w)=e$ (Proposition \ref{modified homotopy lifting} ensures that this group action is well-defined). Then for each $e\in f^{-1}(e')$, the stabilizer subgroup of $e$ in $\Pi_m(E',e')$ is equal to the image of $f_{*}:\Pi_m(E,e)\rightarrow \Pi_m(E',e')$, which is isomorphic to $\Pi_m(E,e)$.
\end{Rem1}

By analogy with \cite[Proposition 3.24]{LL2}, we have

\begin{Prop}\label{modified existence of morphism}
Let $E$, $E_{1}$, $E_{2}$ be Brauer $G$-sets with $E_{1}$ connected, and let $f_{1}:E_{1}\rightarrow E$ and $f_{2}:E_{2}\rightarrow E$ be coverings of Brauer $G$-sets. For $e_i\in E_{i}$ $(i=1,2)$ with $f_{1}(e_1)=f_{2}(e_2)$, there exists a covering $\phi:E_{1}\rightarrow E_{2}$ of Brauer $G$-sets such that $f_{1}=f_{2}\phi$ and $\phi(e_1)=e_2$ if and only if $f_{1*}(\Pi_{m}(E_1,e_1))\subseteq f_{2*}(\Pi_{m}(E_2,e_2))$. Moreover, if such $\phi$ exists, then it is unique.
\end{Prop}

Similar to f-BCs, we have the notion of regular covering between Brauer $G$-sets.

\begin{Def}\label{regular covering}
Let $E,E'$ be connected Brauer $G$-sets. A covering $f:E\rightarrow E'$ is said to be regular if $f_{*}(\Pi(E,e))$ is a normal subgroup of $\Pi(E',f(e))$ for some $e\in E$ (or equivalently, $f_{*}(\Pi(E,e))$ is a normal subgroup of $\Pi(E',f(e))$ for each $e\in E$).
\end{Def}

In the following, we introduce the quotient of a Brauer $G$-set (by a group of automorphisms).

Let $E=(E,U,\tau,d)$ be a Brauer $G$-set and $\Pi$ be a group of automorphisms of $E$. Similar to the case of f-BC, we may define a quadruple $E/\Pi=(E/\Pi,U',\tau',d')$ as follows: $E/\Pi$ is the $G$-set of $\Pi$-orbit of $E$; $U':=\{[e]\in E/\Pi\mid e\in U\}$ is a subset of $E/\Pi$; $\tau'$ is an involution of $U'$ given by $\tau'([e])=[\tau(e)]$ for every $[e]\in U'$; $d':E/\Pi\rightarrow \mathbb{Z}_{>0}$ is a function on $E/\Pi$ given by $d'([e])=d(e)$ (we denote $[e]$ the $\Pi$-orbit of $e\in E$).

The following Lemma is an analogy of \cite[Lemma 3.35]{LL2} and \cite[Lemma 3.36]{LL2}. Note that here the condition that $\Pi$ acts admissibly on $E$ is not needed.

\begin{Lem}\label{modified projection is a covering}
$E/\Pi=(E/\Pi,U',\tau',d')$ is a Brauer $G$-set, and the natural projection $p:E\rightarrow E/\Pi$ is a regular covering of Brauer $G$-sets.
\end{Lem}

Moreover, by imitating the proof of \cite[Theorem 3.37]{LL2}, we see that each regular covering $f:E\rightarrow E'$ can be decomposed into a natural projection $E\rightarrow E/\mathrm{Aut}(f)$ and an isomorphism $E/\mathrm{Aut}(f)\xrightarrow{\sim} E'$, where $\mathrm{Aut}(f)$ is the group which consists of automorphisms $\phi$ of $E$ with $f\phi=f$. If $\Omega$ is a group which is isomorphic to $\mathrm{Aut}(f)$, then we call $f$ a regular covering with group $\Omega$.

By analogy with \cite[Subsection 3.4]{LL2}, we construct the universal cover of a Brauer $G$-set.

Let $E=(E,U,\tau,d)$ be a Brauer $G$-set such that $\tau$ has a fixed point. Define a Brauer $G$-set $\widehat{E}=(\widehat{E},\widehat{U},\widehat{\tau},\widehat{d})$ as follows: $\widehat{E}=E_1\sqcup E_2$ as a $G$-set with $E_1=E_2=E$, and denote the element $e\in E$ by $e_i$ if we consider it as an element of $E_i$. Define $\widehat{U}=U_1\sqcup U_2$ as a subset of $\widehat{E}$, where $U_i$ denotes the subset $U$ of $E_i$ for $i=1,2$, and for each $e\in U$ and each $i=1$, $2$, define
\begin{equation*}
\widehat{\tau}(e_i)= \begin{cases}
\tau(e)_i, \text{ if } \tau(e)\neq e; \\
e_{3-i}, \text{ if } \tau(e)=e.
\end{cases}
\end{equation*}
Define $\widehat{d}(e_i)=d(e)$. Note that $\widehat{\tau}$ has no fixed points, that is, $\widehat{E}$ is an $f_{ms}$-BC such that each polygon of it contains at most $2$ angles. For example, the Brauer $G$-set $\widehat{E}$ for $E$ in Example \ref{modified-fms-BG-Ex-1} can be viewed by the diagram
$$\begin{tikzpicture}
\draw (0,0) circle (0.7);
\fill (-0.7,0) circle (0.5ex);
\fill (0.7,0) circle (0.5ex);
\node at(0.75,0.5) {$e_2$};
\node at(-0.75,0.5) {$e_1$};
\node at(0.75,-0.5) {$e_2'$};
\node at(-0.75,-0.5) {$e_1'$};
\end{tikzpicture},$$
and the Brauer $G$-set $\widehat{E}$ for $E$ in Example \ref{modified-fms-BG-Ex-2} can be viewed by the diagram
$$\begin{tikzpicture}
\fill (0,0) circle (0.5ex);
\draw    (0,0)--(1,0);
\draw    (0,0)--(-1,0);
\draw (-1,0) circle (1);
\node at(0.6,0.2) {$e_{11}$};
\node at(-0.6,0.2) {$e_{21}$};
\node at(0,-0.6) {$e_1'$};
\node at(0.05,0.6) {$e_1$};
\fill (2.0,0) circle (0.5ex);
\draw    (3.0,0)--(1,0);
\draw (3,0) circle (1);
\node at(2.6,0.2) {$e_{22}$};
\node at(1.4,0.2) {$e_{12}$};
\node at(1.9,-0.6) {$e_2$};
\node at(1.9,0.6) {$e_2'$};
\end{tikzpicture}.$$

Let $\phi:\widehat{E}\rightarrow\widehat{E}$ be the map which sends $e_i$ to $e_{3-i}$ for each $e\in E$ and each $i=1$, $2$. It is straightforward to show that $\phi$ is a morphism of Brauer $G$-sets. Since $\phi^2=id$, $\phi$ is an automorphism of $\widehat{E}$. Moreover, $\widehat{E}/\langle\phi\rangle\cong E$ as Brauer $G$-sets. By Lemma \ref{modified projection is a covering}, we have

\begin{Lem}\label{pi is a covering of modified f-BGs}
The map $\pi:\widehat{E}\rightarrow E$, $e_i\mapsto e$ is a regular covering of Brauer $G$-sets with group $\mathbb{Z}/2\mathbb{Z}$.
\end{Lem}

\begin{Def}\label{special-walk-Brauer-G-set}
Let $E=(E,U,\tau,d)$ be a Brauer $G$-set, a walk $w$ of $E$ is called special if it is of the form \begin{multline*} (g^{i_k}\cdot e_k|g^{i_k}|e_k)(e_k|\tau|g^{i_{k-1}}\cdot e_{k-1})(g^{i_{k-1}}\cdot e_{k-1}|g^{i_{k-1}}|e_{k-1})(e_{k-1}|\tau|g^{i_{k-2}}\cdot e_{k-2})\cdots \\
(e_2|\tau|g^{i_{1}}\cdot e_{1})(g^{i_1}\cdot e_1|g^{i_1}|e_1)(e_1|\tau|g^{i_{0}}\cdot e_{0})(g^{i_0}\cdot e_0|g^{i_0}|e_0),  \end{multline*}
where $0\leq i_0<d(e_0)$, $0\leq i_k<d(e_k)$, and $0< i_l<d(e_l)$ for all $1\leq l\leq k-1$.
\end{Def}

\begin{Rem1}
\begin{enumerate}
\item[$(1)$] If $k=0$ and $i_0=0$ in the definition above, then $w=(e_0||e_0)$ is the walk of length $0$ at $e_0$.
\item[$(2)$] If $\tau$ has no fixed points (that is, $E$ is an $f_{ms}$-BC such that each polygon of it contains at most $2$ angles), then a special walk of $E$ is just a special walk of $E$ as an $f_{ms}$-BC.
\end{enumerate}
\end{Rem1}

For a Brauer $G$-set $E=(E,U,\tau,d)$ and for $e\in E$, similar to the case of $f_{ms}$-BC, we can define a connected Brauer $G$-set $B_{(E,e)}=(B_{(E,e)},U_e,\tau_e,d_e)$ such that $\tau_e$ has no fixed points (that is, $B_{(E,e)}$ is an $f_{ms}$-BC such that each polygon of $B_{(E,e)}$ contains at most two angles) as follows: $B_{(E,e)} =\{$special walks of $E$ starting at $e\}$. The action of $G=\langle g\rangle$ on $B_{(E,e)}$ is given by
\begin{equation*}
g\cdot w= \begin{cases}
g^{i_{k}+1}\tau g^{i_{k-1}}\tau\cdots\tau g^{i_1}\tau g^{i_0}, & \text{if } w=g^{i_k}\tau g^{i_{k-1}}\tau\cdots\tau g^{i_1}\tau g^{i_0} \text{ with } i_k <d(t(w))-1; \\
\tau g^{i_{k-1}}\tau\cdots\tau g^{i_1}\tau g^{i_0}, & \text{if } w=g^{i_k}\tau g^{i_{k-1}}\tau\cdots\tau g^{i_1}\tau g^{i_0} \text{ with } i_k =d(t(w))-1.
\end{cases}
\end{equation*}
The subset $U_e$ of $B_{(E,e)}$ is given by $U_e=\{w\in B_{(E,e)}\mid t(w)\in U\}$, the involution $\tau_e$ of $U_e$ is given by $\tau_e(w)=(\tau(e)|\tau|e)$ for $w=(e||e)$ and
\begin{equation*}
\tau_e(w)=\begin{cases}
(\tau(t(w))|\tau|t(w))w, \text{ if } w=g^{i_k}\tau g^{i_{k-1}}\tau\cdots\tau g^{i_1}\tau g^{i_0} \text{ with } i_k >0; \\
g^{i_{k-1}}\tau\cdots\tau g^{i_1}\tau g^{i_0}, \text{ if } w=g^{i_k}\tau g^{i_{k-1}}\tau\cdots\tau g^{i_1}\tau g^{i_0} \text{ with } i_k =0
\end{cases}
\end{equation*}
for $w\neq (e||e)$;
and the degree function $d_e$ is given by $d_e(w)=d(t(w))$. Note that $B_{(E,e)}$ is f-degree-free.

The Brauer $G$-set $\mathbb{Z}B_{(E,e)}=\{(w,n)\mid w\in B_{(E,e)}$ and $n\in\mathbb{Z}\}$ can also be defined similarly as in \cite[Subsection 3.4]{LL2}, whose involution also has no fixed points. Note that when the involution $\tau$ of $E$ has no fixed points (that is, $E$ is an $f_{ms}$-BC such that each polygon of $E$ contains at most two angles), the $f_{ms}$-BC $B_{(E,e)}$ (resp. $\mathbb{Z}B_{(E,e)}$) defined as above is just the $f_{ms}$-BC $B_{(E,e)}$ (resp. $\mathbb{Z}B_{(E,e)}$) defined in \cite[Subsection 3.4]{LL2}.

\begin{Prop}\label{universal-cover-modified-fms-BG}
Let $\sigma$ be the Nakayama automorphism of a Brauer $G$-set $E$. Then there exists a covering of Brauer $G$-sets $q:\mathbb{Z}B_{(E,e)}\rightarrow E$, $(w,n)\mapsto \sigma^{n}(t(w))$, which is universal in the sense of \cite[Corollary 3.26]{LL2}.
\end{Prop}

\begin{proof}
It is straightforward to show that $q$ is a covering of Brauer $G$-sets. By Proposition \ref{modified existence of morphism}, it suffices to show that $\Pi_{m}(\mathbb{Z}B_{(E,e)})=\{1\}$. When the involution $\tau$ of $E$ has no fixed points, it follows from Lemma \ref{two definitions of fundamental group are equal} and \cite[Proposition 3.44]{LL2} that $\Pi_{m}(\mathbb{Z}B_{(E,e)})=\{1\}$.

When the involution $\tau$ of $E$ has a fixed point, let $\pi:\widehat{E}\rightarrow E$ be the covering of Brauer $G$-sets in Lemma \ref{pi is a covering of modified f-BGs}. Choose some $x\in\widehat{E}$ with $\pi(x)=e$. Since $\pi$ is a covering of Brauer $G$-sets, by Remark \ref{lifting-walks} (2), $\pi$ induces a bijection between the set of special walks of $\widehat{E}$ starting at $x$ and the set of special walks of $E$ starting at $e$. Therefore $\pi$ induces an isomorphism between $B_{(\widehat{E},x)}$ and $B_{(E,e)}$. Since $\widehat{E}$ is a Brauer $G$-set whose involution has no fixed points, again by Lemma \ref{two definitions of fundamental group are equal} and \cite[Proposition 3.44]{LL2}, we have $\Pi_{m}(\mathbb{Z}B_{(\widehat{E},x)})=\{1\}$. Therefore $\Pi_{m}(\mathbb{Z}B_{(E,e)})=\{1\}$.
\end{proof}

In general $B_{(E,e)}$ is a finite or infinite tree (up to deleting some half-edges) for any Brauer $G$-set $E$, here is an example:

\begin{Ex1}
Let $E$ be the Brauer $G$-set in Example \ref{modified-fms-BG-Ex-1}. Then $B_{(E,e)}$ is an infinite tree with free f-degree, which is given by the diagram
$$\begin{tikzpicture}
\fill (0,0) circle (0.5ex);
\draw (0,0)--(1,0);
\fill (1,0) circle (0.5ex);
\draw (1,0)--(2,0);
\fill (2,0) circle (0.5ex);
\draw (2,0)--(3,0);
\fill (3,0) circle (0.5ex);
\node at(-0.4,0) {$\cdot$};
\node at(-0.6,0) {$\cdot$};
\node at(-0.8,0) {$\cdot$};
\node at(3.4,0) {$\cdot$};
\node at(3.6,0) {$\cdot$};
\node at(3.8,0) {$\cdot$};
\end{tikzpicture}$$
\end{Ex1}

The following proposition is an analogy of \cite[Proposition 3.46]{LL2}.

\begin{Prop}\label{modified unique factorization}
Let $E=(E,U,\tau,d)$ be a Brauer $G$-set, $w$ be a walk of $E$. Then there exist a unique special walk $v$ and a unique integer $n$ such that $w\approx (t(w)|g^{n d(t(w))}|t(v))v$.
\end{Prop}

\begin{proof}
It is straightforward to show that each walk of $E$ is homotopic to a walk of this form. The rest of the proof is similar to that of \cite[Proposition 3.46]{LL2}, using the covering $q:\mathbb{Z}B_{(E,e)}\rightarrow E$, $(u,n)\mapsto \sigma^{n}(t(u))$ of Brauer $G$-sets in Proposition \ref{universal-cover-modified-fms-BG}, where $e=s(w)$.
\end{proof}

\subsection{Lines and bands}
\

In this subsection, we define lines and bands on Brauer $G$-sets, and compare the numbers of equivalence classes of bands via a covering between Brauer $G$-sets. Moreover, we define the reduced form $R_E$ (which is a BG) of a finite connected $f_{ms}$-BG $E$, and show that the $f_{ms}$-BGA $A_E$ is representation-finite (resp. domestic) if and only if the BGA $A_{R_E}$ is representation-finite (resp. domestic).

Before defining lines and bands on a Brauer $G$-set, we first recall the bands and the equivalence classes of bands over a string algebra (cf. \cite[II.2]{E}).

\begin{Def}  \label{special-biserial-algebra} A finite dimensional $k$-algebra $A$ is called special biserial if there is a
	quiver $Q$ and an admissible ideal $I$ in $kQ$ such that $A$ is Morita equivalent to $kQ/I$ and
	such that $kQ/I$ satisfies the following conditions:\\
	$(1)$ At every vertex $v$ in $Q$ there are at most two arrows starting at $v$ and there are at most
	two arrows ending at $v$.\\
	$(2)$ For every arrow $\alpha$ in $Q$, there exists at most one arrow $\beta$ such that $\beta\alpha\notin I$ and there
	exists at most one arrow $\gamma$ such that $\alpha\gamma\notin I$.
	
	A special biserial algebra $A$ is called a string algebra if the defining ideal $I$ is generated by paths.
\end{Def}

Let $A=kQ/I$ be a special biserial algebra with $(Q,I)$ satisfying conditions $(1), (2)$ in Definition \ref{special-biserial-algebra}. Let $S$ be the sum of socles of all indecomposable projective-injective $A$-modules which are not uniserial. Then $S$ is an ideal of $A$ and $\overline{A}=A/S$ is a string algebra, which has the same representation type as $A$. Therefore to study the representation type of $A$, we may assume that $A$ is a string algebra.

A band in a string algebra $A=kQ/I$ is a closed word $b=\alpha_r\cdots\alpha_2\alpha_1$ in $Q$ of positive length with $\alpha_i\in Q_1$ or $\alpha_{i}^{-1}\in Q_1$ for each $i\in\mathbb{Z}/r\mathbb{Z}=\{1,2,\cdots,r\}$, such that
\begin{enumerate}
\item $\alpha_{i+1}\neq \alpha_{i}^{-1}$ for all $i\in\mathbb{Z}/r\mathbb{Z}=\{1,2,\cdots,r\}$;
\item there is no subpath of $b^n$ or $b^{-n}$ which belongs to $I$ for each positive integer $n$;
\item there is no subword $w$ of $b$ such that $b=w^k$ for some $k>1$.
\end{enumerate}
We can define an equivalence relation on the set of bands in $A$ which identifies each band with its rotations and their inverses.

Now let $E=(E,P,L,d)$ be a finite $f_{ms}$-BG such that each edge of $E$ contains a half-edge $e$ with $d(e)>1$. According to \cite[Section 6]{LL}, $A_E\cong kQ'_E/I'_E$ with $I'_E$ admissible, where $Q_E$ is the subquiver of $Q_E$ given by $(Q'_E)_0=(Q_E)_0$ and $(Q'_E)_1=\{L(e)\mid e\in E $ with $d(e)>1\}$, and $I'_E$ is generated by the following three types of relations:
\begin{itemize}
\item[$(fR1')$] $L(g^{d(e)-1}\cdot e)\cdots L(g\cdot e)L(e)-L(g^{d(h)-1}\cdot h)\cdots L(g\cdot h)L(h)$, where $e,h\in E$ and $d(e),d(h)>1$;
\item[$(fR2')$] $L(e_1)L(e_2)$, where $e_1,e_2\in E$, $d(e_1),d(e_2)>1$, and $e_1\neq g\cdot e_2$;
\item[$(fR3')$] $L(g^{d(e)}\cdot e)\cdots L(g\cdot e)L(e)$, where $e\in E$ and $d(e)>1$.
\end{itemize}
According to \cite[Proposition 6.6]{LL}, $A_E$ is a self-injective special biserial algebra and therefore $A_E/\mathrm{soc}(A_E)$ is a string algebra, which is given by the quiver $Q'_E$ and the admissible ideal $I''_E$ of $kQ'_E$ generated by the following two types of relations
\begin{itemize}
\item[$(a)$] $L(g^{d(e)-1}\cdot e)\cdots L(g\cdot e)L(e)$, where $e\in E$ and $d(e)>1$;
\item[$(b)$] $L(e_1)L(e_2)$, where $e_1,e_2\in E$, $d(e),d(h)>1$, and $e_1\neq g\cdot e_2$.
\end{itemize}
We call $b$ a band of $A_E$ if $b$ is a band of the string algebra $A_E/\mathrm{soc}(A_E)\cong kQ'_E/I''_E$.

\begin{Def}\label{line}
Let $E=(E,U,\tau,d)$ be a Brauer $G$-set. A line of $E$ is an infinite sequence \begin{equation*}l=\cdots\frac{\delta_{3}}{}e_{2}\frac{\delta_{2}}{}e_{1}\frac{\delta_{1}}{}e_{0}
\frac{\delta_{0}}{}e_{-1}\frac{\delta_{-1}}{}e_{-2}\frac{\delta_{-2}}{}\cdots,\end{equation*}
where $e_i\in E$ and $\delta_i\in\{g,g^{-1},\tau\}$ for every $i\in\mathbb{Z}$, such that
\begin{itemize}
\item [$(a)$] $e_{i-1},e_{i}\in U$ if $\delta_{i}=\tau$;
\item [$(b)$] $e_{i+1}=\delta_{i+1}(e_i)$ for every $i\in\mathbb{Z}$.
\end{itemize}
\end{Def}

We will also write a line $l$ as a family $\{(e_i,\delta_i)\}_{i\in\mathbb{Z}}$. For such a line $l$ and for any integer $n$, denote by $l[n]$ the line $\{(e'_i,\delta'_i)\}_{i\in\mathbb{Z}}$, where $e'_i=e_{i+n}$ and $\delta'_i=\delta_{i+n}$. We call $l[n]$ the $n$-th translate of $l$. We may consider a line $l=\{(e_i,\delta_i)\}_{i\in\mathbb{Z}}$ as a walk in $E$ of infinite length. Moreover, define the inverse $l^{-1}$ of the line $l=\{(e_i,\delta_i)\}_{i\in\mathbb{Z}}$ as the line $\{(e''_i,\delta''_i)\}_{i\in\mathbb{Z}}$, where $e''_i=e_{-i}$ and
\begin{equation*}
\delta''_i= \begin{cases}
g, \text{ if } \delta_{1-i}=g^{-1}; \\
g^{-1}, \text{ if } \delta_{1-i}=g; \\
\tau, \text{ if } \delta_{1-i}=\tau.
\end{cases}
\end{equation*}

\begin{Def}\label{band}
Let $E=(E,U,\tau,d)$ be a Brauer $G$-set and $l$ be a line of $E$ such that $l[n]=l$ for some positive integer $n$. Then $l$ is called a band of $E$ if it is of the form (consider $l$ as a walk of infinite length)
\begin{multline*}\cdots(e_2|\tau|g^{-l_{1}}\cdot h_1)(g^{-l_{1}}\cdot h_1| g^{-l_{1}}|h_1)(h_1|\tau|g^{k_1}\cdot e_1)(g^{k_1}\cdot e_1|g^{k_1}|e_1) \\ (e_1|\tau|g^{-l_{0}}\cdot h_0)(g^{-l_{0}}\cdot h_0| g^{-l_{0}}|h_0)(h_0|\tau|g^{k_0}\cdot e_0)(g^{k_0}\cdot e_0|g^{k_0}|e_0)\cdots,\end{multline*}
where $0<k_i<d(e_i)$ and $0<l_i<d(h_i)$ for all $i\in\mathbb{Z}$.
\end{Def}

\begin{Def}\label{equivalence relation of lines}
Let $E=(E,U,\tau,d)$ be a Brauer $G$-set. Define an equivalence relation $\sim$ on the set of bands of $E$: $\sim$ is generated by
\begin{itemize}
\item [$(a)$] $l\sim l[i]$ for any integer $i$;
\item [$(b)$] $l\sim l^{-1}$.
\end{itemize}
\end{Def}

For every band $l$ of $E$, denote by $[l]$ the equivalence class of $l$.

\begin{Lem}\label{two-numbers-of-bands-are-equal}
If $E=(E,P,L,d)$ is a finite $f_{ms}$-BG such that each edge of $E$ contains a half-edge $e$ with $d(e)>1$, then there exists a bijection between the set of equivalence classes of bands of $E$ and the set of equivalence classes of bands of $A_E$.
\end{Lem}

\begin{proof}
For any band $l=\{(e_i,\delta_i)\}_{i\in\mathbb{Z}}$ of $E$, let $n$ be the smallest positive integer such that $l=l[n]$. Then $w=(e_n|\delta_n\cdots\delta_1|e_0)$ is a closed walk of $E$, which induces a closed walk $\phi(l)$ of the quiver $Q'_E$: $\phi(l)=f(e_n|\delta_n|e_{n-1})\cdots f(e_2|\delta_2|e_1)f(e_1|\delta_1|e_0)$, where
\begin{equation*}
f(e_i|\delta_i|e_{i-1})= \begin{cases}
L(e_{i-1}), &\text{ if } \delta_i=g; \\
L(e_i)^{-1}, &\text{ if } \delta_i=g^{-1}; \\
1_{P(e_i)}, &\text{ if } \delta_i=\tau.
\end{cases}
\end{equation*} It can be shown that $\phi(l)$ is a band of $A_E$. Moreover, if a band $l'$ of $E$ is obtained from the band $l$ of $E$ by a translation (resp. by taking inverse), then the band $\phi(l')$ of $A_E$ is obtained form the band $\phi(l)$ of $A_E$ by a rotation (resp. by taking inverse), so $\phi$ induces a map $\widetilde{\phi}$ from the set of equivalence classes of bands of $E$ to the set of equivalence classes of bands of $A_E$.

If $l_1,l_2$ are two bands of $E$ such that $\phi(l_1),\phi(l_2)$ are equivalent, then there exists a sequence of bands $b_0=\phi(l_1),b_1,\cdots,b_{k-1},b_k=\phi(l_2)$ of $A_E$ such that for each $1\leq i\leq k$, $b_i$ is obtained from $b_{i-1}$ by a rotation or by taking inverse. Therefore there exists a band $l'_1$ of $E$ which is equivalent to $l_1$ such that $\phi(l'_1)=\phi(l_2)$. It can be shown that $l'_1=l_2$. So $l_1,l_2$ are equivalent and therefore $\widetilde{\phi}$ is injective.

If $b$ is a band of $A_E$, then up to equivalence we may assume that $b$ is of the form $$\alpha^{-1}_{2r,1}\cdots\alpha^{-1}_{2r,n_{2r}}\alpha_{2r-1,n_{2r-1}}\cdots\alpha_{2r-1,1}\cdots\alpha^{-1}_{4,1}\cdots\alpha^{-1}_{4,n_4}\alpha_{3,n_3}\cdots\alpha_{3,1}
\alpha^{-1}_{2,1}\cdots\alpha^{-1}_{2,n_2}\alpha_{1,n_1}\cdots\alpha_{1,1},$$
where each $\alpha_{ij}$ is an arrow of $Q'_E$. Assume that $\alpha_{i,1}=L(e_i)$ for each $1\leq i\leq 2r$. Since $b$ is a band of $A_E$, the path $\alpha_{i,n_i}\cdots\alpha_{i,1}$ does not belong to $I''_E$ for each $i\in\{1,2,\cdots,2r\}$, and $\alpha_{2j,1}\neq \alpha_{2j+1,1}$, $\alpha_{2j-1,n_{2j-1}}\neq \alpha_{2j,n_{2j}}$ for each $j\in\{1,2,\cdots,r\}$ (we define $\alpha_{2n+1,1}=\alpha_{1,1}$). Therefore $\alpha_{i,n_i}\cdots\alpha_{i,1}=L(g^{n_i-1}\cdot e_i)\cdots L(e_i)$ with $0<n_i<d(e_i)$ each $i\in\{1,2,\cdots,2r\}$, and $e_{2j+1}=\tau(e_{2j})$, $g^{n_{2j}}\cdot e_{2j}=\tau(g^{n_{2j-1}}\cdot e_{2j-1})$ for each $j\in\{1,2,\cdots,r\}$ (we define $e_{2r+1}=e_1$ and denote $\tau$ the involution of $E$ as a Brauer $G$-set). Then $b$ induces a closed walk
\begin{multline*}
w=(e_1|\tau|e_{2r})(e_{2r}|g^{-n_{2r}}|g^{n_{2r}}\cdot e_{2r})(g^{n_{2r}}\cdot e_{2r}|\tau|g^{n_{2r-1}}\cdot e_{2r-1})(g^{n_{2r-1}}\cdot e_{2r-1}|g^{n_{2r-1}}|e_{2r-1})\cdots \\ (e_3|\tau|e_2)(e_2|g^{-n_2}|g^{n_2}\cdot e_2)(g^{n_2}\cdot e_2|\tau|g^{n_1}\cdot e_1)(g^{n_1}\cdot e_1|g^{n_1}|e_1)
\end{multline*}
of $E$. Denote $n=\sum_{i=1}^{2r}n_i+2r$, and let $l=\{(e_i,\delta_i)\}_{i\in\mathbb{Z}}$ be the band of $E$ such that $l[n]=l$ and $(e_n|\delta_n\cdots\delta_1|e_0)=w$. Then $b=\phi(l)$, which implies that $\widetilde{\phi}$ is also surjective.
\end{proof}

For any Brauer $G$-set $E$ such that the number of equivalence classes of bands of $E$ is finite, denote by $N_E$ the number of equivalence classes of bands of $E$.

\begin{Prop}\label{the number of equivalence classes of bands}
Let $E=(E,U,\tau,d)$ be a Brauer $G$-set and let $\Pi$ be a finite group of automorphisms of $E$ of order $n$. Then the number of equivalence classes of bands of $E$ is finite if and only if the number of equivalence classes of bands of $E/\Pi$ is finite. In this case we have $N_{E/\Pi}\leq N_E\leq nN_{E/\Pi}$.
\end{Prop}

\begin{proof}
Let $E/\Pi=(E/\Pi,U',\tau',d')$. For every band $l=\{(e_i,\delta_i)\}_{i\in\mathbb{Z}}$ of $E$, denote by $\overline{l}$ the band $\{([e_i],\delta'_i)\}_{i\in\mathbb{Z}}$ of $E/\Pi$, where $[e_i]$ is the $\Pi$-orbit of $e_i$ and
\begin{equation*}
\delta'_i=\begin{cases}
\delta_i, \text{ if } \delta_i=g \text{ or } \delta_i=g^{-1}; \\
\tau', \text{ if } \delta_i=\tau
\end{cases}
\end{equation*}
for each $i\in\mathbb{Z}$. Since $\overline{l^{-1}}=\overline{l}^{-1}$ and $\overline{l[i]}=\overline{l}[i]$ for any integer $i$, the map $l\mapsto\overline{l}$ defines a map $f$ from the set of equivalence classes of bands of $E$ to the set of equivalence classes of bands of $E/\Pi$. For every band $l'=\{(e'_i,\delta'_i)\}_{i\in\mathbb{Z}}$ of $E/\Pi$ and for every $e_0\in E$ with $[e_0]=e'_0$, since the natural projection $E\rightarrow E/\Pi$ is a covering, there exists a unique band $l=\{(e_i,\delta_i)\}_{i\in\mathbb{Z}}$ of $E$ such that $\overline{l}=l'$. Therefore $f$ is surjective.

For every band $l'=\{(e'_i,\delta'_i)\}_{i\in\mathbb{Z}}$ of $E/\Pi$, choose a band $l=\{(e_i,\delta_i)\}_{i\in\mathbb{Z}}$ of $E$ such that $\overline{l}=l'$. Suppose that $f([b])=[l']$ for some band $b$ of $E$. Then $\overline{b}\sim l'$. Since for every band $v$ of $E$, the operation $v\mapsto\overline{v}$ commutes with translation and taking inverse, we imply that there exists a band $c$ of $E$ such that $b\sim c$ and $\overline{c}=l'$. Let $c=\{(h_i,\epsilon_i)\}_{i\in\mathbb{Z}}$. Since $[h_0]=e'_0=[e_0]$, there exists some $\pi\in\Pi$ such that $\pi(e_0)=h_0$. Since $\pi(l)=\{(\pi(e_i),\delta_i)\}_{i\in\mathbb{Z}}$ and $c=\{(h_i,\epsilon_i)\}_{i\in\mathbb{Z}}$ are bands of $E$ which satisfy $\overline{\pi(l)}=l'=\overline{c}$ and $\pi(e_0)=h_0$, we have $\pi(l)=c$. Therefore $[b]=[c]=[\pi(l)]$ and $f^{-1}([l'])=\{[\mu(l)]\mid \mu\in\Pi\}$. Then we have $1\leq |f^{-1}([l'])|\leq n$ for every band $l'$ of $E/\Pi$, and the conclusion holds.
\end{proof}

Let $E=(E,E,\tau,d)$ be an $f_{ms}$-BG and $\sigma:E\rightarrow E$, $e\mapsto g^{d(e)}\cdot e$ be the Nakayama automorphism of $E$. One may note that the diagram $\Gamma(E)$ of an $f_{ms}$-BG $E$ is a graph and the diagram of a Brauer tree is even a tree. Denote by $\langle \sigma\rangle$ the group of automorphisms of $E$ generated by $\sigma$. Recall that $\langle \sigma\rangle$ is called {\it admissible} if each $\langle \sigma\rangle$-orbit of $E$ meets each edge of $E$ in at most one half-edge (cf. \cite[Definition 3.34]{LL2}). The following lemma is straightforward.

\begin{Lem} Let $E$ and $\sigma$ be as above. If $\langle \sigma\rangle$ is admissible, then $E/\langle \sigma\rangle$ is a Brauer graph; if $\langle \sigma\rangle$ is not admissible, then $E/\langle \sigma\rangle$ is a modified BG which contains double half-edges.
\end{Lem}

If $E$ is a modified BG which contains double half-edges, then recall from Lemma \ref{pi is a covering of modified f-BGs} that there exists a regular covering of Brauer $G$-sets $\pi:\widehat{E}\rightarrow E$ with group $\mathbb{Z}/2\mathbb{Z}$, where $\widehat{E}$ is a Brauer graph.

\begin{Def}
Let $E=(E,E,\tau,d)$ be an $f_{ms}$-BG and $\langle \sigma\rangle$ be the group of automorphisms of $E$ generated by the Nakayama automorphism $\sigma$. Define the reduced form $R_E$ of $E$ as follows: if $\langle \sigma\rangle$ is admissible, then $E/\langle \sigma\rangle$ is a Brauer graph and we define $R_E$ to be $E/\langle \sigma\rangle$; if $\langle \sigma\rangle$ is not admissible, then $E/\langle \sigma\rangle$ is a modified BG which contains double half-edges and we define $R_E$ to be the Brauer graph $\widehat{E/\langle\sigma\rangle}$.
\end{Def}

For an example of this construction, see Example \ref{weakly-symmetric-algebra-example} below. Note that if $E$ is finite (resp. connected), then so is $R_E$.

Before stating our main result in this section, we first recall some notions on representation types of finite dimensional algebras (cf. \cite[Section XIX.3]{SS}).

\begin{Def}\label{rep-type} Let $A$ be a finite dimensional algebra over a field $k$.
\begin{enumerate}
\item $A$ is of finite representation type (or $A$ is representation-finite) if the number of isomorphism classes of finitely generated indecomposable $A$-modules is finite.
\item $A$ is tame representation type (or shortly tame) if it is not representation-finite and if for each positive integer $d$, there exist $A$-$k[t]$-bimodules $M_1,\cdots,M_{m_d}$ that are finitely generated and free as right $k[t]$-modules, such that all but a finite number of the isomorphism classes of indecomposable $A$-modules of dimension $d$ are isomorphic to modules of the form $M_i\otimes_{k[t]}k[t]/(t-\lambda)$, where $i\in\{1,2,\cdots,m_d\}$ and $\lambda\in k$.
\item $A$ is domestic if it is not of finite representation type and if there exist a finite number of $A$-$k[t]$-bimodules $M_1,\cdots,M_n$ that are finitely generated and free as right $k[t]$-modules, such that for each positive integer $d$, all but a finite number of the isomorphism classes of indecomposable $A$-modules of dimension $d$ are isomorphic to modules of the form $M_i\otimes_{k[t]}V$, where $i\in\{1,2,\cdots,n\}$ and $V$ is a finite dimensional indecomposable $k[t]$-module. If $N$ is the least number of such $A$-$k[t]$-bimodules, then $A$ is called $N$-domestic.
\end{enumerate}
\end{Def}

It is well-known that every domestic algebra is tame. For a self-injective special biserial algebra $A$, it is well-known that $A$ is representation-finite if and only if $A$ has no bands. Moreover, according to \cite[Theorem 2.1]{ES}, $A$ is domestic if and only if the number of equivalence classes of bands of $A$ is a positive integer.

\begin{Thm}\label{rep type of f-BCA and its reduced form}
Suppose that the field $k$ is algebraically closed. Let $E=(E,E,\tau,d)$ be a finite connected $f_{ms}$-BG. Then $A_E$ is representation-finite (resp. domestic) if and only if $A_{R_E}$ is representation-finite (resp. domestic).
\end{Thm}

\begin{proof}
If $E$ contains an edge $\{e,\tau(e)\}$ with $d(e)=d(\tau(e))=1$, then it can be shown that $E=\{e,g\cdot e,\cdots, g^{n-1}\cdot e,\tau(e),g\cdot\tau(e),\cdots,g^{n-1}\cdot\tau(e)\}$ for some positive integer $n$, where $g^n\cdot e=e$, $g^n\cdot\tau(e)=\tau(e)$, $\tau(g^i\cdot e)=g^i\cdot\tau(e)$ for $0\leq i\leq n-1$, and the degree of each half-edge of $E$ is equal to $1$. Moreover, $R_E$ is a Brauer tree with free f-degree given by the diagram
$$\begin{tikzpicture}
\fill (0,0) circle (0.5ex);
\fill (2,0) circle (0.5ex);
\draw    (0,0)--(2,0);
\node at(2.4,-0.1) {.};
\end{tikzpicture}$$
Both $A_E$ and $A_{R_E}$ are Nakayama algebras of Loewy length $2$, so they are both representation-finite. Therefore we may assume that each edge of $E$ contains a half-edge $e$ with $d(e)>1$.

Since $E$ is finite, the group $\langle\sigma\rangle$ of automorphisms of $E$ is finite. If $\langle \sigma\rangle$ is admissible, then $R_E=E/\langle\sigma\rangle$. By Proposition \ref{the number of equivalence classes of bands}, the number of equivalence classes of bands of $E$ is finite (resp. zero) if and only if the number of equivalence classes of bands of $R_E$ is finite (resp. zero). Since each edge of $E$ (resp. $R_E$) contains a half-edge whose degree is larger than $1$, by Lemma \ref{two-numbers-of-bands-are-equal}, we imply that the number of equivalence classes of bands of $A_E$ is finite (resp. zero) if and only if the number of equivalence classes of bands of $A_{R_E}$ is finite (resp. zero). Since both $A_E$ and $A_{R_E}$ are self-injective special biserial, $A_E$ is representation-finite (resp. domestic) if and only if $A_{R_E}$ is representation-finite (resp. domestic).

If $\langle \sigma\rangle$ is not admissible, then we have $R_E=\widehat{E/\langle\sigma\rangle}$ and $E/\langle\sigma\rangle\cong R_E/\langle\phi\rangle$, where $\phi$ is the automorphism of $R_E=(E/\langle\sigma\rangle)\sqcup(E/\langle\sigma\rangle)$ given by $\phi(h_i)=h_{3-i}$ for every $h\in E/\langle\sigma\rangle$ (see the paragraph before Lemma \ref{pi is a covering of modified f-BGs}). By using Proposition \ref{the number of equivalence classes of bands} twice, we have that the number of equivalence classes of bands of $E$ is finite (resp. zero) if and only if the number of equivalence classes of bands of $R_E$ is finite (resp. zero). So in this case we also have that $A_E$ is representation-finite (resp. domestic) if and only if $A_{R_E}$ is representation-finite (resp. domestic).
\end{proof}

\begin{Rem1} \label{description-in-terms-of-brauer-graph}
(i) Since $f_{ms}$-BGAs are self-injective special biserial, they are all tame. Moreover, according to \cite[Theorem 2.1]{ES}, an $f_{ms}$-BGA is domestic if and only if it is of polynomial growth.

(ii) Theorem \ref{rep type of f-BCA and its reduced form} gives an effective way to determine the representation type of an $f_{ms}$-BGA in terms of the reduced form of its defining $f_{ms}$-BG.

(iii) Different from the BGA case, an $n$-domestic $f_{ms}$-BGA can have arbitrarily large $n$. For example, let $E$ be an $f_{ms}$-BG given by the diagram
\vspace{0.5cm}
\begin{center}
\tikzset{every picture/.style={line width=0.75pt}}
\begin{tikzpicture}[x=15pt,y=15pt,yscale=1,xscale=1]
\fill (0,0) circle (0.5ex);
\fill (5,0) circle (0.5ex);
\node at(-1.5,1) {$1$};
\node at(-0.6,1.2) {$2$};
\node at(1,0.2) {$n$};
\node at(6,0) {,};
\fill (-0.1,1) circle (0.1ex);
\fill (0.15,0.95) circle (0.1ex);
\fill (0.35,0.9) circle (0.1ex);
\fill (0.55,0.8) circle (0.1ex);
\fill (0.7,0.7) circle (0.1ex);
\fill (4.8,0.9) circle (0.1ex);
\fill (5.05,0.9) circle (0.1ex);
\fill (5.25,0.85) circle (0.1ex);
\fill (5.45,0.8) circle (0.1ex);
\fill (5.6,0.7) circle (0.1ex);
\draw    (0,0) .. controls (-4,3) and (1,3) .. (5,0) ;
\draw    (0,0) .. controls (-2,3) and (4,3) .. (5,0) ;
\draw    (0,0) .. controls (4,3) and (9,3) .. (5,0) ;
\end{tikzpicture}
\end{center}
where the degree function of $E$ takes constant value $2$. By using the covering map $E\rightarrow E/\langle \sigma\rangle$ we see that the $f_{ms}$-BGA $A_E$ is $n$-domestic.

(iv) According to \cite[Theorem 5.13]{GLL}, the graded algebras associated with the radical filtration of $A_E$ and $A_{R_E}$ also have the same representation type.
\end{Rem1}

\begin{Ex1} \label{weakly-symmetric-algebra-example}
Let $E$ be the $f_{ms}$-BG given by the diagram
$$
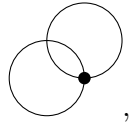
\begin{tikzpicture}
\draw (0,0) circle (0.5);
\draw (0.5,0.5) circle (0.5);
\fill (0.5,0) circle (0.5ex);
\end{tikzpicture},$$
where the f-degree of the unique vertex of $E$ is $\frac{1}{2}$ and the $G$-action is induced from the clockwise order on the half-edges around this vertex. Then $\langle\sigma\rangle$ is not admissible and $E/\langle \sigma\rangle$ is a Brauer $G$-set given in Example \ref{modified-fms-BG-Ex-1}. Therefore $R_E$ is a BG with free f-degree given by the diagram
$$
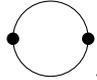
\begin{tikzpicture}
\draw (0,0) circle (0.5);
\fill (0.5,0) circle (0.5ex);
\fill (-0.5,0) circle (0.5ex);
\end{tikzpicture}.$$
According to Theorem \ref{domestic-BGA}, $A_{R_E}$ is domestic, therefore $A_E$ is also domestic. One of particularly interests is that $A_{R_E}$ is symmetric but $A_E$ is weakly symmetric with nonidentity Nakayama automorphism. Indeed, $A_E=kQ_E/I_E$, where $Q_E$ is the quiver
$$\begin{tikzpicture}
\draw[->] (0,0.6) -- (5,0.6);
\draw[->] (0,-0.2) -- (5,-0.2);
\draw[->] (5,0.2) -- (0,0.2);
\draw[->] (5,-0.6) -- (0,-0.6);
\node at(-0.5,0) {$1$};
\node at(5.5,0) {$2$};
\node at(2.5,0.8) {$\alpha_{1}$};
\node at(2.5,0.4) {$\alpha_{2}$};
\node at(2.5,0) {$\alpha_{3}$};
\node at(2.5,-0.4) {$\alpha_{4}$};
\end{tikzpicture},$$
and $I_E$ is generated by $\alpha_4\alpha_1,\alpha_1\alpha_2,\alpha_2\alpha_3,\alpha_3\alpha_4,\alpha_2\alpha_1-\alpha_4\alpha_3,\alpha_3\alpha_2-\alpha_1\alpha_4$. The structure of indecomposable projective modules are as follows:
\begin{center}
\tikzset{every picture/.style={line width=0.75pt}}
\begin{tikzpicture}
\node at(0,0) {$1$};
\node at(-1,1) {$2$};
\node at(1,1) {$2$};
\node at(0,2) {$1$};
\node at(-0.6,1.7) {$\alpha_1$};
\node at(0.6,1.7) {$\alpha_3$};
\node at(-0.6,0.3) {$\alpha_2$};
\node at(0.6,0.3) {$\alpha_4$};
\draw    (-0.2,0.2) -- (-0.8,0.8) ;
\draw    (0.2,0.2) -- (0.8,0.8) ;
\draw    (-0.8,1.2) -- (-0.2,1.8) ;
\draw    (0.8,1.2) -- (0.2,1.8) ;
\node at(3,0) {$2$};
\node at(2,1) {$1$};
\node at(4,1) {$1$};
\node at(3,2) {$2$};
\node at(2.4,1.7) {$\alpha_2$};
\node at(3.6,1.7) {$\alpha_4$};
\node at(2.4,0.3) {$\alpha_3$};
\node at(3.6,0.3) {$\alpha_1$};
\draw    (2.8,0.2) -- (2.2,0.8) ;
\draw    (3.2,0.2) -- (3.8,0.8) ;
\draw    (2.2,1.2) -- (2.8,1.8) ;
\draw    (3.8,1.2) -- (3.2,1.8) ;
\end{tikzpicture}
\end{center}
The Nakayama automorphism of $A_E$ is induced by $e_i\mapsto e_i$ for $i=1, 2$, $\alpha_i\mapsto \alpha_{i+2}$ for $i\in \mathbb{Z}/4\mathbb{Z}=\{1,2,3,4\}$. However, $A_{R_E}$ has the same quiver with the following structure of indecomposable projective modules:
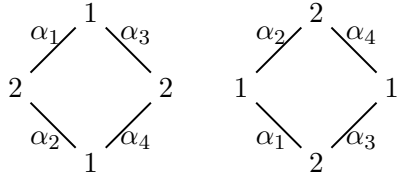
\begin{center}
\tikzset{every picture/.style={line width=0.75pt}}
\begin{tikzpicture}
\node at(0,0) {$1$};
\node at(-1,1) {$2$};
\node at(1,1) {$2$};
\node at(0,2) {$1$};
\node at(-0.6,1.7) {$\alpha_1$};
\node at(0.6,1.7) {$\alpha_3$};
\node at(-0.6,0.3) {$\alpha_2$};
\node at(0.6,0.3) {$\alpha_4$};
\draw    (-0.2,0.2) -- (-0.8,0.8) ;
\draw    (0.2,0.2) -- (0.8,0.8) ;
\draw    (-0.8,1.2) -- (-0.2,1.8) ;
\draw    (0.8,1.2) -- (0.2,1.8) ;
\node at(3,0) {$2$};
\node at(2,1) {$1$};
\node at(4,1) {$1$};
\node at(3,2) {$2$};
\node at(2.4,1.7) {$\alpha_2$};
\node at(3.6,1.7) {$\alpha_4$};
\node at(2.4,0.3) {$\alpha_1$};
\node at(3.6,0.3) {$\alpha_3$};
\draw    (2.8,0.2) -- (2.2,0.8) ;
\draw    (3.2,0.2) -- (3.8,0.8) ;
\draw    (2.2,1.2) -- (2.8,1.8) ;
\draw    (3.8,1.2) -- (3.2,1.8) ;
\end{tikzpicture}
\end{center}
\end{Ex1}

\section{Fundamental groups of fractional Brauer graphs of type MS}

In this section, we characterize representation-finite and domestic $f_{ms}$-BGA $A_E$ in terms of the fundamental group $\Pi(E)$ of the defining $f_{ms}$-BG $E$. Our method for calculating the fundamental group of an $f_{ms}$-BG $E$ is as follows: we first calculate the fundamental groups of modified BGs, then we use the covering $E\rightarrow E/\langle\sigma\rangle$ of Brauer $G$-sets (here $E/\langle\sigma\rangle$ is a modified BG) to reduce the calculation of the fundamental group of $E$ to the calculation of the fundamental group of $E/\langle\sigma\rangle$ by the method mentioned in Remark \ref{general-method-to-compute-the-fundamental-group}.

\subsection{The fundamental groups of modified BGs}
\

In this subsection we calculate the fundamental groups of modified BGs using an analogy of the Van Kampen theorem for Brauer $G$-sets.

We denote by $F\langle x_1,x_2,\cdots,x_n\rangle$ the free group on the set $\{x_1,x_2,\cdots,x_n\}$. The following Lemma should be compared with \cite[Lemma 6.6]{LL2}.

\begin{Lem}\label{a calculation of modified fundamental group}
Let $E=(E,E,\tau,d)$ be the modified BG given by the diagram
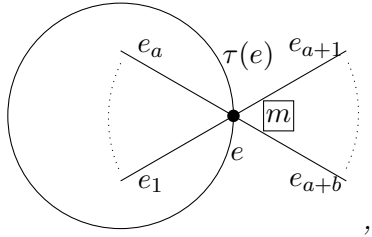
\begin{center}
\begin{tikzpicture}
\draw (0,0) circle (1.5);
\fill (1.5,0) circle (0.5ex);
\node at(2.1,0) {$m$};
\node at(1.55,-0.5) {$e$};
\node at(1.7,0.8) {$\tau(e)$};
\node at(0.4,-0.9) {$e_1$};
\node at(0.4,0.9) {$e_{a}$};
\node at(2.6,-0.9) {$e_{a+b}$};
\node at(2.6,0.9) {$e_{a+1}$};
\draw (1.9,-0.2) rectangle (2.3,0.2);
\draw (0,0.866)--(1.5,0);
\draw[dotted] (0,0.7) arc (155:205:1.732);
\draw (0,-0.866)--(1.5,0);
\draw (3,-0.866)--(1.5,0);
\draw (3,0.866)--(1.5,0);
\draw[dotted] (3,-0.7) arc (-25:25:1.732);
\end{tikzpicture},
\end{center}
where $E$ contains $n=a+b$ double half-edges $e_1$, $\cdots$, $e_{a+b}$, and the f-degree of the unique vertex of $E$ is $m$. Then $\Pi_{m}(E,e)\cong F\langle x,y,z_1,\cdots,z_n\rangle/\langle x^{m}y=yx^{m},x^{m}z_i=z_i x^{m}(1\leq i\leq n), z_{i}^{2}=1(1\leq i\leq n)\rangle$.
\end{Lem}

\begin{proof}
Define a group homomorphism $f':F\langle x,y,z_1,\cdots,z_n\rangle\rightarrow\Pi_{m}(E,e)$ as follows: $f'(x)=[(e|g^{n+2}|e)]$, $f'(y)=[(e|\tau g^{a+1}|e)]$, and \begin{equation*}
f'(z_i)= \begin{cases}
[(e|g^{-i}\tau g^i|e)], & \text{if } 1\leq i\leq a; \\
[(e|g^{-i-1}\tau g^{i+1}|e)], & \text{if } a+1\leq i\leq n.
\end{cases}
\end{equation*}
By imitating the proof of \cite[Lemma 6.6]{LL2}, it can be shown that for every closed special walk $w=(e|g^{i_k}\tau g^{i_{k-1}}\tau\cdots\tau g^{i_1}\tau g^{i_0}|e)$ of $E$ at $e$, $[w]$ belongs to the image of $f'$ (by induction on the number of times that $\tau$ appears in $w$). Then according to Proposition \ref{modified unique factorization}, $f'$ is surjective. Moreover, it is straightforward to show that the kernel of $f'$ contains the normal subgroup of \\ $F\langle x,y,z_1,\cdots,z_n\rangle$ generated by the relations $x^{m}y=yx^{m},x^{m}z_i=z_i x^{m}(1\leq i\leq n), z_{i}^{2}=1(1\leq i\leq n)$. Therefore $f'$ induces a surjective group homomorphism $f:F\langle x,y,z_1,\cdots,z_n\rangle/\langle x^{m}y=yx^{m},x^{m}z_i=z_i x^{m}(1\leq i\leq n), z_{i}^{2}=1(1\leq i\leq n)\rangle\rightarrow\Pi_{m}(E,e)$.

We need to show that $f$ is also injective. Denote by $\overline{a}$ the image of $a\in F\langle x,y,z_1,\cdots,z_n\rangle$ in $F\langle x,y,z_1,\cdots,z_n\rangle/\langle x^{m}y=yx^{m},x^{m}z_i=z_i x^{m}(1\leq i\leq n), z_{i}^{2}=1(1\leq i\leq n)\rangle$. Note that each element of $F\langle x,y,z_1,\cdots,z_n\rangle/\langle x^{m}y=yx^{m},x^{m}z_i=z_i x^{m}(1\leq i\leq n), z_{i}^{2}=1(1\leq i\leq n)\rangle$ is of the form $\overline{x^{lm}\delta_{k}^{l_k}\cdots\delta_{1}^{l_1}}$, where $l\in\mathbb{Z}$, $l_1,\cdots,l_k\in\mathbb{Z}-\{0\}$, $\delta_i\in\{x,y,z_1,\cdots,z_n\}$ for $1\leq i\leq k$, such that \\ (1) $\delta_{i-1}\neq\delta_i$ for $1<i\leq n$; \\ (2) if $\delta_i=x$, then $0< l_i< m$; \\ (3) if $\delta_i=z_r$ for some $1\leq r\leq n$, then $l_i=1$. \\ If $\overline{x^{lm}\delta_{k}^{l_k}\cdots\delta_{1}^{l_1}}\in \mathrm{ker}(f)$, according to Proposition \ref{modified unique factorization}, it is straightforward to show that $k=l=0$, therefore $\overline{x^{lm}\delta_{k}^{l_k}\cdots\delta_{1}^{l_1}}=1$.
\end{proof}

The following lemma should be compared with \cite[Lemma 6.8]{LL2}. For the definition of the fundamental groupoid $\Pi(Q)$ of a quiver $Q$, we refer to \cite[Section 4]{LL2}. By abuse of notation, we also denote by $w$ the corresponding morphism in $\Pi(Q)$ for a walk $w$ in the quiver $Q$.
\begin{Lem}\label{a calculation of modified fundamental groupoid}
Let $E=(E,E,\tau,d)$ be the modified BG given by the diagram
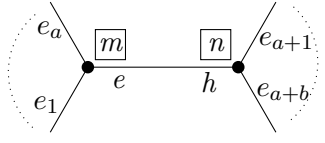
\begin{center}
\begin{tikzpicture}
\draw (0,0)--(2,0);
\draw (-0.5,0.866)--(0,0);
\draw (-0.5,-0.866)--(0,0);
\draw (2.5,0.866)--(2,0);
\draw (2.5,-0.866)--(2,0);
\draw[dotted] (-0.7,0.7) arc (130:230:1);
\draw[dotted] (2.7,-0.7) arc (-50:50:1);
\node at(-0.5,0.5) {$e_a$};
\node at(-0.55,-0.5) {$e_1$};
\node at(2.64,0.34) {$e_{a+1}$};
\node at(2.6,-0.34) {$e_{a+b}$};
\node at(0.4,-0.2) {$e$};
\node at(1.6,-0.2) {$h$};
\fill (0,0) circle (0.5ex);
\fill (2,0) circle (0.5ex);
\node at(0.3,0.3) {$m$};
\draw (0.1,0.1) rectangle (0.5,0.5);
\node at(1.7,0.3) {$n$};
\draw (1.5,0.1) rectangle (1.9,0.5);
\end{tikzpicture},
\end{center}
where $E$ contains $a+b$ double half-edges $e_1$, $\cdots$, $e_a$, $e_{a+1}$, $\cdots$, $e_{a+b}$, and the f-degree of the vertex on the left (resp. right) is $m$ (resp. $n$). Let  $A=\{e,h\}$ be a subset of $E$. Then the fundamental groupoid $\Pi_{m}(E,A)$ is isomorphic to
$$\mathscr{F}/\langle t x^{m}=u^{n} t, x^{m} z_{i}=z_{i} x^{m}, z_{i}^{2}=1_{p} (1\leq i\leq a), u^{n} v_{j}=v_{j} u^{n}, v_{j}^{2}=1_{q} (1\leq j\leq b)\rangle,$$
where $\mathscr{F}$ is the fundamental groupoid of the quiver
\begin{center}
\begin{tikzpicture}
\node at(-0.1,0.1) (a) {};
\draw[->] (a.north) to[out=90,in=45] ([xshift=-0.7cm,yshift=0.7cm]a.east) to[out=-135,in=-180] (a.west);
\node at(0.1,0.1) (a) {};
\draw[->] (a.east) to[out=0,in=-45] ([xshift=0.7cm,yshift=0.7cm]a.south) to[out=135,in=90] (a.north);
\node at(0.1,-0.1) (a) {};
\draw[->] (a.south) to[out=-90,in=-135] ([xshift=0.7cm,yshift=-0.7cm]a.north) to[out=45,in=0] (a.east);
\draw[dotted] (-0.7,0) arc (180:270:0.7);
\node at(2.9,0.1) (a) {};
\draw[->] (a.west) to[out=180,in=-135] ([xshift=-0.7cm,yshift=0.7cm]a.south) to[out=45,in=90] (a.north);
\node at(3.1,0.1) (a) {};
\draw[->] (a.north) to[out=90,in=135] ([xshift=0.7cm,yshift=0.7cm]a.west) to[out=-45,in=0] (a.east);
\node at(2.9,-0.1) (a) {};
\draw[->] (a.south) to[out=-90,in=-45] ([xshift=-0.7cm,yshift=-0.7cm]a.north) to[out=135,in=180] (a.west);
\draw[dotted] (3,-0.7) arc (270:360:0.7);
\draw[->] (0.2,0) -- (2.8,0);
\fill (0,0) circle (0.5ex);
\fill (3,0) circle (0.5ex);
\node at(0.7,0.9) {$x$};
\node at(2.3,0.9) {$u$};
\node at(1.5,0.2) {$t$};
\node at(0.7,-0.9) {$z_1$};
\node at(-0.9,0.7) {$z_a$};
\node at(2.3,-0.9) {$v_{1}$};
\node at(4.1,0.7) {$v_{b}$};
\node at(-0.2,-0.2) {$p$};
\node at(3.2,-0.2) {$q$};
\end{tikzpicture}.
\end{center}
\end{Lem}

\begin{proof}
Define a morphism of groupoids $F':\mathscr{F}\rightarrow\Pi_{m}(E,A)$ as follows: $F'(p)=e$, $F'(q)=h$, $F'(x)=[(e|g^{a+1}|e)]$, $F'(u)=[(h|g^{b+1}|h)]$, $F'(t)=[(h|\tau|e)]$, $F'(z_i)=[(e|g^{-i}\tau g^i|e)]$ ($1\leq i\leq a$), $F'(v_j)=[(h|g^{-j}\tau g^j|h)]$ ($1\leq j\leq b$). We first need to show that $F'$ is full. According to Proposition \ref{modified unique factorization}, it suffices to show that for every special walk $w$ with $s(w),t(w)\in A$, $[w]$ belongs to the image of $F'$.

Let $w=g^{i_k}\tau g^{i_{k-1}}\tau\cdots\tau g^{i_1}\tau g^{i_0}$ be a special walk of $E$ with $s(w),t(w)\in A$. We will show that $[w]$ belongs to the image of $F'$ by induction on $k$, that is, the number of times that $\tau$ appears in $w$. If $k=0$, then $w=(e|g^{r(a+1)}|e)$ or $w=(h|g^{s(b+1)}|h)$, where $r,s\in\mathbb{Z}$, so $[w]$ belongs to the image of $F'$. Now suppose that $k>0$. We may assume that $w$ contains no subwalk of the form $(h|\tau|e)$ or $(e|\tau|h)$, otherwise $w$ can be factored as a composition of special subwalks whose sources and terminals belong to $A$ with $k$ smaller. Therefore we may assume that $w$ is of the form
\begin{equation*}
(e|g^{i_k}|c_k)(c_k|\tau|c_k)(c_{k}|g^{i_{k-1}}|c_{k-1})(c_{k-1}|\tau|c_{k-1})\cdots(c_2|\tau|c_2)(c_2|g^{i_1}|c_1)(c_1|\tau|c_1)(c_1|g^{i_0}|e),  \end{equation*}
where $c_1,\cdots,c_k\in\{e_1,\cdots,e_a\}$ and $0<i_0,i_1,\cdots, i_k<m(a+1)$. It is straightforward to show that $[w]$ belongs to the image of $F'$.

Let $y=t^{-1}ut$ and $z_{a+i}=t^{-1}v_i t$ for $1\leq i\leq b$. Then $\mathscr{F}(p,p)$ is a free group with free generators $x,y,z_1,\cdots z_{a+b}$. Let $f':\mathscr{F}(p,p)\rightarrow\Pi_{m}(E,e)$ be the group homomorphism induced by $F'$. Since $F'$ is full, $f'$ is surjective. We have $f'(x)=[(e|g^{a+1}|e)]$, $f'(y)=[(e|\tau g^{b+1}\tau|e)]$, and \begin{equation*}
f'(z_i)= \begin{cases}
[(e|g^{-i}\tau g^i|e)], & \text{if } 1\leq i\leq a; \\
[(e|\tau g^{-i+a}\tau g^{i-a}\tau|e)], & \text{if } a+1\leq i\leq a+b.
\end{cases}
\end{equation*}
It is straightforward to show that the normal subgroup of $\mathscr{F}(p,p)=F\langle x,y,z_1,\cdots,z_{a+b}\rangle$ generated by relations $x^m=y^n,x^m z_i=z_i x^m,z_{i}^{2}=1 (1\leq i\leq a+b)$ is contained in the kernel of $f'$, therefore $f'$ induces a surjective group homomorphism
$$f:\mathscr{F}(p,p)/\langle x^m=y^n,x^m z_i=z_i x^m,z_{i}^{2}=1 (1\leq i\leq a+b)\rangle\rightarrow\Pi_{m}(E,e).$$
It can be shown directly that each element of $\mathscr{F}(p,p)/\langle x^m=y^n,x^m z_i=z_i x^m,z_{i}^{2}=1 (1\leq i\leq a+b)\rangle$ is of the form $\overline{x^{lm}\delta_{k}^{l_k}\cdots\delta_{1}^{l_1}}$, where $l\in\mathbb{Z}$, $l_1,\cdots,l_k\in\mathbb{Z}-\{0\}$, \\ $\delta_i\in\{x,y,z_1,\cdots,z_{a+b}\}$ for $1\leq i\leq k$, such that \\ (1) $\delta_{i-1}\neq\delta_i$ for $1<i\leq n$; \\ (2) if $\delta_i=x$, then $0< l_i< m$; \\ (3) if $\delta_i=y$, then $0< l_i< n$; \\ (4) if $\delta_i=z_r$ for some $1\leq r\leq a+b$, then $l_i=1$. \\ If $\overline{x^{lm}\delta_{k}^{l_k}\cdots\delta_{1}^{l_1}}\in\mathrm{ker}(f)$, according to Proposition \ref{modified unique factorization}, it is straightforward to show that $k=l=0$, therefore $\overline{x^{lm}\delta_{k}^{l_k}\cdots\delta_{1}^{l_1}}=1$. Thus $f$ is also injective.

Let $\mathscr{G}$ be the groupoid
$$\mathscr{F}/\langle t x^{m}=u^{n} t, x^{m} z_{i}=z_{i} x^{m}, z_{i}^{2}=1_{p} (1\leq i\leq a), u^{n} v_{j}=v_{j} u^{n}, v_{j}^{2}=1_{q} (1\leq j\leq b)\rangle.$$
Then it is straightforward to show that $F'$ induces a morphism of groupoids $F:\mathscr{G}\rightarrow\Pi_{m}(E,A)$. Since $\mathscr{G}(p,p)=\mathscr{F}(p,p)/\langle x^m=y^n,x^m z_i=z_i x^m,z_{i}^{2}=1 (1\leq i\leq a+b)\rangle$, and since $F$ induces a group isomorphism $f:\mathscr{G}(p,p)\rightarrow\Pi_{m}(E,e)$, by \cite[Lemma 6.7]{LL2}, $F$ is an isomorphism of groupoids.
\end{proof}

The following lemma is an analogy of \cite[Lemma 6.3]{LL2}. We omit the proof of it.
\begin{Lem}\label{isomorphism of modified fundamental groupoids}
Let $E=(E,U,\tau,d)$ be a connected Brauer $G$-set. Let $C$ be a subset of $E-U$ such that for each $e\in E$, $e\in C$ if and only if $g^{d(e)}(e)\in C$ (We denote by $g^{n}(h)$ the action of $g^n$ on $h$ for every $h\in E$ and $n\in\mathbb{Z}$).
Let $E'=E-C$ and assume that $E'\neq\emptyset$. Define a Brauer $G$-set structure $(E',U,\tau,d')$ on $E'$ as follows: the action of $G=\langle g\rangle$ on $E'$ is given by
\begin{equation*}
g\cdot h= \begin{cases}
g(h), \text{ if } g(h)\in E'; \\
g^{N}(h), \text{ if } g(h)\notin E', \text{ where } N \text{ is the minimal positive integer such that } g^{N}(h)\in E';
\end{cases}
\end{equation*}
the degree function $d'$ is given by
$$d'(h)=d(h)-|\{i\mid 1\leq i\leq d(h)-1  \text{ and } g^{i}(h)\notin E'\}|.$$ Then $E'$ is a connected Brauer $G$-set, and the groupoids $\Pi_{m}(E,E')$ and $\Pi_{m}(E',E')$ are isomorphic.
\end{Lem}

Especially, if we choose $C=E-U$, then the fundamental groups of $E'$ and $E$ are isomorphic.

\medskip
Let $E=(E,U,\tau,d)$ be a Brauer $G$-set. A {\it sub-Brauer $G$-set} $E'=(E',U',\tau',d')$ of $E$ is a Brauer $G$-set such that $E'$ is a sub-$G$-set of $E$ and the inclusion $E'\rightarrow E$ is a morphism of Brauer $G$-sets. That is, $E'$ is a sub-$G$-set of $E$, $U'$ is a subset of $E'\cap U$ such that $\sigma(U')=U'$ and $\tau(U')=U'$ ($\sigma$ is the Nakayama automorphism of $E$), $\tau'$ is the restriction of $\tau$ to $U'$, and $d'$ is the restriction of $d$ to $E'$.

For a set of sub-Brauer $G$-sets $\{E_{\alpha}=(E_{\alpha},U_{\alpha},\tau_{\alpha},d_{\alpha})\}$ of $E$, define the union (resp. the intersection) of them as $\cup_{\alpha}E_{\alpha}=(\cup_{\alpha}E_{\alpha},\cup_{\alpha}U_{\alpha},\tau',d')$ (resp. $\cap_{\alpha}E_{\alpha}=(\cap_{\alpha}E_{\alpha},\cap_{\alpha}U_{\alpha},\tau'',d'')$), where $\tau'$, (resp. $\tau''$) is the restriction of $\tau$ to $\cup_{\alpha}U_{\alpha}$ (resp. $\cap_{\alpha}U_{\alpha}$), and $d'$ (resp. $d''$) is the restriction of $d$ to $\cup_{\alpha}E_{\alpha}$ (resp. $\cap_{\alpha}E_{\alpha}$).

Similar to \cite[Proposition 6.2]{LL2}, we have the following analogy of the Van Kampen theorem, and we omit the proof of it. Note that in this proposition we do not require the family of sub-Brauer $G$-sets $\{E_{\alpha}\}_{\alpha\in I}$ of $E$ being admissible.

\begin{Prop}\label{modified Van-Kampen}
Let $E$ be a Brauer $G$-set, which is the union of a family of sub-Brauer $G$-sets $\{E_{\alpha}\}_{\alpha\in I}$ which is closed under finite intersections. Let $A$ be a subset of $\bigcap_{\alpha\in I}E_{\alpha}$ such that for each $\alpha\in I$, $A$ meets each connected component of $E_{\alpha}$. Then the groupoid $\Pi_{m}(E,A)$ is the direct limit of groupoids $\Pi_{m}(E_{\alpha},A)$.
\end{Prop}

Now we can calculate the fundamental group of a Brauer $G$-set with integral f-degree (for example, modified BGs).

\begin{Prop}\label{modified fundamental group of modified f-BG}
Let $E=(E,U,\tau,d)$ be a finite connected Brauer $G$-set of integral f-degree with $n$ vertices $v_1$, $\cdots$, $v_n$, $k$ edges, and $l$ double half-edges. Let $d_i$ be the f-degree of $v_i$ for each $1\leq i\leq n$, and let $r=k-n+1$. Then the  fundamental group of $E$ is isomorphic to \\ $F\langle a_1,\cdots,a_n,b_1,\cdots,b_r,c_1,\cdots,c_l\rangle/\langle a_{1}^{d_1}=\cdots=a_{n}^{d_n}, a_{1}^{d_1} b_i=b_i a_{1}^{d_1}$ $(1\leq i\leq r), a_{1}^{d_1} c_j=c_j a_{1}^{d_1}, c_{j}^{2}=1$ $(1\leq j\leq l)\rangle$.
\end{Prop}

\begin{proof}
Note that if we take $C=E-U$, then the modified BG $E'=(E',E',\tau,d')$ constructed in Lemma \ref{isomorphism of modified fundamental groupoids} also has $n$ vertices, $k$ edges, $l$ double half-edges, and the f-degree of each vertex in $E'$ is equal to the f-degree of the corresponding vertex in $E$. Moreover, according to Lemma \ref{isomorphism of modified fundamental groupoids}, the fundamental groups of $E'$ and $E$ are isomorphic. Therefore we may assume that $E=U$.

Let $\{e_1,\tau(e_1)\}$, $\cdots$, $\{e_k,\tau(e_k)\}$ be all edges of $E$. For each $1\leq i\leq k$, define a sub-Brauer $G$-set $E_i=(E_i,U_i,\tau_i,d_i)$ of $E$ as follows: $E_i=E$ as $G$-sets; the subset $U_i$ of $E_i$ is given by $U_i=\{e_i,\tau(e_i)\}$; the involution $\tau_i$ is the restriction of $\tau$ to $U_i$, and the degree function $d_i$ is equal to $d$. Let $E'=(E',U',\tau',d')$ be the intersection of all the $E_i$'s. For each $1\leq i\leq n$, choose $h_i\in v_i$, and let $A=\{h_1,\cdots,h_n\}$ be a subset of $E$. The family $\{E',E_1,\cdots,E_k\}$ of sub-Brauer $G$-sets is closed under finite intersections, and the union of them is $E$. Moreover, $A$ meets each connected component of $E'$ and each connected component of every $E_i$.

For each vertex $v_j$ of $E$, let $l_j$ be the number of double half-edges of $E$ which belongs to $v_j$. For each $1\leq i\leq k$, if the two half-edges $e_i,\tau(e_i)$ belong to the same vertex $v_j$, then we denote by $\mathscr{F}_{i}$ the fundamental groupoid of quiver
$$
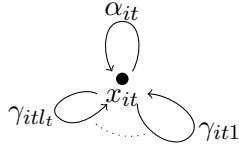
\begin{tikzpicture}
\node at(-0.2,-0.2) (a) {};
\draw[->] (a.west) to[out=150,in=135] ([xshift=-0.8cm,yshift=-0.3cm]a.east) to[out=-45,in=-135] (a.south);
\node at(-0.1,0.1) (a) {};
\draw[->] (a.north) to[out=75,in=45] ([xshift=-0.5cm,yshift=0.7cm]a.east) to[out=-135,in=150] (a.west);
\node at(0.1,0.1) (a) {};
\draw[->] (a.east) to[out=0,in=-45] ([xshift=0.7cm,yshift=0.7cm]a.south) to[out=135,in=60] (a.north);
\node at(0.2,-0.2) (a) {};
\draw[->] (a.south) to[out=-90,in=-135] ([xshift=0.7cm,yshift=-0.7cm]a.north) to[out=45,in=0] (a.east);
\draw[dotted] (-0.35,-0.6) arc (230:290:0.7);
\fill (0,0) circle (0.5ex);
\node at(0.7,0.9) {$\beta_i$};
\node at(1.3,-0.7) {$\gamma_{ij1}$};
\node at(-1.2,-0.5) {$\gamma_{ijl_j}$};
\node at(-0.8,0.7) {$\alpha_{ij}$};
\node at(0,-0.25) {$x_{ij}$};
\end{tikzpicture},$$
and for each $1\leq t\leq n$ with $l\neq j$, denote by $\mathscr{F}_{it}$ the fundamental groupoid of quiver
$$\begin{tikzpicture}
\node at(-0.2,-0.2) (a) {};
\draw[->] (a.west) to[out=150,in=135] ([xshift=-0.8cm,yshift=-0.3cm]a.east) to[out=-45,in=-135] (a.south);
\node at(-0,0.1) (a) {};
\draw[->] (a.east) to[out=60,in=0] ([xshift=0cm,yshift=0.8cm]a.south) to[out=180,in=120] (a.west);
\node at(0.2,-0.2) (a) {};
\draw[->] (a.south) to[out=-90,in=-135] ([xshift=0.7cm,yshift=-0.7cm]a.north) to[out=45,in=0] (a.east);
\draw[dotted] (-0.35,-0.6) arc (230:290:0.7);
\fill (0,0) circle (0.5ex);
\node at(1.3,-0.7) {$\gamma_{it1}$};
\node at(-1.2,-0.5) {$\gamma_{itl_t}$};
\node at(0,0.9) {$\alpha_{it}$};
\node at(0,-0.25) {$x_{it}$};
\end{tikzpicture}.$$
By Lemma \ref{a calculation of modified fundamental group} and Lemma \ref{isomorphism of modified fundamental groupoids}, $\Pi_{m}(E_i,A)$ is isomorphic to the groupoid \\ $\Sigma_i=\bigsqcup_{1\leq t\leq n, t\neq j}(\mathscr{F}_{it}/\langle\alpha_{it}^{d_t}\gamma_{itp}=\gamma_{itp}\alpha_{it}^{d_t}$, $\gamma_{itp}^{2}=1$ $(1\leq p\leq l_t) \rangle)\sqcup (\mathscr{F}_{i}/\langle \alpha_{ij}^{d_j}\beta_j=\beta_j\alpha_{ij}^{d_j}$, $ \alpha_{ij}^{d_j}\gamma_{ijp}=\gamma_{ijp}\alpha_{ij}^{d_j}$, $\gamma_{itp}^{2}=1$ $(1\leq p\leq l_i) \rangle)$. If the two half-edges $e_i,\tau(e_i)$ belong to two different vertices $v_{j_1}$, $v_{j_2}$ with $j_1<j_2$, then we denote by $\mathscr{F}_{i}$ the fundamental groupoid of quiver
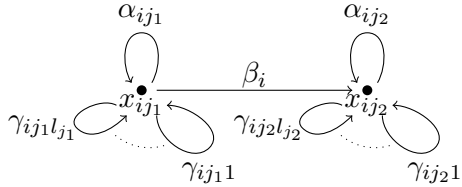
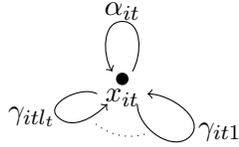
$$\begin{tikzpicture}
\node at(-0.2,-0.2) (a) {};
\draw[->] (a.west) to[out=150,in=135] ([xshift=-0.8cm,yshift=-0.3cm]a.east) to[out=-45,in=-135] (a.south);
\node at(-0,0.1) (a) {};
\draw[->] (a.east) to[out=60,in=0] ([xshift=0cm,yshift=0.8cm]a.south) to[out=180,in=120] (a.west);
\node at(0.2,-0.2) (a) {};
\draw[->] (a.south) to[out=-90,in=-135] ([xshift=0.7cm,yshift=-0.7cm]a.north) to[out=45,in=0] (a.east);
\node at(2.8,-0.2) (a) {};
\draw[->] (a.west) to[out=150,in=135] ([xshift=-0.8cm,yshift=-0.3cm]a.east) to[out=-45,in=-135] (a.south);
\node at(3,0.1) (a) {};
\draw[->] (a.east) to[out=60,in=0] ([xshift=0cm,yshift=0.8cm]a.south) to[out=180,in=120] (a.west);
\node at(3.2,-0.2) (a) {};
\draw[->] (a.south) to[out=-90,in=-135] ([xshift=0.7cm,yshift=-0.7cm]a.north) to[out=45,in=0] (a.east);
\draw[->] (0.2,0) -- (2.8,0);
\draw[dotted] (-0.35,-0.6) arc (230:290:0.7);
\draw[dotted] (2.65,-0.6) arc (230:290:0.7);
\fill (0,0) circle (0.4ex);
\fill (3,0) circle (0.4ex);
\node at(0.9,-1.1) {$\gamma_{ij_{1}1}$};
\node at(-1.3,-0.5) {$\gamma_{ij_{1}l_{j_{1}}}$};
\node at(3.9,-1.1) {$\gamma_{ij_{2}1}$};
\node at(1.7,-0.5) {$\gamma_{ij_{2}l_{j_{2}}}$};
\node at(0,1) {$\alpha_{ij_1}$};
\node at(3,1) {$\alpha_{ij_2}$};
\node at(1.5,0.2) {$\beta_i$};
\node at(0,-0.2) {$x_{i j_1}$};
\node at(3,-0.2) {$x_{i j_2}$};
\end{tikzpicture},$$
and for each $1\leq t\leq n$ with $t\neq j_1,j_2$, denote by $\mathscr{F}_{it}$ the fundamental groupoid of quiver
$$\begin{tikzpicture}
\node at(-0.2,-0.2) (a) {};
\draw[->] (a.west) to[out=150,in=135] ([xshift=-0.8cm,yshift=-0.3cm]a.east) to[out=-45,in=-135] (a.south);
\node at(-0,0.1) (a) {};
\draw[->] (a.east) to[out=60,in=0] ([xshift=0cm,yshift=0.8cm]a.south) to[out=180,in=120] (a.west);
\node at(0.2,-0.2) (a) {};
\draw[->] (a.south) to[out=-90,in=-135] ([xshift=0.7cm,yshift=-0.7cm]a.north) to[out=45,in=0] (a.east);
\draw[dotted] (-0.35,-0.6) arc (230:290:0.7);
\fill (0,0) circle (0.5ex);
\node at(1.3,-0.7) {$\gamma_{it1}$};
\node at(-1.2,-0.5) {$\gamma_{itl_t}$};
\node at(0,0.9) {$\alpha_{it}$};
\node at(0,-0.25) {$x_{it}$};
\end{tikzpicture}.$$
By Lemma \ref{a calculation of modified fundamental groupoid} and Lemma \ref{isomorphism of modified fundamental groupoids}, $\Pi_{m}(E_i,A)$ is isomorphic to the groupoid \\ $\Sigma_i=\bigsqcup_{1\leq t\leq n, t\neq j_1,j_2}(\mathscr{F}_{it}/\langle\alpha_{it}^{d_t}\gamma_{itp}=\gamma_{itp}\alpha_{it}^{d_t}$, $\gamma_{itp}^{2}=1$ $(1\leq p\leq l_t) \rangle)\sqcup (\mathscr{F}_{i}/\langle \beta_i\alpha_{i j_1}^{d_{j_1}}=\alpha_{i j_2}^{d_{j_2}}\beta_i$, $\alpha_{i j_q}^{d_{j_q}}\gamma_{ij_{q}p}=\gamma_{ij_{q}p}\alpha_{i j_q}^{d_{j_q}}$, $\gamma_{ij_{q}p}^{2}=1$ $(q=1,2,$ $1\leq p\leq l_{j_q})\rangle)$.

For each $1\leq j\leq n$, let $\mathscr{G}_j$ be the fundamental groupoid of quiver
$$\begin{tikzpicture}
\node at(-0.2,-0.2) (a) {};
\draw[->] (a.west) to[out=150,in=135] ([xshift=-0.8cm,yshift=-0.3cm]a.east) to[out=-45,in=-135] (a.south);
\node at(-0,0.1) (a) {};
\draw[->] (a.east) to[out=60,in=0] ([xshift=0cm,yshift=0.8cm]a.south) to[out=180,in=120] (a.west);
\node at(0.2,-0.2) (a) {};
\draw[->] (a.south) to[out=-90,in=-135] ([xshift=0.7cm,yshift=-0.7cm]a.north) to[out=45,in=0] (a.east);
\draw[dotted] (-0.35,-0.6) arc (230:290:0.7);
\fill (0,0) circle (0.5ex);
\node at(1.3,-0.7) {$\gamma_{j1}$};
\node at(-1.2,-0.5) {$\gamma_{jl_j}$};
\node at(0,0.9) {$\alpha_{j}$};
\node at(0,-0.25) {$x_{j}$};
\end{tikzpicture}.$$
It is straightforward to show $\Pi_{m}(E',A)$ is isomorphic to $\Sigma'=\bigsqcup_{1\leq j\leq n}\mathscr{G}_j/\langle \alpha_{j}^{d_j}\gamma_{jp}=\gamma_{jp}\alpha_{j}^{d_j}$, $\gamma_{jp}^{2}=1$ $(1\leq p\leq l_j) \rangle$. The direct system $\{\Pi_{m}(E',A)\rightarrow\Pi_{m}(E_i,A)\}_{1\leq i\leq k}$ is isomorphic to the direct system $\{\mu_i:\Sigma'\rightarrow\Sigma_i\}_{1\leq i\leq k}$, where $\mu_i$ is defined by $\mu_i(x_j)=x_{ij}$, $\mu_i(\alpha_j)=\alpha_{ij}$, $\mu_i(\gamma_{jp})=\gamma_{ijp}$ for $1\leq j\leq n$, $1\leq p\leq l_j$.

Define a quiver $Q$ as follows: $Q_0=\{v_1,\cdots,v_n\}$, $Q_1=\{\alpha'_j,\beta'_i,\gamma'_{jp}\mid 1\leq j\leq n$, $1\leq i\leq k$, $1\leq p\leq l_j\}$; define $s(\alpha'_j)=t(\alpha'_j)=s(\gamma'_{jp})=t(\gamma'_{jp})=v_j$ for $1\leq j\leq n$ and $1\leq p\leq l_j$; for $1\leq i\leq k$, if the two half-edges $e_i,\tau(e_i)$ belong to the same vertex $v_j$, define $s(\beta'_i)=t(\beta'_i)=v_j$; if the two half-edges $e_i,\tau(e_i)$ belong to two different vertices $v_{j_1}$, $v_{j_2}$ with $j_1<j_2$, define $s(\beta'_i)=v_{j_1}$ and $t(\beta'_i)=v_{j_2}$. Let $\Pi(Q)$ be the fundamental groupoid of the quiver $Q$, and let $\Sigma$ be the groupoid $\Pi(Q)/\langle(\alpha'_{q(i)})^{d_{q(i)}}\beta'_i=\beta'_i(\alpha'_{p(i)})^{d_{p(i)}}$, $(\alpha'_{j})^{d_{j}}\gamma'_{jp}=\gamma'_{jp}(\alpha'_{j})^{d_{j}}$, $(\gamma'_{jp})^{2}=1\mid 1\leq i\leq k$, $1\leq j\leq n$, $1\leq p\leq l_j \rangle$, where $s(\beta'_i)=v_{p(i)}$ and $t(\beta'_i)=v_{q(i)}$. It can be shown that $\Sigma$ is the direct limit of the direct system $\{\mu_i:\Sigma'\rightarrow\Sigma_i\}_{1\leq i\leq k}$. By Proposition \ref{modified Van-Kampen}, the groupoid $\Pi_{m}(E,A)$ is isomorphic to $\Sigma$. The rest of the proof is similar to that of \cite[Proposition 6.9]{LL2}.
\end{proof}

\subsection{Fundamental groups of the defining $f_{ms}$-BGs of representation-finite and domestic $f_{ms}$-BGAs}
\

In this subsection, we assume that the field $k$ is algebraically closed. For a Brauer $G$-set $E$, we always denote by $\sigma$ the Nakayama automorphism of $E$. Moreover, for $e\in E$, we denote by $e^{\langle \sigma\rangle}$ the $\langle \sigma\rangle$-orbit of $e$.

\begin{Lem}\label{free}
Let $E=(E,U,\tau,d)$ be a connected Brauer $G$-set and let $\Pi=\langle \sigma\rangle\leq \mathrm{Aut}(E)$. Then the action of $\Pi$ on $E$ is free, that is, $\phi(e)\neq e$ for each $e\in E$ and each $\phi\neq 1$ in $\Pi$.
\end{Lem}

\begin{proof}
Suppose that $\sigma^{n}(e)=e$ for some $e\in E$ and for some integer $n$. For any $h\in E$, since $E$ is connected, we can choose a walk $w$ of $E$ from $e$ to $h$. Let $w=(h|\delta_r\cdots\delta_2\delta_1|e)$, where $\delta_i\in\{g,g^{-1},\tau\}$ for $1\leq i\leq r$. Then $\sigma^{n}(h)=\sigma^{n}(\delta_r\cdots\delta_2\delta_1(e))=\delta_r\cdots\delta_2\delta_1(\sigma^{n}(e))=\delta_r\cdots\delta_2\delta_1(e)=h$. Therefore $\sigma^{n}=1$.
\end{proof}

\begin{Lem}\label{odd degree}
Let $E=(E,E,\tau,d)$ be a finite connected $f_{ms}$-BG such that $\langle \sigma\rangle\leq \mathrm{Aut}(E)$ is not admissible, and let $e$ be a half-edge of $E$ such that $\tau(e)\in e^{\langle \sigma\rangle}$. Let $r$ be the smallest positive integer with the property that $\sigma^{r}(e)=\tau(e)$. We have
\begin{itemize}
\item [$(1)$] The $\langle\sigma\rangle$-orbit $e^{\langle \sigma\rangle}$ of $e$ contains $2r$ half-edges $e,\sigma(e),\cdots,\sigma^{2r-1}(e)$.
\item [$(2)$] The order of the Nakayama automorphism $\sigma$ of $E$ is $2r$.
\item [$(3)$] If $N$ is the smallest positive integer with the property that $g^{N}\cdot e\in e^{\langle \sigma\rangle}$ and suppose that $g^{N}\cdot e=\sigma^{p}(e)$ for some $0\leq p\leq 2r-1$, then $(p,2r)=1$. Especially, $p$ is odd.
\item [$(4)$] The f-degree of each vertex of $E/\langle \sigma\rangle$ is odd.
\end{itemize}
\end{Lem}

\begin{proof}
Since $\sigma^{i}(\tau(e))=\tau(\sigma^{i}(e))$ for any integer $i$, $r$ is also the smallest positive integer with the property that $\sigma^{r}(\tau(e))=e$. If there exists $0<i<2r$ such that $\sigma^{i}(e)=e$, then $0<i<r$ or $r<i<2r$. If $0<i<r$, then $\sigma^{r-i}(e)=\sigma^{r-i}(\sigma^{i}(e))=\sigma^r(e)=\tau(e)$, which contradicts the minimality of $r$. If $r<i<2r$, then $\sigma^{i-r}(\tau(e))=\sigma^{i-r}(\sigma^{r}(e))=\sigma^i(e)=e$, where $0<i-r<r$, which also contradicts the minimality of $r$. Therefore $2r$ is the minimal positive integer with the property that $\sigma^{2r}(e)=e$, so $(1)$ holds, and $(2)$ follows from Lemma \ref{free}.

Note that $N$ divides $d(e)$. Otherwise, let $d(e)=aN+b$ with $a$, $b\in\mathbb{Z}$ and $0<b<N$. We have $\sigma(e)=g^{d(e)}\cdot e=g^{aN+b}\cdot e=g^{b}\cdot \sigma^{ap}(e)$ and $g^{b}\cdot e=\sigma^{1-ap}(e)\in e^{\langle \sigma\rangle}$, contradicts the minimality of $N$. Let $d(e)=aN$ for some integer $a$. Then $\sigma(e)=g^{d(e)}\cdot e=g^{aN}\cdot e=\sigma^{ap}(e)$. Since $2r$ is the minimal positive integer with the property that $\sigma^{2r}(e)=e$, $2r$ divides $1-ap$ and $(p,2r)=1$, which implies $(3)$. Since $2r$ divides $1-ap$, $a$ is odd. Since $N$ is equal to the cardinality of the vertex of $E/\langle \sigma\rangle$ containing $e^{\langle \sigma\rangle}$, $a$ is equal to the f-degree of the vertex of $E/\langle \sigma\rangle$ containing $e^{\langle \sigma\rangle}$. For every half-edge $h$ of $E$, since $\langle \sigma\rangle$ acts freely on $E$ and since the order of $\sigma$ is $2r$, $h^{\langle \sigma\rangle}$ contains $2r$ elements. Using the same method we can show that the f-degree of the vertex of $E/\langle \sigma\rangle$ which contains $h^{\langle \sigma\rangle}$ is odd, which implies $(4)$.
\end{proof}

\subsubsection{Fundamental groups of the defining $f_{ms}$-BGs of representation-finite $f_{ms}$-BGAs}
\

We first characterize representation-finite $f_{ms}$-BGA $A_E$ in terms of the quotient $E/\langle\sigma\rangle$.

\begin{Lem}\label{B-finite-type}
Let $E=(E,E,\tau,d)$ be a finite connected $f_{ms}$-BG and let $B=E/\langle\sigma\rangle$. Suppose that the modified BG $B$ has $k$ edges, $l$ double half-edges, and $n$ vertices $v_1$, $\cdots$, $v_n$ of f-degree $d_1$, $\cdots$, $d_n$ respectively.  Then $A_E$ is representation-finite if and only if one of the following holds
\begin{itemize}
\item[$(1)$] $l=1$, $k-n+1=0$, $d_i=1$ for $1\leq i\leq n$;
\item[$(2)$] $l=0$, $k-n+1=0$, $d_i=1$ for all but at most one $1\leq i\leq n$ (that is, $B$ is a Brauer tree).
\end{itemize}
\end{Lem}

\begin{proof}
By Theorem \ref{rep type of f-BCA and its reduced form}, we have that $A_E$ is representation-finite if and only if $R_E$ is a Brauer tree.

"$\Rightarrow$" If the group $\langle\sigma\rangle$ of automorphisms of $E$ is admissible, then $E/\langle\sigma\rangle=R_E$ is a Brauer tree. Otherwise, $E/\langle \sigma\rangle$ contains double half-edges and $R_E=\widehat{E/\langle \sigma\rangle}$. Then $R_E$ has $2n$ vertices and $2k+l$ edges. Since $R_E$ is a Brauer tree, $(2k+l)-2n+1=0$, and $E/\langle \sigma\rangle$ is f-degree-free (otherwise, the Brauer tree $R_E$ will have at least two vertices with f-degree larger than $1$). Since $E/\langle \sigma\rangle$ is connected, $k\geq n-1$. Therefore $k=n-1$ and $l=1$.

"$\Leftarrow$" Since in both cases $R_E$ is a Brauer tree, $A_E$ is representation-finite.
\end{proof}

\begin{Thm}\label{fundamental group of rep-finite f-BGA}
Let $E=(E,E,\tau,d)$ be a finite connected $f_{ms}$-BG. Then $A_E$ is representation-finite if and only if $\Pi(E)\cong\mathbb{Z}$.
\end{Thm}

\begin{proof}
"$\Rightarrow$" Suppose that $B=E/\langle \sigma\rangle$ has $n$ vertices, $k$-edges, and $l$ double half-edges.

If $B$ belongs to case $(1)$ of Lemma \ref{B-finite-type}, by Proposition \ref{modified fundamental group of modified f-BG}, $\Pi_{m}(E/\langle \sigma\rangle)\cong F\langle a_1,\cdots,a_n,c_1\rangle/\langle a_1=\cdots=a_n$, $a_1 c_1=c_1 a_1$, $c_{1}^{2}=1\rangle\cong F\langle a,c\rangle/\langle ac=ca$, $c^2=1\rangle\cong\mathbb{Z}\oplus\mathbb{Z}/2\mathbb{Z}$. Let $e^{\langle \sigma\rangle}$ be the unique double half-edge of $E/\langle \sigma\rangle$. Since $\tau(e)\in e^{\langle \sigma\rangle}$, there exists a minimal positive integer $r$ such that $\tau(e)=\sigma^{r}(e)$. By Lemma \ref{odd degree} $(1)$, $e^{\langle \sigma\rangle}=\{\sigma^{i}(e)\mid 0\leq i\leq 2r-1\}$. By the proof of Proposition \ref{modified fundamental group of modified f-BG}, the isomorphism $f:\mathbb{Z}\oplus\mathbb{Z}/2\mathbb{Z}\xrightarrow{\sim}\Pi_{m}(E/\langle \sigma\rangle,e^{\langle \sigma\rangle})$ is given by $f(1,0)=[(e^{\langle \sigma\rangle}|g^{d(e)}|e^{\langle \sigma\rangle})]$ and $f(0,\overline{1})=[(e^{\langle \sigma\rangle}|\tau|e^{\langle \sigma\rangle})]$. According to Remark \ref{general-method-to-compute-the-fundamental-group}, the covering $\phi:E\rightarrow E/\langle\sigma\rangle$ of Brauer $G$-sets induces a $\Pi_{m}(E/\langle \sigma\rangle,e^{\langle \sigma\rangle})$-set structure on $\phi^{-1}(e^{\langle \sigma\rangle})=e^{\langle \sigma\rangle}$, and the stabilizer subgroup of $e$ in $\Pi_{m}(E/\langle \sigma\rangle,e^{\langle \sigma\rangle})$ is isomorphic to $\Pi_{m}(E,e)$, which is also isomorphic to $\Pi(E,e)$ by Lemma \ref{two definitions of fundamental group are equal}. The action of $\mathbb{Z}\oplus\mathbb{Z}/2\mathbb{Z}$ on $e^{\langle \sigma\rangle}$ via the group isomorphism $f$ is given by $(a,\overline{0})\cdot e=\sigma^{a}(e)$ and $(a,\overline{1})\cdot e=\sigma^{a+r}(e)$ for any $a\in\mathbb{Z}$. We have $(a,\overline{0})\cdot e=e$ if and only if $2r|a$, and $(a,\overline{1})\cdot e=e$ if and only if $a=br$ with $b$ odd. Therefore $\Pi(E,e)$ is isomorphic to the subgroup of $\mathbb{Z}\oplus\mathbb{Z}/2\mathbb{Z}$ generated by $(r,\overline{1})$, which is isomorphic to $\mathbb{Z}$.

If $B$ belongs to case $(2)$ of Lemma \ref{B-finite-type}, then $B$ is a Brauer tree, and by \cite[Proposition 6.9]{LL2} we have $\Pi(B)\cong\mathbb{Z}$. Since there exists a covering $E\rightarrow B$ of $f_{ms}$-BGs, by \cite[Theorem 3.17]{LL2}, $\Pi(E)$ is isomorphic to a subgroup of $\Pi(B)$. Since $E$ is finite, the order of the automorphism $\sigma$ of $E$ is finite, and suppose that $o(\sigma)=r$. For any $e\in E$, by \cite[Proposition 3.46]{LL2}, $\overline{(e|g^{d(e)r}|e)}\neq \overline{(e||e)}$ in $\Pi(E)$, therefore $\Pi(E)\neq \{1\}$. Since $\Pi(E)$ is a subgroup of $\Pi(B)$, it follows that $\Pi(E)\cong\mathbb{Z}$.

"$\Leftarrow$" Suppose that $A_E$ is not representation-finite, then $A_E$ has a band. Since $A_E$ is not a Nakayama algebra, there exists some $e\in E$ with $d(e)>1$. According to the proof of Lemma \ref{two-numbers-of-bands-are-equal}, $E$ has a band $l$, which corresponds to a closed walk $w$ of $E$ of the form
$$(e_{2k}|\tau g^{-i_{2k}}|e_{2k-1})(e_{2k-1}|\tau g^{i_{2k-1}}|e_{2k-2})\cdots(e_2|\tau g^{-i_2}|e_1)(e_1|\tau g^{i_1}|e_0),$$
where $e_0=e_{2k}$ and $0<i_{j}<d(e_{j-1})$ for $1\leq j\leq 2k$. Using $w$ we obtain a special walk
\begin{multline*}w'=(\sigma^{s}(e_{2k})|\tau g^{d(e_{2k-1})-i_{2k}}|\sigma^{k-1}(e_{2k-1}))(\sigma^{k-1}(e_{2k-1})|\tau g^{i_{2k-1}}|\sigma^{k-1}(e_{2k-2}))\cdots \\ (\sigma(e_2)|\tau g^{d(e_1)-i_2}|e_1)(e_1|\tau g^{i_1}|e_0)\end{multline*}
of $E$. Let $e=e_0=e_{2k}$. Since $E$ is finite, the order of the automorphism $\sigma$ of $E$ is finite. Suppose that $\sigma^{r}=id_{E}$, then the walk $v=\sigma^{k(r-1)}(w')\cdots\sigma^{k}(w')w'$ is a closed special walk of $E$ at $e$. Moreover, $v^{l}$ is also a closed special walk of $E$ at $e$ for any positive integer $l$. Let $u=(e|g^{rd(e)}|e)$ be a closed walk of $E$ at $e$. By the definition of $\Pi(E,e)$, it is easy to see that $\overline{uv}=\overline{vu}$ in $\Pi(E,e)$. By \cite[Proposition 3.46]{LL2}, the subgroup of $\Pi(E,e)$ generated by $\overline{u}$ and $\overline{v}$ is isomorphic to $\mathbb{Z}\oplus\mathbb{Z}$. Therefore $\Pi(E,e)$ is not isomorphic to $\mathbb{Z}$, a contradiction.
\end{proof}

\subsubsection{Fundamental groups of the defining $f_{ms}$-BGs of domestic $f_{ms}$-BGAs}
\

We first characterize domestic $f_{ms}$-BGA $A_E$ in terms of the quotient $E/\langle\sigma\rangle$.

\begin{Lem}\label{B}
Let $E=(E,E,\tau,d)$ be a finite connected $f_{ms}$-BG and let $B=E/\langle\sigma\rangle$. Suppose that the modified BG $B$ has $k$ edges, $l$ double half-edges, and $n$ vertices $v_1$, $\cdots$, $v_n$ of f-degree $d_1$, $\cdots$, $d_n$ respectively.  Then $A_E$ is domestic if and only if one of the following holds
\begin{itemize}
\item[$(1)$] $l=2$, $k-n+1=0$, $d_i=1$ for $1\leq i\leq n$;
\item[$(2)$] $l=0$, $k-n+1=0$, $d_i=2$ for exactly two numbers $i=i_0$, $i_1$, and $d_i=1$ for $i\neq i_0$, $i_1$;
\item[$(3)$] $l=0$, $k-n+1=1$, $d_i=1$ for $1\leq i\leq n$.
\end{itemize}
\end{Lem}

\begin{proof}
"$\Rightarrow$" If the subgroup $\langle\sigma\rangle$ of $\mathrm{Aut}(E)$ is not admissible, by Lemma \ref{odd degree}, the f-degree of each vertex of $B$ is odd. Moreover, $R_E=\widehat{B}$ is a Brauer graph whose first Betti number $\chi(R_E)$ is $2(k-n+1)+l-1$. By Theorem \ref{rep type of f-BCA and its reduced form}, $A_{R_E}$ is domestic. Since the f-degree of each vertex of $B$ is odd, the f-degree of each vertex of $R_E$ is also odd. Then by Theorem \ref{domestic-BGA} we imply that $R_E$ is an f-degree-free BG with $\chi(R_E)=1$. Therefore we have $k-n+1=0$, $l=2$, and $d_i=1$ for $1\leq i\leq n$. Then $(1)$ holds.

If the subgroup $\langle\sigma\rangle$ of $\mathrm{Aut}(E)$ is admissible, then $B=R_E$ is a Brauer graph. By Theorem \ref{rep type of f-BCA and its reduced form}, $A_{R_E}$ is domestic. According to Theorem \ref{domestic-BGA}, either $(2)$ or $(3)$ holds.

"$\Leftarrow$" If $(1)$ holds then $R_E=\widehat{B}$ is a Brauer graph with free f-degree and the underlying graph of $R_E$ contains a unique cycle. According to Theorem \ref{domestic-BGA} and Theorem \ref{rep type of f-BCA and its reduced form}, $A_E$ is domestic. If $(2)$ or $(3)$ holds, then $R_E=B$, and by Theorem \ref{domestic-BGA}, $A_{R_E}$ is domestic. So by Theorem \ref{rep type of f-BCA and its reduced form} $A_E$ is also domestic.
\end{proof}

To calculate the fundamental groups of the defining $f_{ms}$-BGs of domestic $f_{ms}$-BGAs, we need some preparations.

\begin{Lem}\label{generators of normal subgroup}
Let $\Omega$ be a group and let $N$ be the normal subgroup of $\Omega$ generated by $X$, where $X$ is a subset of $\Omega$. Let $H$ be a subgroup of $\Omega$ such that $N\subseteq H$, and $I$ be a subset of $\Omega$ which contains exactly one element from each right coset of $H$. Then $N$ is the normal subgroup of $H$ generated by $Y=\{bab^{-1}\mid b\in I$, $a\in X\}$.
\end{Lem}

\begin{proof}
Since $N$ is the normal subgroup of $\Omega$ generated by $X$, it is the subgroup of $\Omega$ generated by the set $X'=\{\omega a {\omega}^{-1}\mid \omega \in \Omega$, $a\in X\}$. Since each element of $X'$ is conjugate to some element of $Y$ in $H$, $N$ is the normal subgroup of $H$ generated by $Y$.
\end{proof}

\begin{Def}\label{Schreier system} {\rm(\cite[Chapter 6, Section 8]{M})}
Let $F$ be a free group on the set $S$. For any element $a\neq 1$ of $F$, express $a$ as a reduced word in the generators: $a=x_1 x_2 \cdots x_k$, where $x_i\in S$ or $x_{i}^{-1}\in S$ for $1\leq i\leq k$. Define $a'=x_1 x_2 \cdots x_{k-1}$. A nonempty subset $A$ of $F$ is said to be a Schreier system in $F$ if $a'\in A$ for any $a\in A$ with $a\neq 1$.
\end{Def}

\begin{Prop}\label{free generators} {\rm(\cite[Chapter 6, Theorem 8.1]{M})}
Let $F$ be a free group on the set $S$, $F'$ be a subgroup of $F$, and $A$ be a Schreier system in $F$ which contains exactly one element from each right coset of $F'$. Then $F'$ is a free group on the set $\{as\Phi(as)^{-1}\mid a\in A$, $s\in S$, $as\Phi(as)^{-1}\neq 1\}$, where the map $\Phi:F\rightarrow A$ assigns each element of $F$ the unique element of $A$ in the same coset of $F'$.
\end{Prop}

\begin{Prop}\label{rep type to fundamental group}
Let $E=(E,E,\tau,d)$ be a finite connected $f_{ms}$-BG such that $A_E$ is domestic. Then $\Pi(E)\cong F\langle a,b\rangle/\langle a^2=b^2\rangle$ or $\Pi(E)\cong\mathbb{Z}\oplus\mathbb{Z}$.
\end{Prop}

\begin{proof}
Suppose that the modified BG $B=E/\langle\sigma\rangle$ has $k$-edges, $l$ double half-edges, and $n$ vertices $v_1$, $\cdots$, $v_n$ of f-degree $d_1$, $\cdots$, $d_n$ respectively. Since $A_E$ is domestic, one of the conditions $(1)$, $(2)$, $(3)$ in Lemma \ref{B} holds.

Suppose that condition $(1)$ in Lemma \ref{B} holds, that is, $l=2$, $k-n+1=0$ and $d_i=1$ for $1\leq i\leq n$. Let $B=(B,B,\tau',d')$, and let $e^{\langle\sigma\rangle}$ and $h^{\langle\sigma\rangle}$ be the two double half-edges of $B$. Since $E$ is connected, so is $B$. By the proof of Proposition \ref{modified fundamental group of modified f-BG}, there is a group isomorphism
$$f:F\langle a,c_1,c_2\rangle/\langle a c_1=c_1 a, a c_2=c_2 a, c_{1}^{2}=c_{2}^{2}=1\rangle\xrightarrow{\sim}\Pi_{m}(B,e^{\langle\sigma\rangle})$$
such that $f(\overline{a})=[(e^{\langle\sigma\rangle}|g^{d(e)}|e^{\langle\sigma\rangle})]$, $f(\overline{c_1})=[(e^{\langle\sigma\rangle}|\tau'|e^{\langle\sigma\rangle})]$, $f(\overline{c_2})=[(w')^{-1}(h^{\langle\sigma\rangle}|\tau'|h^{\langle\sigma\rangle})w']$, where $w'$ is a  walk of $B$ from $e^{\langle\sigma\rangle}$ to $h^{\langle\sigma\rangle}$. We may assume that $w'$ lifts to a walk $w$ of $E$ from $e$ to $h$.

Let $r$ be the minimal positive integer with the property that $\sigma^{r}(e)=\tau(e)$. Then by Lemma \ref{odd degree} $(1)$ and $(2)$, $|e^{\langle\sigma\rangle}|=2r$ and the order of the Nakayama automorphism $\sigma$ of $E$ is $2r$. By Lemma \ref{free}, we also have $|h^{\langle\sigma\rangle}|=2r$. Then by Lemma \ref{odd degree} $(1)$, $r$ is also the minimal positive integer with the property that $\sigma^{r}(h)=\tau(h)$. The covering of Brauer $G$-sets $\phi:E\rightarrow B$ induces a $\Pi_{m}(B,e^{\langle\sigma\rangle})$-set structure on $\phi^{-1}(e^{\langle\sigma\rangle})=e^{\langle\sigma\rangle}$, and $\Pi(E,e)\cong\Pi_{m}(E,e)$ is isomorphic to the stabilizer subgroup of $e$ in $\Pi_{m}(B,e^{\langle\sigma\rangle})$ by Remark \ref{general-method-to-compute-the-fundamental-group}. Using the group isomorphism $f$ we may view $e^{\langle\sigma\rangle}$ as an $F\langle a,c_1,c_2\rangle/\langle a c_1=c_1 a$, $a c_2=c_2 a$, $c_{1}^{2}=c_{2}^{2}=1\rangle$-set. Since
\begin{multline*}[(w')^{-1}(h^{\langle\sigma\rangle}|\tau'|h^{\langle\sigma\rangle}) w']\cdot e=[(w')^{-1}(h^{\langle\sigma\rangle}|\tau'|h^{\langle\sigma\rangle})]\cdot h \\ =[(w')^{-1}]\cdot\tau(h)=[(w')^{-1}]\cdot\sigma^{r}(h)=\sigma^{r}([(w')^{-1}]\cdot h)=\sigma^{r}(e),\end{multline*}
we have $\overline{c_2}\cdot e=\sigma^{r}(e)$. Moreover, we have $\overline{a}\cdot e=\sigma(e)$ and $\overline{c_1}\cdot e=\sigma^{r}(e)$. For each $1\leq i\leq 2r$, identify the element $\sigma^{i}(e)$ of $e^{\langle\sigma\rangle}$ with the integer $i$, and let
$$\rho:F\langle a,c_1,c_2\rangle/\langle a c_1=c_1 a, a c_2=c_2 a, c_{1}^{2}=c_{2}^{2}=1\rangle\rightarrow S_{2r}$$
be the group homomorphism given by the action of $F\langle a,c_1,c_2\rangle/\langle a c_1=c_1 a$, $a c_2=c_2 a$, $c_{1}^{2}=c_{2}^{2}=1\rangle$ on $e^{\langle\sigma\rangle}$. Then $\rho(\overline{a})=(1$ $2\cdots 2r)$ and $\rho(\overline {c_1})=\rho(\overline{c_2})=(1$ $r+1)(2$ $r+2)\cdots (r$ $2r)$.

Let $\widetilde{\rho}:F\langle a,c_1,c_2\rangle\rightarrow S_{2r}$ be the group homomorphism induced by $\rho$, and let $H=\{x\in F\langle a,c_1,c_2\rangle\mid \widetilde{\rho}(x)(1)=1\}$ be a subgroup of $F\langle a,c_1,c_2\rangle$. Then $\Pi(E,e)$ is isomorphic to $H/N$, where $N$ is the normal subgroup of $F\langle a,c_1,c_2\rangle$ generated by $c_{1}^{2}$, $c_{2}^{2}$, $c_1 a c_{1}^{-1} a^{-1}$, $c_2 a c_{2}^{-1} a^{-1}$. Let $A=\{1,a,a^2,\cdots,a^{2r-1}\}$ be a subset of $F\langle a,c_1,c_2\rangle$. Then $A$ is a Schreier system in $F\langle a,c_1,c_2\rangle$ which contains exactly one element from each right coset of $H$. By Proposition \ref{free generators}, $H$ is a free group on the set $\{x_i,y_i,z\mid 0\leq i\leq 2r-1\}$, where
\begin{equation*}
x_i= \begin{cases}
a^{i}c_1 a^{-r-i}, & \text{if } 0\leq i\leq r-1; \\
a^{i}c_1 a^{r-i}, & \text{if } r\leq i\leq 2r-1;
\end{cases}
\end{equation*}
\begin{equation*}
y_i= \begin{cases}
a^{i}c_2 a^{-r-i}, & \text{if } 0\leq i\leq r-1; \\
a^{i}c_2 a^{r-i}, & \text{if } r\leq i\leq 2r-1;
\end{cases}
\end{equation*}
and $z=a^{2r}$. Since $N$ is the normal subgroup of $F\langle a,c_1,c_2\rangle$ generated by $c_{1}^{2}$, $c_{2}^{2}$, $c_1 a c_{1}^{-1} a^{-1}$, $c_2 a c_{2}^{-1} a^{-1}$, by Lemma \ref{generators of normal subgroup}, $N$ is the normal subgroup of $H$ generated by the set $\{a^{i} c_{j}^{2} a^{-i}$, $a^{i} c_j a c_{j}^{-1} a^{-i-1}\mid 0\leq i\leq 2r-1$, $j=1,2\}$. A calculation shows that
\begin{equation*}
a^i c_{1}^{2} a^{-i}= \begin{cases}
x_i x_{r+i}, & \text{if } 0\leq i\leq r-1; \\
x_i x_{i-r}, & \text{if } r\leq i\leq 2r-1;
\end{cases}
\end{equation*}
\begin{equation*}
a^i c_{2}^{2} a^{-i}= \begin{cases}
y_i y_{r+i}, & \text{if } 0\leq i\leq r-1; \\
y_i y_{i-r}, & \text{if } r\leq i\leq 2r-1;
\end{cases}
\end{equation*}
\begin{equation*}
a^{i} c_1 a c_{1}^{-1} a^{-i-1}= \begin{cases}
x_i x_{i+1}^{-1}, & \text{if } 0\leq i\leq r-2 \text{ or } r\leq i\leq 2r-2; \\
x_{r-1} z x_{r}^{-1}, & \text{if } i=r-1; \\
x_{2r-1} x_{0}^{-1} z^{-1}, & \text{if } i=2r-1;
\end{cases}
\end{equation*}
\begin{equation*}
a^{i} c_2 a c_{2}^{-1} a^{-i-1}= \begin{cases}
y_i y_{i+1}^{-1}, & \text{if } 0\leq i\leq r-2 \text{ or } r\leq i\leq 2r-2; \\
y_{r-1} z y_{r}^{-1}, & \text{if } i=r-1; \\
y_{2r-1} y_{0}^{-1} z^{-1}, & \text{if } i=2r-1.
\end{cases}
\end{equation*}
Therefore \begin{multline*}\Pi(E,e)\cong H/N=F\langle x_0,\cdots,x_{2r-1},y_0,\cdots,y_{2r-1},z\rangle/\langle x_0=\cdots=x_{r-1}=x_{r}^{-1}=\cdots \\ =x_{2r-1}^{-1}$,  $y_0=\cdots=y_{r-1}=y_{r}^{-1}=\cdots=y_{2r-1}^{-1}$, $z=x_{0}^{-2}=y_{0}^{-2}\rangle\cong F\langle x_0,y_0\rangle/\langle x_{0}^{2}=y_{0}^{2}\rangle\end{multline*}

Suppose that condition $(2)$ in Lemma \ref{B} holds, that is, $l=0$, $k-n+1=0$, $d_i=2$ for exactly two numbers $i=i_0$, $i_1$, and $d_i=1$ for $i\neq i_0$, $i_1$. Suppose that $e$, $h$ are two half-edges of $E$ with $e^{\langle\sigma\rangle}\in v_{i_0}$ and $h^{\langle\sigma\rangle}\in v_{i_1}$. By the proof of \cite[Proposition 6.9]{LL2}, there is a group isomorphism
$$f:F\langle a,b\rangle/\langle a^2=b^2\rangle \xrightarrow{\sim}\Pi(B,e^{\langle\sigma\rangle})$$
such that
$$f(\overline{a})=\overline{(e^{\langle\sigma\rangle}|g^{\frac{d(e)}{2}}|e^{\langle\sigma\rangle})} \text{ and } f(\overline{b})=\overline{(w')^{-1}(h^{\langle\sigma\rangle}|g^{\frac{d(h)}{2}}|h^{\langle\sigma\rangle})w'},$$
where $w'$ is a walk of $B$ from $e^{\langle\sigma\rangle}$ to $h^{\langle\sigma\rangle}$. We may assume that $w'$ lifts to a walk $w$ of $E$ from $e$ to $h$. Suppose that the $G$-orbit of $e$ contains $M$ half-edges, then $v_{i_0}$ contains $N=(M,d(e))$ half-edges. Since $d_{i_0}=2$, $d(e)=2N$. Since $1=(\frac{M}{N},\frac{d(e)}{N})=(\frac{M}{N},2)$, $\frac{M}{N}$ is odd. So $2N$ divides $M+N$. We have
$$\sigma^{\frac{M+N}{2N}}(e)=g^{\frac{d(e)(M+N)}{2N}}\cdot e=g^{M+N}\cdot e=g^{N}\cdot e.$$
Similarly, suppose that the $G$-orbit of $h$ contains $M'$ half-edges, then $v_{i_1}$ contains $N'=(M',d(h))$ half-edges. Moreover, $2N'$ divides $M'+N'$ and $\sigma^{\frac{M'+N'}{2N'}}(h)=g^{N'}\cdot h$. Since $\frac{M}{N}$ (resp. $\frac{M'}{N'}$) is the minimal positive integer $i$ (resp. $j$) such that $\sigma^{i}(e)=e$ (resp. $\sigma^{j}(h)=h$), by Lemma \ref{free}, $\frac{M}{N}=o(\sigma)=\frac{M'}{N'}$.

The covering of $f_{ms}$-BGs $\phi:E\rightarrow B$ induces a $\Pi(B,e^{\langle\sigma\rangle})$-set structure on $\phi^{-1}(e^{\langle\sigma\rangle})=e^{\langle\sigma\rangle}$, and $\Pi(E,e)$ is isomorphic to the stabilizer subgroup of $e$ in $\Pi(B,e^{\langle\sigma\rangle})$.
Denote by $o(\sigma)=r$, and identify the element $g^{iN}\cdot e$ of $e^{\langle\sigma\rangle}$ with the integer $i$ for any $1\leq i\leq r$. Then $e^{\langle\sigma\rangle}=\{1,2,\cdots,r\}$ becomes an $F\langle a,b\rangle/\langle a^2=b^2\rangle$-set via the isomorphism $f:F\langle a,b\rangle/\langle a^2=b^2\rangle \xrightarrow{\sim}\Pi(B,e^{\langle\sigma\rangle})$. Since
$$a\cdot (g^{iN}\cdot e)=\overline{(e^{\langle\sigma\rangle}|g^{\frac{d(e)}{2}}|e^{\langle\sigma\rangle})}\cdot (g^{iN}\cdot e)=g^{\frac{d(e)}{2}}g^{iN}\cdot e=g^{N}g^{iN}\cdot e=g^{(i+1)N}\cdot e,$$
the action of $a$ on $e^{\langle\sigma\rangle}=\{1,2,\cdots,r\}$ corresponds to the permutation $(12\cdots r)$. Suppose that the walk $w$ of $E$ from $e$ to $h$ is of the form $w=(h|\delta_{s}\cdots\delta_1|e)$ with $\delta_i\in\{g,g^{-1},\tau\}$. Let
\begin{equation*}
\delta_i^{-1}:= \begin{cases}
g^{-1}, & \text{if } \delta_i=g; \\
g, & \text{if } \delta_i=g^{-1}; \\
\tau, & \text{if } \delta_i=\tau.
\end{cases}
\end{equation*}
Since
\begin{multline*}b\cdot e=\overline{(w')^{-1}(h^{\langle\sigma\rangle}|g^{\frac{d(h)}{2}}|h^{\langle\sigma\rangle})w'}\cdot e
=\delta_{1}^{-1}\cdots\delta_{s}^{-1}g^{\frac{d(h)}{2}}\delta_{s}\cdots\delta_1 (e)=\delta_{1}^{-1}\cdots\delta_{s}^{-1}g^{N'}\delta_{s}\cdots\delta_1 (e) \\ =\delta_{1}^{-1}\cdots\delta_{s}^{-1}g^{N'}\cdot h=\delta_{1}^{-1}\cdots\delta_{s}^{-1}\sigma^{\frac{M'+N'}{2N'}}(h)
=\delta_{1}^{-1}\cdots\delta_{s}^{-1}\sigma^{\frac{M+N}{2N}}(h) \\ =\sigma^{\frac{M+N}{2N}}\delta_{1}^{-1}\cdots\delta_{s}^{-1}(h)
=\sigma^{\frac{M+N}{2N}}(e)=g^{N}\cdot e,\end{multline*}
and since the action of $F\langle a,b\rangle/\langle a^2=b^2\rangle$ on $e^{\langle\sigma\rangle}$ commutes with the action of $\langle\sigma\rangle$ on $e^{\langle\sigma\rangle}$, the action of $b$ on $e^{\langle\sigma\rangle}=\{1,2,\cdots,r\}$ also corresponds to the permutation $(12\cdots r)$.

The action of $F\langle a,b\rangle/\langle a^2=b^2\rangle$ on $e^{\langle\sigma\rangle}=\{1,2,\cdots,r\}$ defines a group homomorphism $\rho:F\langle a,b\rangle/\langle a^2=b^2\rangle\rightarrow S_{r}$, and let $\widetilde{\rho}:F\langle a,b\rangle\rightarrow S_r$ be the homomorphism induced by $\rho$. Let
$$H=\{x\in F\langle a,b\rangle\mid\widetilde{\rho}(x)(r)=r\}$$
be a subgroup of $F\langle a,b\rangle$. Then $\Pi(E,e)$ is isomorphic to $H/K$, where $K$ is the normal subgroup of $F\langle a,b\rangle$ generated by $a^2 b^{-2}$. Since $A=\{1,a,\cdots,a^{r-1}\}$ is a Schreier system in $F\langle a,b\rangle$ which consists exactly one element from each right coset of $H$, by Proposition \ref{free generators}, $H$ is a free group on the set $\{x_i\mid 0\leq i\leq r\}$, where
\begin{equation*}
x_i= \begin{cases}
a^i b a^{-i-1}, & \text{if } 0\leq i\leq r-2; \\
a^{r-1} b, & \text{if } i=r-1; \\
a^r, & \text{if } i=r.
\end{cases}
\end{equation*}
By Lemma \ref{generators of normal subgroup}, $K$ is the normal subgroup of $H$ generated by $\{a^{i+2}b^{-2}a^{-i}\mid 0\leq i\leq r-1\}$. A calculation shows that
\begin{equation*}
a^{i+2}b^{-2}a^{-i}= \begin{cases}
x_{i+1}^{-1}x_{i}^{-1}, & \text{if } 0\leq i\leq r-3; \\
x_r x_{r-1}^{-1} x_{r-2}^{-1}, & \text{if } i=r-2; \\
x_r x_{0}^{-1} x_{r-1}^{-1}, & \text{if } i=r-1.
\end{cases}
\end{equation*}
Since $r=\frac{M}{N}$ is odd, we have $x_0=x_{1}^{-1}=x_2=\cdots=x_{r-3}=x_{r-2}^{-1}$ and $x_r=x_{r-2}x_{r-1}=x_{r-1}x_0$ in $H/K$. Therefore $H/K\cong F\langle x_0,x_{r-1}\rangle/\langle x_{0}^{-1}x_{r-1}=x_{r-1}x_0\rangle$. Let $y=x_{r-1}$ and $z=x_{r-1}x_0$. Then the relation $x_{0}^{-1}x_{r-1}=x_{r-1}x_0$ is equivalent to the relation $y^2=z^2$, so $H/K\cong F\langle y,z\rangle/\langle y^2=z^2\rangle$.

Suppose that condition $(3)$ in Lemma \ref{B} holds, that is, $l=0$, $k-n+1=1$ and $d_i=1$ for $1\leq i\leq n$. By \cite[Proposition 6.9]{LL2}, we have $\Pi(B)\cong\mathbb{Z}\oplus\mathbb{Z}$. Since $\Pi(E)$ is isomorphic to a nonzero subgroup of $\Pi(B)$, either $\Pi(E)\cong\mathbb{Z}$ or $\Pi(E)\cong\mathbb{Z}\oplus\mathbb{Z}$. Since $A_E$ is not representation-finite, by Theorem \ref{fundamental group of rep-finite f-BGA}, we have $\Pi(E)\cong\mathbb{Z}\oplus\mathbb{Z}$.
\end{proof}

Let $E=(E,U,\tau,d)$ be a finite connected Brauer $G$-set, where the order of the Nakayama automorphism $\sigma$ of $E$ is $r$. Define an equivalence relation $\approx'$ on the set of walks of $E$ as follows: $w\approx' v$ if and only if $w\approx (e|g^{krd(e)}|e)v$ for some integer $k$, where $e$ is the terminal of $v$. Note that if $w_{1}\approx' w_{2}$, then by homotopy relations $(mh2)$ and $(mh3)$ in Definition \ref{homotopy of modified walks}, we have $uw_{1}\approx' uw_{2}$ and $w_{1}v\approx' w_{2}v$ whenever the compositions make sense. So we can define composition on the level of equivalence classes of walks under $\approx'$. Denote by $[[w]]$ the equivalence class of $w$ under the equivalence relation $\approx'$, and denote by $\Pi_{m}'(E,e)$ the group of equivalence classes of closed walks of $E$ at $e$ under $\approx'$, which is called the {\it reduced fundamental group} of $E$ at $e$. Since $E$ is connected, $\Pi_{m}'(E,e)$ is independent of the choice of $e$ up to isomorphism. Therefore we may simply write $\Pi_{m}'(E,e)$ as $\Pi_{m}'(E)$.

Denote by $\pi: \Pi_{m}(E,e)\rightarrow \Pi_{m}'(E,e)$ the natural group homomorphism. Using Proposition \ref{modified unique factorization}, it is straightforward  to show that there is an exact sequence of groups $$0\rightarrow\mathbb{Z}\xrightarrow{i}{}\Pi_{m}(E,e)\xrightarrow{\pi}{}\Pi_{m}'(E,e)\rightarrow 1,$$ where $i(1)=[(e|g^{rd(e)}|e)]$ and $\pi([w])=[[w]]$.

\begin{Lem}\label{injective-reduced-modified-fundamental-group}
A covering $f:E\rightarrow E'$ of finite connected Brauer $G$-sets induces an injective group homomorphism $f_{*}:\Pi_{m}'(E,e)\rightarrow \Pi_{m}'(E',f(e))$.
\end{Lem}

\begin{proof}
Define $f_{*}:\Pi_{m}'(E,e)\rightarrow \Pi_{m}'(E',f(e))$ as $f_{*}([[w]])=[[f(w)]]$. Assume that $E=(E,U,\tau,d)$ and $E'=(E',U',\tau',d')$. Denote by $\sigma$ (resp. $\sigma'$) the Nakayama automorphism of $E$ (resp. $E'$), and suppose that the order of $\sigma$ (resp. $\sigma'$) is $r$ (resp. $r'$). If $w$, $v$ are closed walks of $E$ at $e$ with $w\approx' v$, then $w\approx (e|g^{krd(e)}|e)v$ for some integer $k$. Therefore $f(w)\approx (f(e)|g^{krd(e)}|f(e))f(v)$, where $d(e)=d'(f(e))$. For any $y\in E'$, since $E'$ is connected, there exists $x\in E$ such that $f(x)=y$. Then $(\sigma')^{r}(y)=(\sigma')^{r}(f(x))=f(\sigma^{r}(x))=f(x)=y$. By Lemma \ref{free}, we have $(\sigma')^r=id_{E'}$. So the order of $\sigma'$ divides $r$. Therefore $f(w)\approx' f(v)$ and $f_{*}$ is well-defined. To show that $f_{*}$ is injective, suppose that $f_{*}([[w]])=1$, then $f(w)\approx (f(e)|g^{k'r'd'(f(e))}|f(e))$ for some integer $k'$. By Proposition \ref{modified homotopy lifting}, $w\approx (e|g^{k'r'd'(f(e))}|e)$. Since $d(e)=d'(f(e))$, we have $\sigma^{k'r'}(e)=g^{k'r'd(e)}\cdot e=g^{k'r'd'(f(e))}\cdot e=e$. By Lemma \ref{free}, $\sigma^{k'r'}=id_E$, then $r$ divides $k'r'$. Therefore $w\approx' (e||e)$ and $[[w]]=1$.
\end{proof}

\begin{Lem}\label{iso of reduced modified fundamental groups}
If $E=(E,E,\tau,d)$ is a finite connected $f_{ms}$-BG, then $\Pi_{m}'(E,e)$ is isomorphic to $\Pi_{m}'(E/\langle\sigma\rangle,e^{\langle\sigma\rangle})$.
\end{Lem}

\begin{proof}
Let $f:E\rightarrow E/\langle\sigma\rangle$ be the natural projection. According to Lemma \ref{injective-reduced-modified-fundamental-group}, $f_{*}:\Pi'_{m}(E,e)\rightarrow\Pi'_{m}(E/\langle\sigma\rangle,e^{\langle\sigma\rangle})$ is injective. For any $[[w']]\in\Pi'_{m}(E/\langle\sigma\rangle,e^{\langle\sigma\rangle})$, lift $w'$ to a walk $w$ of $E$ with source $e$. Since the terminal of $w$ belongs to $e^{\langle\sigma\rangle}$, we may assume that $t(w)=\sigma^{n}(e)$ for some integer $n$. Let $v=(e|g^{-nd(e)}|t(w))w$ be a closed walk of $E$ at $e$. Then $f(v)=(e^{\langle\sigma\rangle}|g^{-nd(e)}|e^{\langle\sigma\rangle})w'$, where $d(e)$ is equal to the degree of $e^{\langle\sigma\rangle}$ in $E/\langle\sigma\rangle$. Since the order of the Nakayama automorphism of $E/\langle\sigma\rangle$ is $1$, we have $f(v)\approx' w'$. Then $f_{*}([[v]])=[[f(v)]]=[[w']]$ and $f_{*}:\Pi'_{m}(E,e)\rightarrow\Pi'_{m}(E/\langle\sigma\rangle,e^{\langle\sigma\rangle})$ is also surjective.
\end{proof}

For a group $\Pi$, denote by $\widetilde{\Pi}$ the abelianization of $\Pi$.

\begin{Prop}\label{rep type and reduced group}
For a finite connected $f_{ms}$-BG $E=(E,E,\tau,d)$, the following are equivalent:
\begin{itemize}
\item [$(a)$] $A_E$ is domestic;
\item [$(b)$] $\Pi'_{m}(E)\cong\mathbb{Z}$ or $\Pi'_{m}(E)\cong F\langle a,b\rangle/\langle a^2=b^2=1\rangle$;
\item [$(c)$] $\widetilde{\Pi'_{m}(E)}\cong\mathbb{Z}$ or $\widetilde{\Pi'_{m}(E)}\cong\mathbb{Z}/2\mathbb{Z}\oplus\mathbb{Z}/2\mathbb{Z}$.
\end{itemize}
\end{Prop}

\begin{proof}
Suppose that the modified BG $B=E/\langle\sigma\rangle$ has $k$-edges, $l$ double half-edges, and $n$ vertices $v_1$, $\cdots$, $v_n$ of f-degree $d_1$, $\cdots$, $d_n$ respectively. According to Lemma \ref{B}, $A_E$ is domestic if and only if $B$ satisfies one of the conditions $(1)$, $(2)$, $(3)$ in Lemma \ref{B}. We choose some $e^{\langle\sigma\rangle}\in B$.

''$(a)\Rightarrow(b)$'' Suppose that $B$ satisfies condition $(1)$ in Lemma \ref{B}, that is, $l=2$, $k-n+1=0$, $d_i=1$ for $1\leq i\leq n$. By the proof of Proposition \ref{modified fundamental group of modified f-BG}, there exists an isomorphism
$$F\langle a,c_1,c_2\rangle/\langle a c_1=c_1 a, a c_1=c_1 a, c_{1}^{2}=c_{2}^{2}=1\rangle\rightarrow\Pi_{m}(B,e^{\langle\sigma\rangle})$$
which maps $\overline{a}$ to $[(e^{\langle\sigma\rangle}|g^{d(e)}|e^{\langle\sigma\rangle})]$. By the exact sequence $$0\rightarrow\mathbb{Z}\xrightarrow{i}{}\Pi_{m}(B,e^{\langle\sigma\rangle})\xrightarrow{\pi}{}
\Pi'_{m}(B,e^{\langle\sigma\rangle})\rightarrow 0,$$
we see that $\Pi'_{m}(B,e^{\langle\sigma\rangle})$ is isomorphic to $F\langle c_1,c_2\rangle/\langle c_{1}^{2}=c_{2}^{2}=1\rangle$. By Lemma \ref{iso of reduced modified fundamental groups} we have $\Pi'_{m}(E,e)\cong F\langle c_1,c_2\rangle/\langle c_{1}^{2}=c_{2}^{2}=1\rangle$. Using the same method, it can be shown that $\Pi'_{m}(E,e)\cong F\langle a,b\rangle/\langle a^2=b^2=1\rangle$ if $B$ satisfies condition $(2)$ in Lemma \ref{B}, and $\Pi'_{m}(E,e)\cong\mathbb{Z}$ if $B$ satisfies condition $(3)$ in Lemma \ref{B}.

"$(b)\Rightarrow(c)$" By a direct calculation.

"$(c)\Rightarrow(a)$" Let $r=k-n+1$. By the proof of Proposition \ref{modified fundamental group of modified f-BG}, there exists an isomorphism \begin{multline*} f:F\langle a_1,\cdots,a_n,b_1,\cdots,b_r,c_1,\cdots,c_l\rangle/\langle a_{1}^{d_1}=\cdots=a_{n}^{d_n}, a_{1}^{d_1} b_i=b_i a_{1}^{d_1} (1\leq i\leq r), \\ a_{1}^{d_1} c_j=c_j a_{1}^{d_1}, c_{j}^{2}=1 (1\leq j\leq l)\rangle\rightarrow\Pi_{m}(B,e^{\langle\sigma\rangle})\end{multline*}
such that $f(\overline{a_{1}^{d_1}})=[(e^{\langle\sigma\rangle}|g^{d(e)}|e^{\langle\sigma\rangle})]$.  By the exact sequence $$0\rightarrow\mathbb{Z}\xrightarrow{i}{}\Pi_{m}(B,e^{\langle\sigma\rangle})\xrightarrow{\pi}{}
\Pi'_{m}(B,e^{\langle\sigma\rangle})\rightarrow 0,$$
we see that $\Pi'_{m}(B,e^{\langle\sigma\rangle})$ is isomorphic to $$F\langle a_1,\cdots,a_n,b_1,\cdots,b_r,c_1,\cdots,c_l\rangle/\langle a_{1}^{d_1}=\cdots=a_{n}^{d_n}=1, c_{j}^{2}=1 (1\leq j\leq l)\rangle.$$
Therefore the abelianization of $\Pi'_{m}(B,e^{\langle\sigma\rangle})$ is isomorphic to $$\mathbb{Z}^{r}\oplus\mathbb{Z}/d_1\mathbb{Z}\oplus\cdots\oplus\mathbb{Z}/d_n\mathbb{Z}\oplus(\mathbb{Z}/2\mathbb{Z})^{l}.$$
Since $\Pi'_{m}(B)$ is isomorphic to $\Pi'_{m}(E)$ by Lemma \ref{iso of reduced modified fundamental groups}, we see that the group $$\mathbb{Z}^{r}\oplus\mathbb{Z}/d_1\mathbb{Z}\oplus\cdots\oplus\mathbb{Z}/d_n\mathbb{Z}\oplus(\mathbb{Z}/2\mathbb{Z})^{l}$$
is isomorphic to $\mathbb{Z}$ or $\mathbb{Z}/2\mathbb{Z}\oplus\mathbb{Z}/2\mathbb{Z}$. Therefore there are only four possible cases: \\ $(\romannumeral1)$ $r=1$, $l=0$ and $d_1=\cdots=d_n=1$; \\ $(\romannumeral2)$ $r=0$, $l=0$, $d_i=2$ for exactly two numbers $i=i_0$, $i_1$, and $d_i=1$ for $i\neq i_0$, $i_1$; \\ $(\romannumeral3)$ $r=0$, $l=1$, $d_i=2$ for exactly one number $i=i_0$, and $d_i=1$ for $i\neq i_0$; \\ $(\romannumeral4)$ $r=0$, $l=2$, and $d_1=\cdots=d_n=1$.

According to Lemma \ref{odd degree}, case $(\romannumeral3)$ can not happen. Since cases $(\romannumeral1)$, $(\romannumeral2)$, $(\romannumeral4)$ correspond to conditions $(3)$, $(2)$, $(1)$ in Lemma \ref{B}, respectively, $A_E$ is domestic.
\end{proof}

\begin{Lem}\label{center 1}
The center of the group $F\langle a,b\rangle/\langle a^2=b^2\rangle$ is an infinite cyclic group generated by $\overline{a^2}$, where $\overline{a^2}$ is the image of $a^2\in F\langle a,b\rangle$ in $F\langle a,b\rangle/\langle a^2=b^2\rangle$.
\end{Lem}

\begin{proof}
It is straightforward to show that each element of $F\langle a,b\rangle/\langle a^2=b^2\rangle$ is of one of the following types: $(1)$ $\overline{a^{2m}(ab)^{n}a}$, $m\in\mathbb{Z}$, $n\in\mathbb{N}$; $(2)$ $\overline{a^{2m}(ba)^{n}b}$, $m\in\mathbb{Z}$, $n\in\mathbb{N}$; $(3)$ $\overline{a^{2m}(ab)^{n}}$, $m\in\mathbb{Z}$, $n\in\mathbb{N}$; $(4)$ $\overline{a^{2m}(ba)^{n}}$, $m\in\mathbb{Z}$, $n\in\mathbb{N}$. Let $E$ be the Brauer graph
$$\begin{tikzpicture}
\draw (0,0)--(2,0);
\fill (0,0) circle (0.5ex);
\fill (2,0) circle (0.5ex);
\node at(-0.4,0) {2};
\draw (-0.6,-0.2) rectangle (-0.2,0.2);
\node at(2.4,0) {2};
\draw (2.2,-0.2) rectangle (2.6,0.2);
\end{tikzpicture}$$
and let $e$ be a half-edge of $E$. By the proof of \cite[Lemma 6.8]{LL2}, there exists an isomorphism
$$f:F\langle a,b\rangle/\langle a^2=b^2\rangle\rightarrow\Pi(E,e)$$
such that $f(\overline{a})=\overline{(e|g|e)}$, $f(\overline{b})=\overline{(e|\tau g\tau|e)}$. For any $m\in\mathbb{Z}$ and $n\in\mathbb{N}$, if $\overline{a^{2m}(ab)^{n}a}$ belongs to the center of $F\langle a,b\rangle/\langle a^2=b^2\rangle$, then so does $\overline{(ab)^{n}a}$. Therefore
$$\overline{(e|(\tau g\tau g)^{n+1}|e)}=f(\overline{b(ab)^{n}a})=f(\overline{(ab)^{n}ab})=\overline{(e|(g\tau g\tau)^{n+1}|e)}.$$
Since $(e|(\tau g\tau g)^{n+1}|e)$ and $(e|(g\tau g\tau)^{n+1}|e)$ are homotopic special walks of $E$, by \cite[Proposition 3.46]{LL2} they are equal, a contradiction. Therefore any element of $F\langle a,b\rangle/\langle a^2=b^2\rangle$ of type $(1)$ does not belong to the center of $F\langle a,b\rangle/\langle a^2=b^2\rangle$. Using the same method, it can be shown that each element of type $(2)$ does not belong to the center of $F\langle a,b\rangle/\langle a^2=b^2\rangle$, and each element of type $(3)$ or type $(4)$ belongs to the center of $F\langle a,b\rangle/\langle a^2=b^2\rangle$ if and only if $n=0$. Therefore center of $F\langle a,b\rangle/\langle a^2=b^2\rangle$ is generated by $\overline{a^2}$. Since $f(\overline{a^2})=\overline{(e|g^2|e)}$, by \cite[Proposition 3.46]{LL2}, the order of $\overline{a^2}$ is infinite.
\end{proof}

\begin{Lem}\label{center 2}
For any positive integers $m$, $n\geq 2$, the center of $F\langle a,b\rangle/\langle a^m=b^n=1\rangle$ is trivial.
\end{Lem}

\begin{proof}
Let $E$ be the Brauer graph
$$\begin{tikzpicture}
\draw (0,0)--(2,0);
\fill (0,0) circle (0.5ex);
\fill (2,0) circle (0.5ex);
\node at(-0.4,0) {$m$};
\draw (-0.6,-0.2) rectangle (-0.2,0.2);
\node at(2.4,0) {$n$};
\draw (2.2,-0.2) rectangle (2.6,0.2);
\end{tikzpicture}.$$
Suppose that $e$ is the half-edge of $E$ on the left of the diagram. By the proof of Lemma \ref{a calculation of modified fundamental groupoid}, there exists an isomorphism $$f:F\langle a,b\rangle/\langle a^m=b^n\rangle\rightarrow\Pi_{m}(E,e)$$
such that $f(\overline{a})=[(e|g|e)]$, $f(\overline{b})=[(e|\tau g\tau|e)]$. Therefore $f$ induces an isomorphism
$$\widetilde{f}:F\langle a,b\rangle/\langle a^m=b^n=1\rangle\rightarrow\Pi'_{m}(E,e).$$
It is obvious that every element of $F\langle a,b\rangle/\langle a^m=b^n=1\rangle$ can be expressed as the following form
$$\overline{a^{i_k}b^{j_k}a^{i_{k-1}}b^{j_{k-1}}\cdots a^{i_1}b^{j_1}a^{i_0}},$$
where $0\leq i_0$, $i_k<m$, $0<i_l<m$ for $1\leq l\leq k-1$, and $0<j_l<n$ for $1\leq l\leq k$. We need show that such an expression is unique: Suppose that $\overline{a^{i_k}b^{j_k}a^{i_{k-1}}b^{j_{k-1}}\cdots a^{i_1}b^{j_1}a^{i_0}}$ and $\overline{a^{p_r}b^{q_r}a^{p_{r-1}}b^{q_{r-1}}\cdots a^{p_1}b^{q_1}a^{p_0}}$ are two such expressions with
$$\overline{a^{i_k}b^{j_k}a^{i_{k-1}}b^{j_{k-1}}\cdots a^{i_1}b^{j_1}a^{i_0}}=\overline{a^{p_r}b^{q_r}a^{p_{r-1}}b^{q_{r-1}}\cdots a^{p_1}b^{q_1}a^{p_0}}.$$
Then
$$[[w]]=\widetilde{f}(\overline{a^{i_k}b^{j_k}a^{i_{k-1}}b^{j_{k-1}}\cdots a^{i_1}b^{j_1}a^{i_0}})=\widetilde{f}(\overline{a^{p_r}b^{q_r}a^{p_{r-1}}b^{q_{r-1}}\cdots a^{p_1}b^{q_1}a^{p_0}})=[[v]],$$
where
$$w=(e|g^{i_k}\tau g^{j_k}\tau g^{i_{k-1}}\tau g^{j_{k-1}}\tau\cdots g^{i_1}\tau g^{j_1}\tau g^{i_0}|e),$$
$$v=(e|g^{p_r}\tau g^{q_r}\tau g^{p_{r-1}}\tau g^{q_{r-1}}\tau\cdots g^{p_1}\tau g^{q_1}\tau g^{p_0}|e)$$ are two special walks of $E$. Since $w\approx' v$, we have $w\approx (e|g^{mN}|e)v$ for some integer $N$. By Proposition \ref{modified unique factorization}, $N=0$ and $w=v$. Therefore $k=r$, $i_l=p_l$ for $0\leq l\leq k$, and $j_l=q_l$ for $1\leq l\leq k$.

If $x$ belongs to the center of the group $F\langle a,b\rangle/\langle a^m=b^n=1\rangle$, write $x$ as the standard form $\overline{a^{i_k}b^{j_k}a^{i_{k-1}}b^{j_{k-1}}\cdots a^{i_1}b^{j_1}a^{i_0}}$ as above. Since $x$ commutes with $\overline{a}$ and $\overline{b}$, by the uniqueness expression property, we know that $x$ must be equal to $1$.  \end{proof}

The following result shows that the converse of Proposition \ref{rep type to fundamental group} also holds.

\begin{Prop}\label{fundamental group to rep type}
Let $E=(E,E,\tau,d)$ be a finite connected $f_{ms}$-BG. If $\Pi(E)\cong F\langle a,b\rangle/\langle a^2=b^2\rangle$ or $\Pi(E)\cong\mathbb{Z}\oplus\mathbb{Z}$, then $A_E$ is domestic.
\end{Prop}

\begin{proof}
If $\Pi(E)\cong F\langle a,b\rangle/\langle a^2=b^2\rangle$, by Lemma \ref{two definitions of fundamental group are equal}, $\Pi_{m}(E)\cong F\langle a,b\rangle/\langle a^2=b^2\rangle$. Since there exists an exact sequence
$$0\rightarrow\mathbb{Z}\xrightarrow{i}{}\Pi_{m}(E,e)\xrightarrow{\pi}{}\Pi_{m}'(E,e)\rightarrow 0,$$
where $i(1)$ belongs to the center of $\Pi_{m}(E,e)$, by Lemma \ref{center 1} we see that $\Pi_{m}'(E)\cong F\langle a,b\rangle/\langle a^2=b^2$, $a^{2N}=1\rangle$ for some positive integer $N$.
Then the abelianization of $\Pi_{m}'(E)$ is isomorphic to $\mathbb{Z}/2\mathbb{Z}\oplus\mathbb{Z}/2N\mathbb{Z}$. Suppose that $E/\langle \sigma\rangle$ has $k$-edges, $l$ double half-edges, and $n$ vertices $v_1$, $\cdots$, $v_n$ of f-degree $d_1$, $\cdots$, $d_n$, respectively. By Lemma \ref{iso of reduced modified fundamental groups}, $\Pi_{m}'(E)\cong\Pi_{m}'(E/\langle\sigma\rangle)$, which is isomorphic to
$$F\langle a_1,\cdots,a_n,b_1,\cdots,b_r,c_1,\cdots,c_l\rangle/\langle a_{1}^{d_1}=\cdots=a_{n}^{d_n}=c_{1}^{2}=\cdots=c_{l}^{2}=1\rangle$$ by Proposition \ref{modified fundamental group of modified f-BG}, where $r=k-n+1$. The abelianization of
$$F\langle a_1,\cdots,a_n,b_1,\cdots,b_r,c_1,\cdots,c_l\rangle/\langle a_{1}^{d_1}=\cdots=a_{n}^{d_n}=c_{1}^{2}=\cdots=c_{l}^{2}=1\rangle$$
is $\mathbb{Z}^{r}\oplus\mathbb{Z}/d_1\mathbb{Z}\oplus\cdots\oplus\mathbb{Z}/d_n\mathbb{Z}\oplus(\mathbb{Z}/2\mathbb{Z})^{l}$, which is isomorphic to $\mathbb{Z}/2\mathbb{Z}\oplus\mathbb{Z}/2N\mathbb{Z}$. Suppose that $N>1$, then there are two possible cases: \\ $(1)$ $r=l=0$, $d_{i_0}=2N$, $d_{i_1}=2$ for some $1\leq i_0,i_1\leq n$, and $d_i=1$ for $i\neq i_0$, $i_1$; \\ $(2)$ $r=0$, $l=1$, $d_{i_0}=2N$ for some $1\leq i_0\leq n$, and $d_i=1$ for $i\neq i_0$. \\ If case $(1)$ occurs, then $$\Pi'_{m}(E)\cong\Pi_{m}'(E/\langle\sigma\rangle)\cong F\langle a_1,a_2\rangle/\langle a_{1}^{2N}=a_{2}^{2}=1\rangle.$$
So by Lemma \ref{center 2}, the center of $\Pi'_{m}(E)$ is trivial. But we also have $\Pi_{m}'(E)\cong F\langle a,b\rangle/\langle a^2=b^2$, $a^{2N}=1\rangle$, therefore the center of $\Pi'_{m}(E)$ contains an element $\overline{a^2}$, which is not equal to the identity element, a contradiction. If case $(2)$ occurs, since $E/\langle \sigma\rangle$ contains double half-edges, the automorphism group $\langle\sigma\rangle$ of $E$ is not admissible. By Lemma \ref{odd degree}, the f-degree of each vertex of $E/\langle \sigma\rangle$ is odd, a contradiction. Therefore $N=1$, and $\Pi_{m}'(E)\cong F\langle a,b\rangle/\langle a^2=b^2=1\rangle$. By Proposition \ref{rep type and reduced group}, $A_E$ is domestic.

If $\Pi(E)\cong\mathbb{Z}\oplus\mathbb{Z}$, by Lemma \ref{two definitions of fundamental group are equal}, $\Pi_{m}(E)$ is also isomorphic to $\mathbb{Z}\oplus\mathbb{Z}$. By the exact sequence
$$0\rightarrow\mathbb{Z}\rightarrow\Pi_{m}(E)\rightarrow\Pi_{m}'(E)\rightarrow 0,$$
we see that $\Pi_{m}'(E)\cong\mathbb{Z}\oplus\mathbb{Z}/N\mathbb{Z}$ for some positive integer $N$. Suppose that $E/\langle \sigma\rangle$ has $k$-edges, $l$ double half-edges, and $n$ vertices $v_1$, $\cdots$, $v_n$ of f-degree $d_1$, $\cdots$, $d_n$, respectively. By Lemma \ref{iso of reduced modified fundamental groups}, $\Pi_{m}'(E)\cong\Pi_{m}'(E/\langle\sigma\rangle)$, which is isomorphic to
$$F\langle a_1,\cdots,a_n,b_1,\cdots,b_r,c_1,\cdots,c_l\rangle/\langle a_{1}^{d_1}=\cdots=a_{n}^{d_n}=c_{1}^{2}=\cdots=c_{l}^{2}=1\rangle$$ by Proposition \ref{modified fundamental group of modified f-BG}, where $r=k-n+1$. The abelianization of
$$F\langle a_1,\cdots,a_n,b_1,\cdots,b_r,c_1,\cdots,c_l\rangle/\langle a_{1}^{d_1}=\cdots=a_{n}^{d_n}=c_{1}^{2}=\cdots=c_{l}^{2}=1\rangle$$
is $\mathbb{Z}^{r}\oplus\mathbb{Z}/d_1\mathbb{Z}\oplus\cdots\oplus\mathbb{Z}/d_n\mathbb{Z}\oplus(\mathbb{Z}/2\mathbb{Z})^{l}$, which is isomorphic to $\mathbb{Z}\oplus\mathbb{Z}/N\mathbb{Z}$. We have four possible cases: \\ $(1)$ $N=1$, $r=1$, $l=0$, $d_1=\cdots=d_n=1$; \\ $(2)$ $N=2$, $r=1$, $l=0$, $d_i=2$ for some $1\leq i\leq n$, and $d_j=1$ for $j\neq i$. \\ $(3)$ $N=2$, $r=1$, $l=1$, $d_1=\cdots=d_n=1$; \\ $(4)$ $N>2$, $r=1$, $l=0$, $d_i=N$ for some $1\leq i\leq n$, and $d_j=1$ for $j\neq i$. \\ Since $\Pi_{m}'(E)$ is isomorphic to $\mathbb{Z}\oplus\mathbb{Z}/N\mathbb{Z}$, it is abelian. However, $\Pi_{m}'(E)$ is also isomorphic to $$F\langle a_1,\cdots,a_n,b_1,\cdots,b_r,c_1,\cdots,c_l\rangle/\langle a_{1}^{d_1}=\cdots=a_{n}^{d_n}=c_{1}^{2}=\cdots=c_{l}^{2}=1\rangle,$$
which is isomorphic to $F\langle a,b\rangle/\langle a^2=1\rangle$, $F\langle b,c\rangle/\langle c^2=1\rangle$, and $F\langle a,b\rangle/\langle a^N=1\rangle$ for case $(2)$, $(3)$, and $(4)$, respectively. We see that $\Pi_{m}'(E)$ is non-abelian in any of these cases, a contradiction. Therefore only case $(1)$ can occurs. In this case, $\Pi_{m}'(E)\cong\mathbb{Z}$, and by Proposition \ref{rep type and reduced group}, $A_E$ is domestic.
\end{proof}

\begin{Thm}\label{Theorem-fundamental-group-of-domestic-fms-BGA}
For a finite connected $f_{ms}$-BG $E$, $A_E$ is domestic if and only if $\Pi(E)\cong F\langle a,b\rangle/\langle a^2=b^2\rangle$ or $\Pi(E)\cong\mathbb{Z}\oplus\mathbb{Z}$.
\end{Thm}

\begin{proof}
By Proposition \ref{rep type to fundamental group} and Proposition \ref{fundamental group to rep type}.
\end{proof}

\section{Representation-finite fractional Brauer graph algebras of type MS}

In this section we assume that $k$ is an algebraically closed field. We abbreviate indecomposable, basic, representation-finite self-injective algebra over $k$ (not isomorphic to the underlying field $k$) by RFS algebra. In this paper, we only concentrate on RFS algebras of class $A_n$. For the definition and classification of RFS algebras under Morita equivalence, we refer to the classical works in \cite{R,R2,BLR}.

In \cite{R}, based on a general description on AR-quivers of RFS algebras, Riedtmann constructed RFS algebras of class $A_n$ from their AR-quivers. By using the quivers with relations of RFS algebras of class $A_n$ given in \cite[Section 6.2]{R}, we show that these algebras coincide with representation-finite $f_{ms}$-BGAs (see Proposition \ref{RFS-class-An}). According to this result, we would like to determine the AR-quiver of an RFS algebra of class $A_n$ using its defining $f_{ms}$-BG $E$ (according to Proposition \ref{modified existence of morphism}, $E$ is determined by its quotient $E/\langle\sigma\rangle$ together with the image of the group homomorphism $\Pi_{m}(E)\rightarrow\Pi_{m}(E/\langle\sigma\rangle)$ up to isomorphism). First we need some knowledge on configurations of stable translation-quivers of tree class $A_n$.

\subsection{Review on configurations of stable translation-quivers of tree class $A_n$}
\

\begin{Def}\label{combinatorial-configuration} {\rm(\cite[Definition 2.3]{R})}
Let $\Gamma$ be a stable translation quiver and let $k(\Gamma)$ be the mesh category of $\Gamma$. A (combinatorial)
configuration $\mathcal{C}$ is a set of vertices of $\Gamma$ which
satisfies the following conditions:
\begin{enumerate}
\item For any $e, f\in \mathcal{C}$,
$\mathrm{Hom}_{k(\Gamma)}(e,f)= \left\{\begin{array}{ll} 0 & (e\neq f), \\
k & (e=f).\end{array}\right.$
\item For any $e\in\Gamma_0$, there exists some $f\in \mathcal{C}$ such that $\mathrm{Hom}_{k(\Gamma)}(e,f)\neq 0$.
\end{enumerate}
\end{Def}

Let $\Pi$ be an admissible group of automorphisms of $\mathbb{Z}A_n$, and let $\mathscr{C}$ be a $\Pi$-stable configuration of $\mathbb{Z}A_n$. By \cite[Theorem 5]{R}, there exists an RFS algebra $A_{\mathscr{C},\Pi}$ of class $A_n$ such that $\mathrm{ind}A_{\mathscr{C},\Pi}$ is isomorphic to the mesh category of the translation quiver $(\mathbb{Z}A_n)_{\mathscr{C}}/\Pi$, where $\mathrm{ind}A_{\mathscr{C},\Pi}$ denotes the category of a chosen set of representatives of non-isomorphic indecomposable finitely generated $A_{\mathscr{C},\Pi}$-modules. Together with \cite[Theorem 4.1]{R}, we have

\begin{Prop}\label{configurations and RFS algebras}{\rm(\cite{R})}
The map $\mathscr{C}\mapsto A_{\mathscr{C},\Pi}$ is a bijection between the isomorphism classes of $\Pi$-stable configurations of $\mathbb{Z}A_n$ and the isomorphism classes of RFS algebras of class $A_n$ with admissible group $\Pi$.
\end{Prop}

\begin{Def}\label{Brauer relation} {\rm(\cite[Definition 2.6]{R})}
A Brauer relation of order $n$ is an equivalence relation on the set $\sqrt[n]{1}=\{e^{\frac{2m\pi}{n}i}\mid m\in\mathbb{Z}\}\subseteq\mathbb{C}$ such that the convex hulls of distinct equivalence classes are disjoint.
\end{Def}

If $\mathcal{B}$ is a Brauer relation of order $n$, we denote by $\beta_{\mathcal{B}}$ the permutation of $\sqrt[n]{1}$ assigning to each point $s$ its successor in the equivalence class of $s$ endowed with the anti-clockwise orientation, see \cite[Section 2.6]{R}.

\begin{Prop}\label{a one-to-one correspond} {\rm(\cite[Proposition 2.6]{R})}
Let $\mathcal{B}$ be a Brauer relation of order $n$ and denote by $\mathscr{C}_{\mathcal{B}}$ the set of vertices $(i,j)$ of $\mathbb{Z}A_n$ such that $e_{n}(i+j)=\beta_{\mathcal{B}}(e_{n}(i))$, where $e_{n}(m)=e^{\frac{2m\pi}{n}i}$. The map $\mathcal{B}\mapsto\mathscr{C}_{\mathcal{B}}$ is a bijection between the Brauer relations of order $n$ and the configurations of $\mathbb{Z}A_n$.
\end{Prop}

For any integer $p$, let $\Delta_{p}=\{(p,i)\mid 1\leq i\leq n\}\subseteq (\mathbb{Z}A_n)_0$ be the ``going up diagonal'' and $\nabla_{p}=\{(p-i,i)\mid 1\leq i\leq n\}\subseteq (\mathbb{Z}A_n)_0$ be the ``going down diagonal'' of $\mathbb{Z}A_n$. If $\mathscr{C}$ is a configuration of $\mathbb{Z}A_n$, define two permutations $\alpha_{\mathscr{C}}$ and $\beta_{\mathscr{C}}$ of $\mathbb{Z}$ (see \cite[Section 3.4]{R}) as follows: $\alpha_{\mathscr{C}}(p)=n+x+1$, where $(x,y)$ is the unique point of $\mathscr{C}$ in $\nabla_{p}$, and $\beta_{\mathscr{C}}(p)=p+i$, where $(p,i)$ is the unique point of $\mathscr{C}$ in $\Delta_{p}$. It is straightforward to show that $\alpha_{\mathscr{C}}\beta_{\mathscr{C}}(p)=p+n+1$ for any $p\in\mathbb{Z}$.

Given any admissible group of automorphisms $\Pi$ of $\mathbb{Z}A_n$ and any $\Pi$-stable configuration $\mathscr{C}$ of $\mathbb{Z}A_n$, the RFS algebra $A_{\mathscr{C},\Pi}$ of class $A_n$ can be described as follows (see \cite[Section 6.2]{R}): Let $\alpha=\alpha_{\mathscr{C}}$ and $\beta=\beta_{\mathscr{C}}$ be the permutations of $\mathbb{Z}$ associated with $\mathscr{C}$. Let $c_r$ be the unique point of $\mathscr{C}$ on $\nabla_{r}$. Define an action of $\Pi$ on $\mathbb{Z}$ by setting $c_{\pi r}=\pi c_r$ for any $\pi\in\Pi$ and $r\in\mathbb{Z}$. Let $\widetilde{Q}=\widetilde{Q}_{\mathscr{C}}$ be the quiver with $\widetilde{Q}_{0}=\mathbb{Z}$ and $\widetilde{Q}_{1}=\{\alpha_r,\beta_r\mid r\in\mathbb{Z}\}$, where $\alpha_r:r\rightarrow\alpha(r)$ and $\beta_r:r\rightarrow\beta(r)$. For any $\pi\in\Pi$, it can be shown that either $\pi\alpha=\alpha \pi$ and $\pi\beta=\beta \pi$, or $\pi\alpha=\beta \pi$ and $\pi\beta=\alpha \pi$, depending on whether $\pi$ is a translation or a translation-reflection. Therefore $\pi$ induces an automorphism of $\widetilde{Q}$. Denote by $Q_{\mathscr{C},\Pi}$ the residue quiver $\widetilde{Q}/\Pi$, and denote by $\overline{\alpha_r}$, $\overline{\beta_r}$ the residue classes of $\alpha_r$, $\beta_r$ modulo $\Pi$ respectively.

\begin{Thm}\label{quiver and relation} {\rm(\cite[Theorem 6.2]{R})}
$A_{\mathscr{C},\Pi}$ is isomorphic to the algebra defined by the quiver $Q_{\mathscr{C},\Pi}$ and the relations $\overline{\beta}_{\alpha(r)}\overline{\alpha}_{r}=\overline{\alpha}_{\beta(r)}\overline{\beta}_{r}=0$ and $\overline{\alpha}_{\alpha^{a_r-1}(r)}\cdots\overline{\alpha}_{\alpha(r)}\overline{\alpha}_{r}
=\overline{\beta}_{\beta^{b_r-1}(r)}\cdots\overline{\beta}_{\beta(r)}\overline{\beta}_{r}$, where $r\in\mathbb{Z}$ and $a_r$, $b_r$ are defined by $\alpha^{a_r}(r)=r+n=\beta^{b_r}(r)$.
\end{Thm}

According to \cite[Section 3.1]{Br-G}, the standard form $\overline{A}$ of a basic representation-finite algebra $A$ is defined to be the opposite algebra of the algebra $\oplus_{p,q}k(\Gamma_{A})(p,q)$, where $\Gamma_A$ is the AR-quiver of $A$ and $p,q$ range over all projective vertices of $\Gamma_A$. Note that each RFS algebra of class $A_n$ is isomorphic to its standard form. It was shown in \cite[Section 3.1]{Br-G} that the standard form $\overline{A}$ of $A$ is isomorphic to $kQ_A/I_A$, where $Q_A$ is the quiver of $A$ and $I_A$ is the ideal in $kQ_A$ generated by the non-stable paths and by the differences $v-w$ where $(v,w)$ range over all stable contours. For each RFS algebra $A$, using the quiver with relations $(Q_A,I_A)$ we can construct an $f_s$-BC $E$ with $\overline{A}\cong A_E$, see \cite[Proposition 7.13]{LL}. By the description of the quivers with relations of RFS algebras of class $A_n$ in Theorem \ref{quiver and relation}, we have

\begin{Prop} \label{RFS-class-An}
The class of representation-finite $f_{ms}$-BGAs coincides with the class of RFS algebras of class $A_n$.
\end{Prop}

\begin{proof} Let $A_E$ be a representation-finite $f_{ms}$-BGA. Since $A_E$ is self-injective special biserial, for every vertex $x$ of the stable AR-quiver of $A_E$, there are at most two arrows starting at $x$. Therefore $A_E$ cannot be an RFS algebra of class $D_n$ or class $E_n$.

On the other hand, let $A=kQ/I$ be the RFS algebra of class $A_n$ given by the admissible group of automorphisms $\Pi$ of $\mathbb{Z}A_n$ and the $\Pi$-stable configuration $\mathscr{C}$, where $Q=Q_{\mathscr{C},\Pi}$ and $I$ is the ideal in $kQ_{\mathscr{C},\Pi}$ generated by $\overline{\beta}_{\alpha(r)}\overline{\alpha}_{r}$, $\overline{\alpha}_{\beta(r)}\overline{\beta}_{r}$ and $\overline{\alpha}_{\alpha^{a_r-1}(r)}\cdots\overline{\alpha}_{\alpha(r)}\overline{\alpha}_{r}
-\overline{\beta}_{\beta^{b_r-1}(r)}\cdots\overline{\beta}_{\beta(r)}\overline{\beta}_{r}$, where $r\in\mathbb{Z}$ and $a_r$, $b_r$ are defined by $\alpha^{a_r}(r)=r+n=\beta^{b_r}(r)$. Denote by $\widetilde{A}$ the locally representation-finite category $k\widetilde{Q}/\widetilde{I}$, where $\widetilde{Q}=\widetilde{Q}_{\mathscr{C}}$ and $\widetilde{I}$ is the ideal in $k\widetilde{Q}$ generated by $\beta_{\alpha(r)}\alpha_{r}$, $\alpha_{\beta(r)}\beta_{r}$ and $\alpha_{\alpha^{a_r-1}(r)}\cdots\alpha_{\alpha(r)}\alpha_{r}
-\beta_{\beta^{b_r-1}(r)}\cdots\beta_{\beta(r)}\beta_{r}$, where $r\in\mathbb{Z}$ and $a_r$, $b_r$ are defined by $\alpha^{a_r}(r)=r+n=\beta^{b_r}(r)$. Since $(\widetilde{Q},\widetilde{I})$ is a universal cover of $(Q,I)$ and $\widetilde{Q}$ contains no oriented cycles, by \cite{MP} $\widetilde{A}$ is a universal cover of $A_{\mathscr{C},\Pi}$ in the sense of \cite[Section 2.1]{G}. Then a path of $Q$ is stable (that is, it lifts to a nonzero path in $\widetilde{Q}$, see \cite{Br-G}) if and only if it is a subpath of a path of the form $\overline{\alpha}_{\alpha^{a_r-1}(r)}\cdots\overline{\alpha}_{\alpha(r)}\overline{\alpha}_{r}$ or $\overline{\beta}_{\beta^{b_r-1}(r)}\cdots\overline{\beta}_{\beta(r)}\overline{\beta}_{r}$. By \cite[Lemma 7.7]{LL}, each pair $$(\overline{\alpha}_{\alpha^{a_r-1}(r)}\cdots\overline{\alpha}_{\alpha(r)}\overline{\alpha}_{r},
\overline{\beta}_{\beta^{b_r-1}(r)}\cdots\overline{\beta}_{\beta(r)}\overline{\beta}_{r})$$
is a stable contour of $A$. Therefore $I$ is contained in the ideal $I_A$ of $kQ$ given in \cite[Section 3.1]{Br-G}. Since $A=kQ/I$ and $\overline{A}=kQ/I_A$ are isomorphic, we have $I=I_A$.

Let $E=(E,P,L,d)$ be the associated $f_s$-BC of $\overline{A}$ (cf. remarks before \cite[Proposition 7.13]{LL}), where
$$E=\{e_v\mid v\mbox{ is a path of }Q\mbox{ which belongs to }\mathrm{soc}\overline{A}-\{0\}\}$$
and for every $e_v\in E$, $P(e_v)=\{e_w\mid s(v)=s(w)\}$ and $L(e_v)=\{e_w\mid \alpha_v=\alpha_w\}$ ($\alpha_v$ is the initial arrow of $v$). Since $(Q,I_A)=(Q,I)$ is special biserial, each polygon $P(e_v)$ of $E$ contains at most two elements and each $L(e_v)=\{e_v\}$. By adding a half-edge to every single half-edge of $E$, we obtain an $f_{ms}$-BG $E'$ such that the corresponding algebras of $E$ and $E'$ are isomorphic. Therefore $A\cong\overline{A}\cong A_{E'}$ is an $f_{ms}$-BGA.
\end{proof}

Next we will give a Brauer relation of order $n$ from an f-degree-free Brauer tree with $n$ edges. Let $B=(B,P,L,d)$ be such a Brauer tree and fix $e\in B$.
Note that we can view $B$ as a Brauer $G$-set with $\tau$ the involution on $B$.
Then for each half-edge $b$ of $B$, there exists a unique integer $0\leq i\leq 2n-1$ such that $b=(\tau g)^i(e)$. Denote by $\alpha_i$ the arrow $L((\tau g)^{i-1}(e))$ in $Q_B$, where $1\leq i\leq 2n$. Call the arrows $\alpha_{2j-1}$ ($1\leq j\leq n$) the {\it $\beta$-arrows} and the arrows $\alpha_{2j}$ ($1\leq j\leq n$) the {\it $\alpha$-arrows} of $Q_B$. Note that for each half-edge $h$ of $B$, $L(h)$ is a $\beta$-arrow if and only if $L(\tau(h))$ is an $\alpha$-arrow, and $L(h)$ is a $\beta$-arrow if and only if $L(g\cdot h)$ is a $\beta$-arrow. Moreover, call a path of $Q_B$ a {\it $\beta$-path} (resp. {\it $\alpha$-path}) if each arrow of this path is a $\beta$-arrow (resp. $\alpha$-arrow). Define an equivalence relation $\mathcal{B}$ on $\sqrt[n]{1}=\{e^{\frac{2m\pi}{n}i}\mid m\in\mathbb{Z}\}\subseteq\mathbb{C}$ as follows: $e^{\frac{2k\pi}{n}i}$ and $e^{\frac{2l\pi}{n}i}$ are equivalent if and only if the vertices $P((\tau g)^{2k}(e))$ and $P((\tau g)^{2l}(e))$ of $Q_B$ can be connected by a $\beta$-path (the elements in $\sqrt[n]{1}$ are in one-to-one correspondence with the vertices of $Q_B$, which is given by the map $e^{\frac{2m\pi}{n}i}\mapsto P((\tau g)^{2m}(e))$). Moreover, we denote by $\beta_{\mathcal{B}}$ the permutation of $\sqrt[n]{1}$ assigning to each point $s$ its successor in the equivalence class of $s$ endowed with the anti-clockwise orientation.

\begin{Prop}\label{to be a Brauer relation}
$\mathcal{B}$ is a Brauer relation of order $n$. Moreover, for each $e^{\frac{2k\pi}{n}i}\in\sqrt[n]{1}$, suppose the terminal of the $\beta$-arrow $\alpha_{2k+1}$ of $Q_B$ is $P((\tau g)^{2l}(e))$, then $\beta_{\mathcal{B}}(e^{\frac{2k\pi}{n}i})=e^{\frac{2l\pi}{n}i}$.
\end{Prop}

\begin{proof}
For $e^{\frac{2k\pi}{n}i}, e^{\frac{2l\pi}{n}i}\in\sqrt[n]{1}$, suppose that $0\leq k,l <n$. Let
\begin{equation*}
l'= \begin{cases}
l, & \text{if } l> k; \\
l+n, & l\leq k,
\end{cases}
\end{equation*}
and define $(e^{\frac{2k\pi}{n}i}, e^{\frac{2l\pi}{n}i})=\{e^{\frac{2r\pi}{n}i}\in\sqrt[n]{1}\mid k<r<l'\}$.

For any $\beta$-arrow $\alpha_{2k+1}$ of $Q_B$, it can be shown that $g(\tau g)^{2k}(e)=(\tau g)^{2(k+l)}(e)$, where $l$ is the number of vertices of $B$ which can be connected to the vertex $G\cdot(\tau g)^{2k}(e)$ of $B$ via a path of $B$ that contains the edge $P((\tau g)^{2k+1}(e))$ (here we consider $B$ as a graph). That is, $l$ is the number of edges in the dotted circle in Figure $1$ below (Here the order of half-edges which are connected to a vertex in the diagram of $B$ is taken to be anticlockwise).
$$\begin{tikzpicture}
\fill (0,0) circle (0.5ex);
\draw (0,0)--(-1.5,-2.598);
\draw (0,0)--(6,0);
\fill (6,0) circle (0.5ex);
\draw[->] (-0.75,-1.299) arc (-120:0:1.5);
\node at(1.5,-1.5) {$\alpha_{2k+1}$};
\node at(-1.35,-0.75) {$(\tau g)^{2k}(e)$};
\node at(1.5,0.3) {$g(\tau g)^{2k}(e)$};
\node at(1.5,0.8) {$||$};
\node at(1.5,1.3) {$(\tau g)^{2(k+l)}(e)$};
\node at(4.8,0.3) {$(\tau g)^{2k+1}(e)$};
\draw[dotted] (0.4,0.1) arc (10:230:0.45);
\draw[dotted] (-1.6,-2.7712)--(-2.1,-3.6372);
\draw (6,0)--(7,1.732);
\draw (6,0)--(7,-1.732);
\draw[dotted] (6.55,-0.866) arc (-60:60:1);
\draw[dotted] (7.1,1.9052)--(7.6,2.7712);
\draw[dotted] (7.1,-1.9052)--(7.6,-2.7712);
\draw[dotted] (9,0) arc (0:360:4.5);
\end{tikzpicture}$$
$$\mbox{Figure 1: The diagram of $B$}$$
So the terminal of the $\beta$-arrow $\alpha_{2k+1}$ is $P((\tau g)^{2(k+l)}(e))$, and we have $\beta_{\mathcal{B}}(e^{\frac{2k\pi}{n}i})=e^{\frac{2(k+l)\pi}{n}i}$. To show that $\mathcal{B}$ is a Brauer relation of order $n$, it suffices to show that for any $e^{\frac{2r\pi}{n}i}\in (e^{\frac{2k\pi}{n}i}, e^{\frac{2(k+l)\pi}{n}i})$, $e^{\frac{2(r+s)\pi}{n}i}$ also belongs to $(e^{\frac{2k\pi}{n}i}, e^{\frac{2(k+l)\pi}{n}i})$, where $1\leq s\leq n$ is the integer such that $g(\tau g)^{2r}(e)=(\tau g)^{2(r+s)}(e)$.

Since $e^{\frac{2r\pi}{n}i}\in (e^{\frac{2k\pi}{n}i}, e^{\frac{2(k+l)\pi}{n}i})$, we may assume that $k<r<k+l$. So the half-edge $(\tau g)^{2r}(e)$ of $B$ belongs to the dotted circle in Figure $1$ (Note that the set of half-edges in the dotted circle in Figure $1$ is $\{(\tau g)^{2k+i}(e)\mid 1\leq i\leq 2l\}$). Since $(\tau g)^{2r}(e)\neq (\tau g)^{2(k+l)}(e)$, the half-edge $(\tau g)^{2(r+s)}(e)=g(\tau g)^{2r}(e)$ also belongs to the dotted circle in Figure $1$. So $(\tau g)^{2(r+s)}(e)\in\{(\tau g)^{2k+i}(e)\mid 1\leq i\leq 2l\}$. Suppose $(\tau g)^{2(r+s)}(e)=(\tau g)^{2k+i}(e)$ for some $1\leq i\leq 2l$, then $2n|(2(r+s)-(2k+i))$, so we have $i=2j$ for some $1\leq j\leq l$ and $e^{\frac{2(r+s)\pi}{n}i}=e^{\frac{2(k+j)\pi}{n}i}$. If $j=l$, then $g(\tau g)^{2r}(e)=(\tau g)^{2(r+s)}(e)=(\tau g)^{2(k+l)}(e)=g(\tau g)^{2k}(e)$ (the last identity follows from Figure $1$) and $(\tau g)^{2r}(e)=(\tau g)^{2k}(e)$. Therefore $(\tau g)^{2(r-k)}(e)=e$, and $2n|2(r-k)$, a contradiction. So $1\leq j<l$ and $e^{\frac{2(r+s)\pi}{n}i}=e^{\frac{2(k+j)\pi}{n}i}\in (e^{\frac{2k\pi}{n}i}, e^{\frac{2(k+l)\pi}{n}i})$.
\end{proof}

\subsection{AR-quivers of representation-finite $f_{ms}$-BGAs: the main statements.}
\

Let $E=(E,E,\tau,d)$ be a finite connected $f_{ms}$-BG with Nakayama automorphism $\sigma$ such that $\Lambda_E$ is representation-finite (equivalently, $A_E$ is representation-finite) and let $R_E$ be the reduced form of $E$ (see Section 2). According to Lemma \ref{B-finite-type}, $E/\langle\sigma\rangle$ belongs to one of the following cases: $(a)$ a Brauer tree; $(b)$ a modified BG of free f-degree with a unique double half-edge, which has $p+1$ vertices and $2p+1$ half-edges ($p\geq 0$). For case $(a)$, choose a half-edge $h$ of $E/\langle\sigma\rangle$ which belongs to the unique exceptional vertex of $E/\langle\sigma\rangle$; and for case $(b)$, choose $h$ to be the unique double half-edge of $E/\langle\sigma\rangle$. Then for case $(a)$ we have $\Pi_{m}(E/\langle\sigma\rangle,h)=\langle x\rangle\cong\mathbb{Z}$, where $x=[(h|g^{l}|h)]$ and $l$ is the cardinality of the $G$-orbit of $h$, and for case $(b)$ we have $\Pi_{m}(E/\langle\sigma\rangle,h)=\langle x,y\rangle\cong\mathbb{Z}\oplus\mathbb{Z}/2\mathbb{Z}$, where $x=[(h|g^{l}|h)]$ with $l$ the cardinality of the $G$-orbit of $h$ and $y=[(h|\tau|h)]$ (Here $\tau$ denotes the involution of the modified BG $E/\langle\sigma\rangle$).

For case $(a)$, suppose that the f-degree of the exceptional vertex of $E/\langle\sigma\rangle$ is $m$, and suppose that $B=B_{(E/\langle\sigma\rangle,h)}$ is a Brauer tree with free f-degree given in \cite[Example 3.40]{LL2}. Let $n$ be the number of edges of $B$. Then $E/\langle\sigma\rangle$ has $\frac{n}{m}$ edges. There exists a covering $p:B\rightarrow E/\langle\sigma\rangle$ such that the image of the fundamental group of $B$ in $\Pi(E/\langle\sigma\rangle,h)\cong\Pi_{m}(E/\langle\sigma\rangle,h)$ is $\langle x^m\rangle$. Choose a half-edge $e$ of $B$ with $p(e)=h$, then the pair $(B,e)$ defines a Brauer relation of $\mathcal{B}$ order $n$ (cf. the remarks before Proposition \ref{to be a Brauer relation}). Let $\mathscr{C}=\mathscr{C}_{\mathcal{B}}$ be the configuration of $\mathbb{Z}A_n$ corresponding to $\mathcal{B}$ (see Proposition \ref{a one-to-one correspond}). Suppose the image of the fundamental group of $E$ in $\Pi(E/\langle\sigma\rangle,h)\cong\Pi_{m}(E/\langle\sigma\rangle,h)$ via the homomorphism induced by the covering $E\rightarrow E/\langle\sigma\rangle$ is $\langle x^r\rangle$, then we have

\begin{Thm}\label{case (a): AR-quiver of A_E}
For case $(a)$, the configuration $\mathscr{C}$ of $\mathbb{Z}A_n$ is $\tau^{\frac{n}{m}}$-stable and the AR-quiver $\Gamma_{\Lambda_E}$ of $\Lambda_E$ is isomorphic to $(\mathbb{Z}A_n)_{\mathscr{C}}/\langle\tau^{\frac{nr}{m}}\rangle$, where $\tau$ denotes the automorphism of $\mathbb{Z}A_n$ induced from the translation of the translation quiver $\mathbb{Z}A_n$ and the positive integers $n$, $m$, $r$ are defined as above.
\end{Thm}

For case $(b)$, let $B=R_E=\widehat{E/\langle\sigma\rangle}$ and $e=h_1\in (E/\langle\sigma\rangle)_{1}$. Then $B$ is a Brauer tree with free f-degree, which has $n=2p+1$ edges. Let $\mathcal{B}$ be the Brauer relation given by the pair $(B,e)$, and $\mathscr{C}=\mathscr{C}_{\mathcal{B}}$ be the configuration of $\mathbb{Z}A_n$ corresponding to $\mathcal{B}$. According to Theorem \ref{fundamental group of rep-finite f-BGA}, the fundamental group of $E$ is isomorphic to $\mathbb{Z}$. Moreover, according to the proof of Theorem \ref{fundamental group of rep-finite f-BGA}, the image of the fundamental group of $E$ in $\Pi(E/\langle\sigma\rangle,h)\cong\Pi_{m}(E/\langle\sigma\rangle,h)=\langle x,y\rangle\cong\mathbb{Z}\oplus\mathbb{Z}/2\mathbb{Z}$ via the homomorphism induced by the covering $E\rightarrow E/\langle\sigma\rangle$ is generated by $x^r y$ for some $r\in \mathbb{Z}_{>0}$. We have

\begin{Thm}\label{case (b): AR-quiver of A_E}
For case $(b)$, denote by $\phi$ the involution of $\mathbb{Z}A_n$ given by $\phi(i,j)=(i+j-\frac{n+1}{2},n+1-j)$, then the configuration $\mathscr{C}$ defined above is $\phi$-stable, and the AR-quiver $\Gamma_{\Lambda_E}$ of $\Lambda_E$ is isomorphic to $(\mathbb{Z}A_n)_{\mathscr{C}}/\langle\tau^{nr}\phi\rangle$, where $\tau$ denotes the automorphism of $(\mathbb{Z}A_n)_{\mathscr{C}}$ induced from the translation of the translation quiver $\mathbb{Z}A_n$.
\end{Thm}

\subsection{Proofs of the main statements.}
\

Let $B=(B,P,L,d)$ be an f-degree-free Brauer tree which has $n$ edges and fix $e\in B$, $\mathcal{B}$ be the Brauer relation given by the pair $(B,e)$, and $\mathscr{C}=\mathscr{C}_{\mathcal{B}}$ be the configuration of $\mathbb{Z}A_n$ which corresponds to $\mathcal{B}$. Let $\widetilde{Q}=\widetilde{Q}_{\mathscr{C}}$ be the quiver given by the configuration $\mathscr{C}$ (see the remarks before Theorem \ref{quiver and relation}). Denote by $\tau$ the involution of $B$ as a Brauer $G$-set. For each edge $P(h)$ of $B$, there exists a unique number $i\in\{0,1,\cdots,n-1\}$ such that $P(h)=P((\tau g)^{2i}(e))$, and we define $f(P(h))=i$.

Let $\mathscr{X}$ be a subset of $B$ such that $\mathscr{X}$ meets each vertex of $B$ in exactly one half-edge, and such that for any $h\in\mathscr{X}$, $f(P(h))\leq f(P(g^k\cdot h))$ for every integer $k$. Define an $f_{ms}$-BG $\mathbb{Z}B=(\mathbb{Z}B,\widetilde{P},\widetilde{L},\widetilde{d})$ as follows: $\mathbb{Z}B=\{(h,k)\mid h\in B,k\in\mathbb{Z}\}$, where the $G$-set structure of $\mathbb{Z}B$ is given by
\begin{equation*}
g\cdot(h,k)=\begin{cases}
(g\cdot h,k), &\text{ if } g\cdot h\notin\mathscr{X}; \\
(g\cdot h,k+1), &\text{ if } g\cdot h\in\mathscr{X};
\end{cases}
\end{equation*}
$\widetilde{P}(h,k)=\{(h',k)\mid h'\in P(h)\}$, $\widetilde{L}(h,k)=\{(h,k)\}$, $\widetilde{d}(h,k)=d(h)$ for every $(h,k)\in\mathbb{Z}B$. It is easy to show that $\mathbb{Z}B$ is an universal cover of $B$.

Define $f:(Q_{\mathbb{Z}B})_0\rightarrow\mathbb{Z}$ as the function given by $f(\widetilde{P}(h,k))=f(P(h))+nk$ for each edge $\widetilde{P}(h,k)$ of $\mathbb{Z}B$. Note that for each half-edge $(h,k)$ of $\mathbb{Z}B$, we have $1\leq f(\widetilde{P}(g\cdot(h,k)))-f(\widetilde{P}(h,k))\leq n$. For each half-edge $(h,k)$ of $\mathbb{Z}B$, call the arrow $\widetilde{L}(h,k)$ of $Q_{\mathbb{Z}B}$ a {\it $\beta$-arrow} (resp. an {\it $\alpha$-arrow}) of $Q_{\mathbb{Z}B}$ if $L(h)$ is a $\beta$-arrow (resp. an $\alpha$-arrow) of $Q_{B}$ (cf. the remarks before Proposition \ref{to be a Brauer relation}).

\begin{Lem}\label{values}
Let $(h,k)\in\mathbb{Z}B$ with $\widetilde{L}(h,k)$ a $\beta$-arrow of $Q_{\mathbb{Z}B}$. Then $f(\widetilde{P}((\tau g)^{2}(h,k)))=f(\widetilde{P}(h,k))+n+1$, where $\tau$ denotes the involution of $\mathbb{Z}B$ as a Brauer $G$-set.
\end{Lem}

\begin{proof}
Note that $1\leq f(\widetilde{P}(g\cdot(h,k)))-f(\widetilde{P}(h,k))\leq n$ and $1\leq f(\widetilde{P}(g\tau g(h,k)))-f(\widetilde{P}(\tau g(h,k)))\leq n$. Since $\widetilde{P}(g\cdot(h,k))=\widetilde{P}(\tau g(h,k))$ and $\widetilde{P}(g\tau g(h,k))=\widetilde{P}((\tau g)^{2}(h,k))$, $2\leq f(\widetilde{P}((\tau g)^{2}(h,k)))-f(\widetilde{P}(h,k))\leq 2n$. Since $\widetilde{L}(h,k)$ a $\beta$-arrow of $Q_{\mathbb{Z}B}$, $L(h)$ a $\beta$-arrow of $Q_{B}$. Then $f(P((\tau g)^{2}(h)))-f(P(h))\equiv 1$ $($mod $n)$. Since $f(\widetilde{P}((\tau g)^{2}(h,k)))-f(\widetilde{P}(h,k))\equiv f(P((\tau g)^{2}(h)))-f(P(h))$ $($mod $n)$, we have $f(\widetilde{P}((\tau g)^{2}(h,k)))-f(\widetilde{P}(h,k))=n+1$.
\end{proof}

\begin{Lem}\label{quiver isomorphism}
There exists an isomorphism of quivers $Q_{\mathbb{Z}B}\rightarrow \widetilde{Q}$ which maps each vertex $v$ of $Q_{\mathbb{Z}B}$ to the vertex $f(v)$ of $\widetilde{Q}$.
\end{Lem}

\begin{proof}
For each integer $r$, there exists a unique half-edge $(h,k)$ of $\mathbb{Z}B$ such that $\widetilde{L}(h,k)$ is a $\beta$-arrow of $Q_{\mathbb{Z}B}$ and $f(\widetilde{P}(h,k))=r$. Define $f(\widetilde{L}(h,k))=\beta_r$ and $f(\widetilde{L}(\tau(h),k))=\alpha_r$. We need to show that $f$ defines a quiver isomorphism $Q_{\mathbb{Z}B}\rightarrow \widetilde{Q}$.

For each arrow $\gamma$ of $Q_{\mathbb{Z}B}$, by the definition of $f$ we have $f(s(\gamma))=s(f(\gamma))$. To show that $f:Q_{\mathbb{Z}B}\rightarrow \widetilde{Q}$ is a morphism of quivers, it suffices to show that $f(t(\gamma))=t(f(\gamma))$.

Since $1\leq f(\widetilde{P}(g\cdot(h,k)))-f(\widetilde{P}(h,k))\leq n$ for each half-edge $(h,k)$ of $\mathbb{Z}B$, we have
\begin{equation}\label{equation 1}
f(\widetilde{P}(g\cdot(h,k)))-f(\widetilde{P}(h,k))= \begin{cases}
f(P(g\cdot h))-f(P(h)), &\text{ if } f(P(g\cdot h))>f(P(h)); \\
f(P(g\cdot h))-f(P(h))+n, &\text{ if } f(P(g\cdot h))\leq f(P(h)).
\end{cases}
\end{equation}

For each arrow $\gamma=\widetilde{L}(h,k)$ of $Q_{\mathbb{Z}B}$, if $f(\widetilde{L}(h,k))=\beta_r$, then $L(h)$ is a $\beta$-arrow of $Q_B$ and $f(P(h))+nk=f(\widetilde{P}(h,k))=r$. $f$ maps the terminal $\widetilde{P}(g\cdot(h,k))$ of $\widetilde{L}(h,k)$ to $f(\widetilde{P}(g\cdot(h,k)))$. By Equation (\ref{equation 1}) we have
\begin{equation*}
f(\widetilde{P}(g\cdot(h,k)))-r= \begin{cases}
f(P(g\cdot h))-f(P(h)), &\text{ if } f(P(g\cdot h))>f(P(h)); \\
f(P(g\cdot h))-f(P(h))+n, &\text{ if } f(P(g\cdot h))\leq f(P(h)).
\end{cases}
\end{equation*}
Since $L(h)$ is a $\beta$-arrow of $Q_B$, $e^{\frac{2\pi f(P(g\cdot h))}{n}i}$ is the successor of $e^{\frac{2\pi f(P(h))}{n}i}$ in the Brauer relation $\mathscr{B}$, so $f(P(g\cdot h))-f(P(h))\equiv \beta(f(P(h)))-f(P(h))=\beta(r)-r$ $($mod $n)$. Since $1\leq \beta(r)-r\leq n$,
\begin{equation*}
\beta(r)-r= \begin{cases}
f(P(g\cdot h))-f(P(h)), &\text{ if } f(P(g\cdot h))>f(P(h)); \\
f(P(g\cdot h))-f(P(h))+n, &\text{ if } f(P(g\cdot h))\leq f(P(h)).
\end{cases}
\end{equation*} So we have $f(\widetilde{P}(g\cdot(h,k)))=\beta(r)$, which is the terminal of $\beta_r$ in $\widetilde{Q}$.

If $f(\widetilde{L}(h,k))=\alpha_r$, then $\widetilde{L}(h,k)$ is an $\alpha$-arrow of $Q_{\mathbb{Z}B}$. Therefore $\widetilde{L}(\tau(h,k))$ and $\widetilde{L}(g^{-1}\tau(h,k))$ are $\beta$-arrows of $Q_{\mathbb{Z}B}$. According to Lemma \ref{values}, $f(\widetilde{P}(g\cdot(h,k)))-f(\widetilde{P}(g^{-1}\cdot(\tau(h),k)))=n+1$. Suppose that $f(\widetilde{P}(g^{-1}\cdot(\tau(h),k)))=t$, then $f(\widetilde{L}(g^{-1}(\tau(h),k)))=\beta_t$, and therefore $r=f(\widetilde{P}(h,k))=f(\widetilde{P}(\tau(h),k))=\beta(t)$, where the last identity follows from last paragraph. So $f(\widetilde{P}(g\cdot(h,k)))=n+1+t=\alpha\beta(t)=\alpha(r)$, which is the terminal of $\alpha_r$ in $\widetilde{Q}$.

By the arguments above, $f:Q_{\mathbb{Z}B}\rightarrow \widetilde{Q}$ is a morphism of quivers. Clearly $f$ is an isomorphism.
\end{proof}

Let $\widetilde{\Lambda}$ be the $k$-category $k\widetilde{Q}/\widetilde{I}$, where $\widetilde{I}$ is the ideal of $k\widetilde{Q}$ generated by the following relations: \\$(a)$ $\beta_{\alpha(r)}\alpha_{r}=\alpha_{\beta(r)}\beta_{r}=0$; \\ $(b)$ $\alpha_{\alpha^{a_r-1}(r)}\cdots\alpha_{\alpha(r)}\alpha_{r}
=\beta_{\beta^{b_r-1}(r)}\cdots\beta_{\beta(r)}\beta_{r}$, where $r\in\mathbb{Z}$ and $a_r$, $b_r$ are defined by $\alpha^{a_r}(r)=r+n=\beta^{b_r}(r)$. \\ If $\Pi$ is an admissible automorphism group of $\mathbb{Z}A_n$ which stabilizes $\mathscr{C}$, then each $\pi\in\Pi$ induces an automorphism of $\widetilde{Q}$ (see the remarks before Theorem \ref{quiver and relation}), which also induces a $k$-linear automorphism $\widetilde{\pi}$ of $\widetilde{\Lambda}$. Denote by $\widetilde{\Pi}$ the group of automorphisms of $\widetilde{\Lambda}$ formed by $\widetilde{\pi}$ with $\pi\in\Pi$. Since $\Pi$ is admissible, it acts freely on $\mathscr{C}$, therefore $\widetilde{\Pi}$ acts freely on $\widetilde{\Lambda}$. Let $\Lambda=\widetilde{\Lambda}/\widetilde{\Pi}$ be the quotient category (see \cite[Section 3]{G}) and $A=\bigoplus_{x,y\in\Lambda}\Lambda(x,y)$. By Theorem \ref{quiver and relation}, $A$ is isomorphic to $A_{\mathscr{C},\Pi}$.

\begin{Lem}\label{category isomorphism}
The quiver isomorphism $Q_{\mathbb{Z}B}\rightarrow \widetilde{Q}$ in Lemma \ref{quiver isomorphism} induces an isomorphism $\Lambda_{\mathbb{Z}B}\rightarrow\widetilde{\Lambda}$ of $k$-categories.
\end{Lem}

\begin{proof}
By definition $\Lambda_{\mathbb{Z}B}=kQ_{\mathbb{Z}B}/I_{\mathbb{Z}B}$, where $I_{\mathbb{Z}B}$ is generated by the following relations: \\ $(a')$ $\widetilde{L}(\tau g(h,k))\widetilde{L}(h,k)=0$; \\ $(b')$ $\widetilde{L}(g^{d(h)-1}\cdot(h,k))\cdots \widetilde{L}(g\cdot(h,k))\widetilde{L}(h,k)=\widetilde{L}(g^{d(\tau(h))-1}\cdot(\tau(h),k))\cdots \widetilde{L}(g\cdot(\tau(h),k))\widetilde{L}(\tau(h),k)$. \\ Since for each $(h,k)\in \mathbb{Z}B$, $\widetilde{L}(h,k)$ is an $\alpha$-arrow (resp. a $\beta$-arrow) of $Q_{\mathbb{Z}B}$ if and only if $\widetilde{L}(\tau g(h,k))$ is a $\beta$-arrow (resp. an $\alpha$-arrow) of $Q_{\mathbb{Z}B}$, the quiver isomorphism $Q_{\mathbb{Z}B}\rightarrow \widetilde{Q}$ maps the relations of type $(a')$ in $I_{\mathbb{Z}B}$ to the relations of type $(a)$ in $\widetilde{I}$. Since $f(\widetilde{P}(g^{d(h)}\cdot(h,k)))=f(\widetilde{P}(h,k+1))=f(\widetilde{P}(h,k))+n$, we see that the quiver isomorphism $Q_{\mathbb{Z}B}\rightarrow \widetilde{Q}$ maps the relations of type $(b')$ in $I_{\mathbb{Z}B}$ to the relations of type $(b)$ in $\widetilde{I}$.
\end{proof}

\begin{proof}[{\bf Proof of Theorem \ref{case (a): AR-quiver of A_E}}]
\

\medskip
{\it Step 1: To show that the configuration $\mathscr{C}$ in Theorem \ref{case (a): AR-quiver of A_E} is $\tau^{\frac{n}{m}}$-stable.}

Since the Brauer tree $E/\langle\sigma\rangle$ has $\frac{n}{m}$ edges, we have $(\tau g)^{\frac{2n}{m}}(h)=h$. Then $p((\tau g)^{\frac{2n}{m}}(e))=(\tau g)^{\frac{2n}{m}}(p(e))=(\tau g)^{\frac{2n}{m}}(h)=h=p(e)$. Since the fundamental group $\Pi(E/\langle\sigma\rangle,h)\cong\mathbb{Z}$ of $E/\langle\sigma\rangle$ is abelian, the covering $p:B\rightarrow E/\langle\sigma\rangle$ is regular. Therefore there exists an automorphism $\nu$ of $B$ such that $\nu(e)=(\tau g)^{\frac{2n}{m}}(e)$. Since each half-edge $b$ of $B$ is of the form $(\tau g)^{i}(e)$ for some integer $i$, we have $\nu(b)=(\tau g)^{\frac{2n}{m}}(b)$ for every $b\in B$. For any $i\in\mathbb{Z}$, let $j$ be the unique integer in $[1,n]$ such that $g(\tau g)^{2i}(e)=(\tau g)^{2(i+j)}(e)$. Then $(i,j)$ is the unique vertex of $\mathbb{Z}A_n$ in $\Delta_i\cap\mathscr{C}$. Since $g(\tau g)^{2(i+\frac{n}{m})}(e)=g(\tau g)^{2i}(\nu(e))=\nu(g(\tau g)^{2i}(e))=\nu((\tau g)^{2(i+j)}(e))=(\tau g)^{2(i+j+\frac{n}{m})}(e)$, $(i+\frac{n}{m},j)\in\mathscr{C}$, and $\mathscr{C}$ is $\tau^{\frac{n}{m}}$-stable.

\medskip
{\it Step 2: To show that $\Lambda_E$ is isomorphic to the $k$-category $\widetilde{\Lambda}/\widetilde{\Pi}$, where $\widetilde{\Pi}$ is induced by the group of automorphisms $\Pi$ of $\mathbb{Z}A_n$ generated by $\tau^{\frac{nr}{m}}$.}

Since the image of the fundamental group of $E$ in $\Pi(E/\langle\sigma\rangle,h)\cong\Pi_{m}(E/\langle\sigma\rangle,h)$ via the homomorphism induced by the covering $E\rightarrow E/\langle\sigma\rangle$ is $\langle x^r\rangle$, where $x=[(h|g^{l}|h)]$ and $l$ is the cardinality of the $G$-orbit of $h$, $E$ is isomorphic to $\mathbb{Z}B/\langle\mu\rangle$, where $\mu$ is the automorphisms of $\mathbb{Z}B$ such that $\mu(e,0)=g^{lr}(e,0)$. Let $u$ be the automorphism of $Q_{\mathbb{Z}B}$ induced by $\mu$. Then $\Lambda_{E}$ is isomorphic to $\Lambda_{\mathbb{Z}B}/\langle\widetilde{u}\rangle$, where $\widetilde{u}$ denotes the automorphism of $\Lambda_{\mathbb{Z}B}$ induced by $u$. Let $f:Q_{\mathbb{Z}B}\rightarrow\widetilde{Q}$ be the isomorphism of quivers in Lemma \ref{quiver isomorphism}. To show that $\Lambda_E$ is isomorphic to $\widetilde{\Lambda}/\widetilde{\Pi}$, it suffices to show the diagram
$$\xymatrix@R=2pc@C=2pc {
	Q_{\mathbb{Z}B}\ar[r]^{f}\ar[d]_{u} & \widetilde{Q}\ar[d]^{v} \\
    Q_{\mathbb{Z}B}\ar[r]^{f} & \widetilde{Q} \\
}$$
commutes, where $v$ is the automorphism of $\widetilde{Q}$ given by $v(i)=i+\frac{nr}{m}$ for any $i\in\widetilde{Q}_{0}$ and $v(\alpha_i)=\alpha_{i+\frac{nr}{m}}$, $v(\beta_i)=\beta_{i+\frac{nr}{m}}$ for any $i\in\mathbb{Z}$.

Let $r=am+b$ with $a,b\in\mathbb{N}$ and $0\leq b<m$. Then $g^{lr}\cdot(e,0)=g^{lb}g^{lma}\cdot(e,0)=g^{lb}\cdot(e,a)$. It can be shown that $f(P(g^{lb}\cdot e))=\frac{bn}{m}$, so $f(\widetilde{P}(g^{lr}\cdot(e,0)))=f(\widetilde{P}(g^{lb}\cdot(e,a)))=an+\frac{bn}{m}=n(a+\frac{b}{m})=\frac{nr}{m}$. For each vertex $\widetilde{P}(x,k)$ of $Q_{\mathbb{Z}B}$ with $\widetilde{L}(x,k)$ a $\beta$-arrow of $Q_{\mathbb{Z}B}$, assume that $f(\widetilde{P}(x,k))=i$. Denote by $\sigma$ the Nakayama automorphism of $\mathbb{Z}B$. Since $\widetilde{L}((\sigma^{-1}(\tau g)^2)^{i}(e,0))$ is a $\beta$-arrow of $Q_{\mathbb{Z}B}$ with $f(\widetilde{P}((\sigma^{-1}(\tau g)^2)^{i}(e,0)))=i$ by Lemma \ref{values}, we have $(\sigma^{-1}(\tau g)^2)^{i}(e,0)=(x,k)$. So $u(\widetilde{P}(x,k))=\widetilde{P}(\mu(x,k))=\widetilde{P}(\mu((\sigma^{-1}(\tau g)^2)^{i}(e,0)))=\widetilde{P}((\sigma^{-1}(\tau g)^2)^{i}\mu(e,0))=\widetilde{P}((\sigma^{-1}(\tau g)^2)^{i}g^{lr}(e,0))$. Since $\widetilde{L}(g^{lr}\cdot(e,0))$ is a $\beta$-arrow of $Q_{\mathbb{Z}B}$, by Lemma \ref{values}, $f(\widetilde{P}((\sigma^{-1}(\tau g)^2)^{i}g^{lr}(e,0)))=f(\widetilde{P}(g^{lr}\cdot(e,0)))+i=\frac{nr}{m}+i=v(i)$. Therefore $fu(\widetilde{P}(x,k))=vf(\widetilde{P}(x,k))$. Similarly we have $fu(\widetilde{L}(x,k))=vf(\widetilde{L}(x,k))$ and $fu(\widetilde{L}(\tau(x,k)))=vf(\widetilde{L}(\tau(x,k)))$, so the diagram above is commutative.

\medskip
Let $A=\bigoplus_{x,y\in\widetilde{\Lambda}/\widetilde{\Pi}}(\widetilde{\Lambda}/\widetilde{\Pi})(x,y)$. According to Theorem \ref{quiver and relation}, $A\cong A_{\mathscr{C},\Pi}$, where $\Pi$ is generated by the automorphism $\tau^{\frac{nr}{m}}$ of $\mathbb{Z}A_n$. By Step 2, $\Lambda_E$ is isomorphic to $\widetilde{\Lambda}/\widetilde{\Pi}$, so $\Gamma_{\Lambda_E}\cong\Gamma_{A}\cong(\mathbb{Z}A_n)_{\mathscr{C}}/\langle\tau^{\frac{nr}{m}}\rangle$.

\end{proof}

\begin{proof}[{\bf Proof of Theorem \ref{case (b): AR-quiver of A_E}}]
\

\medskip
{\it Step 1: To show that the configuration $\mathscr{C}$ in Theorem \ref{case (b): AR-quiver of A_E} is $\phi$-stable, where $\phi$ is the involution of $\mathbb{Z}A_n$ given by $\phi(i,j)=(i+j-\frac{n+1}{2},n+1-j)$}.

Recall that $n=2p+1$. We need to show that for any $(i,j)\in\mathscr{C}$, the vertex $(i+j-p-1,n+1-j)$ of $\mathbb{Z}A_n$ also belongs to $\mathscr{C}$.

Let $\iota$ be the automorphism of $B=\widehat{E/\langle\sigma\rangle}$ which maps $x_{i}\in(E/\langle\sigma\rangle)_{i}$ to $x_{3-i}\in(E/\langle\sigma\rangle)_{3-i}$ for each $x\in E/\langle\sigma\rangle$ and $i=1,2$. Denote by $\tau$ the involution of $B$ as a Brauer $G$-set. Since $\iota(e)=\tau(e)=(\tau g)^{n}(e)$ and since each half-edge $b$ of $B$ is of the form $(\tau g)^{i}(e)$ for some integer $i$, we have $\iota(b)=(\tau g)^{n}(b)$ for any $b\in B$. Since $(i,j)\in\mathscr{C}$, we have $g(\tau g)^{2i}(e)=(\tau g)^{2(i+j)}(e)$. Therefore
$$(\tau g)^{2(i+j)+n}(e)=\iota((\tau g)^{2(i+j)}(e))=\iota(g(\tau g)^{2i}(e))=g\cdot\iota((\tau g)^{2i}(e))=g(\tau g)^{2i+n}(e)$$
and
$$\tau(\tau g)^{2(i+j)+n}(e)=(\tau g)^{2i+n+1}(e)=(\tau g)^{2(i+p+1)}(e).$$
Since
$$\tau(\tau g)^{2(i+j)+n}(e)=g(\tau g)^{2(i+j)+n-1}(e)=g(\tau g)^{2(i+j+p)}(e),$$
$\beta_{\mathcal{B}}(e^{\frac{2(i+j+p)\pi}{n}i})=e^{\frac{2(i+p+1)\pi}{n}i}$, where $\beta_{\mathcal{B}}$ is the permutation of $\sqrt[n]{1}$ assigning to each point $s$ its successor in the equivalence class of $s$ endowed with the anti-clockwise orientation (see the remarks after Definition \ref{Brauer relation}). Since $(i+p+1)-(i+j+p)=1-j\equiv n+1-j$ $($mod $n)$ and $1\leq n+1-j\leq n$, we have $(i+j+p,n+1-j)\in\mathscr{C}$. Since $\mathscr{C}$ is $\tau^{n}$-stable, $(i+j-p-1,n+1-j)$ also belongs to $\mathscr{C}$.

\medskip
{\it Step 2: To show that the automorphism $\psi$ of $\mathbb{Z}B$ which maps $(e,0)$ to $(\tau(e),0)$ induces an automorphism $u$ of $Q_{\mathbb{Z}B}$ such that the diagram
$$\xymatrix@R=2pc@C=2pc {
	Q_{\mathbb{Z}B}\ar[r]^{f}\ar[d]_{u} & \widetilde{Q}\ar[d]^{v} \\
    Q_{\mathbb{Z}B}\ar[r]^{f} & \widetilde{Q} \\
}$$
$$\mbox{Figure 2: The main diagram in Step 2}$$
commutes, where $f$ is the isomorphism of quivers in Lemma \ref{quiver isomorphism} and $v$ is the automorphism of $\widetilde{Q}$ induced by the automorphism $\phi$ of $(\mathbb{Z}A_n)_{\mathscr{C}}$ in Theorem \ref{case (b): AR-quiver of A_E} (see the paragraph before Theorem \ref{quiver and relation}).}

For each $r\in\widetilde{Q}_{0}$, let $(i,j)$ be the unique vertex of $\mathbb{Z}A_n$ which belongs to $\nabla_r\cap\mathscr{C}$, and let $(x,k)\in\mathbb{Z}B$ such that $f(\widetilde{P}(x,k))=i$ and $\widetilde{L}(x,k)$ is a $\beta$-arrow of $Q_{\mathbb{Z}B}$. By the definition of $f$, $f(\widetilde{L}(x,k))=\beta_i$, so $f(\widetilde{P}(g\cdot(x,k)))=\beta(i)=r$. Denote by $\sigma$ the Nakayama automorphism of $\mathbb{Z}B$. By Lemma \ref{values}, $f(\widetilde{P}((\sigma^{-1}(\tau g)^{2})^{i}(e,0)))=i$. Since both $\widetilde{L}(x,k)$ and $\widetilde{L}((\sigma^{-1}(\tau g)^{2})^{i}(e,0))$ are $\beta$-arrows of $Q_{\mathbb{Z}B}$ and since $f(\widetilde{P}(x,k))=f(\widetilde{P}((\sigma^{-1}(\tau g)^{2})^{i}(e,0)))=i$, we have $(x,k)=(\sigma^{-1}(\tau g)^{2})^{i}(e,0)$. Then $g\cdot(x,k)=g(\sigma^{-1}(\tau g)^{2})^{i}(e,0)$ and $\psi(g(x,k))=g(\sigma^{-1}(\tau g)^{2})^{i}\psi(e,0)=g(\sigma^{-1}(\tau g)^{2})^{i}\tau(e,0)=\sigma^{-i}(g\tau)^{2i+1}(e,0)$. Therefore \begin{multline*}u(\widetilde{P}(g\cdot(x,k)))=\widetilde{P}(\psi(g\cdot(x,k)))=\widetilde{P}(\sigma^{-i}(g\tau)^{2i+1}(e,0)) \\ =\widetilde{P}(\tau\sigma^{-i}(g\tau)^{2i+1}(e,0))=\widetilde{P}(\sigma^{-i}(\tau g)^{2i}\tau g\tau(e,0)).\end{multline*}
Since $\tau g\tau(e,0)$ is a $\beta$-arrow of $Q_{\mathbb{Z}B}$, by Lemma \ref{values}, $f(\widetilde{P}(\sigma^{-i}(\tau g)^{2i}\tau g\tau(e,0)))=f(\widetilde{P}(\tau g\tau(e,0)))+i$. Since $\tau(e)=(\tau g)^{n}(e)$, $\tau g\tau(e)=(\tau g)^{n+1}(e)=(\tau g)^{2(p+1)}(e)$. So
$$f(\widetilde{P}(\tau g\tau(e,0)))\equiv f(P(\tau g\tau(e)))=p+1 \quad(\text{mod } n).$$
Since
$$f(\widetilde{P}(\tau g\tau(e,0)))=f(\widetilde{P}(\tau g\tau(e,0)))-f(\widetilde{P}((e,0)))=f(\widetilde{P}(g\tau(e,0)))-f(\widetilde{P}(\tau(e,0))),$$
$1\leq f(\widetilde{P}(\tau g\tau(e,0)))\leq n$. Therefore $f(\widetilde{P}(\tau g\tau(e,0)))=p+1$, and
$$fu(\widetilde{P}(g\cdot(x,k)))=f(\widetilde{P}(\sigma^{-i}(\tau g)^{2i}\tau g\tau(e,0)))=f(\widetilde{P}(\tau g\tau(e,0)))+i=p+1+i.$$
Since $\phi(i,j)=(i+j-p-1,n+1-j)$, $v(r)=(i+j-p-1)+(n+1-j)=i+p+1$. Then $fuf^{-1}(r)=i+p+1=v(r)$, and the above diagram is commutative on vertices. Since $u$ maps each $\alpha$-arrow (resp. $\beta$-arrow) of $Q_{\mathbb{Z}B}$ to a $\beta$-arrow (resp. an $\alpha$-arrow) of $Q_{\mathbb{Z}B}$ and $v$ maps each $\alpha$-arrow (resp. $\beta$-arrow) of $\widetilde{Q}$ to a $\beta$-arrow (resp. an $\alpha$-arrow) of $\widetilde{Q}$, we see that the above diagram is also commutative on arrows.

\medskip
{\it Step 3: To show that the AR-quiver $\Gamma_{\Lambda_E}$ of $\Lambda_E$ is isomorphic to $(\mathbb{Z}A_n)_{\mathscr{C}}/\langle\tau^{nr}\phi\rangle$, where $\phi$ is the automorphism of $(\mathbb{Z}A_n)_{\mathscr{C}}$ which induces an involution $(i,j)\mapsto(i+j-\frac{n+1}{2},n+1-j)$} of $\mathbb{Z}A_n$.

Let $\psi$ be the automorphism of $\mathbb{Z}B$ which maps $(e,0)$ to $(\tau(e),0)$. Since the image of the fundamental group of $E$ in $\Pi(E/\langle\sigma\rangle,h)\cong\Pi_{m}(E/\langle\sigma\rangle,h)$ is generated by $x^r y$, according to Proposition \ref{modified existence of morphism}, $E$ is isomorphic to $\mathbb{Z}B/\langle\sigma^{r}\psi\rangle$ (here $\sigma$ denotes the Nakayama automorphism of $\mathbb{Z}B$). So $\Lambda_E$ is isomorphic to $\Lambda_{\mathbb{Z}B}/\Pi$, where $\Pi$ is a group of automorphisms of $\Lambda_{\mathbb{Z}B}$ generated by the automorphism of $\Lambda_{\mathbb{Z}B}$ which is induced by the automorphism $wu$ of $Q_{\mathbb{Z}B}$, where $w$ (resp. $u$) is the automorphism of $Q_{\mathbb{Z}B}$ induced by the automorphism $\sigma^{r}$ (resp. $\psi$) of $\mathbb{Z}B$. Let $t$ (resp. $v$) be the automorphism of $\widetilde{Q}$ induced by the automorphism $\tau^{nr}$ (resp. $\phi$) of $(\mathbb{Z}A_n)_{\mathscr{C}}$. Then $t(i)=i+nr$ and $t(\alpha_i)=\alpha_{i+nr}$, $t(\beta_i)=\beta_{i+nr}$ for each $i\in\mathbb{Z}$. We have a commutative diagram
$$\xymatrix@R=2pc@C=2pc {
	Q_{\mathbb{Z}B}\ar[r]^{f}\ar[d]_{w} & \widetilde{Q}\ar[d]^{t} \\
    Q_{\mathbb{Z}B}\ar[r]^{f} & \widetilde{Q} \\
},$$
where $f$ is the isomorphism of quivers in Lemma \ref{quiver isomorphism}. Combining Figure $2$ in Step 2, we have a commutative diagram
$$\xymatrix@R=2pc@C=2pc {
	Q_{\mathbb{Z}B}\ar[r]^{f}\ar[d]_{wu} & \widetilde{Q}\ar[d]^{tv} \\
    Q_{\mathbb{Z}B}\ar[r]^{f} & \widetilde{Q} \\
}.$$
So the quiver isomorphism $f:Q_{\mathbb{Z}B}\rightarrow\widetilde{Q}$ induces an isomorphism $\Lambda_{\mathbb{Z}B}/\Pi\rightarrow\widetilde{\Lambda}/\widetilde{\Pi}$ of $k$-categories, where $\widetilde{\Pi}$ is generated by the automorphism of $\widetilde{\Lambda}$ induced by the automorphism $tv$ of $\widetilde{Q}$. Let $A=\bigoplus_{x,y\in\widetilde{\Lambda}/\widetilde{\Pi}}(\widetilde{\Lambda}/\widetilde{\Pi})(x,y)$. According to Theorem \ref{quiver and relation}, $A\cong A_{\mathscr{C},\langle\tau^{nr}\phi\rangle}$. Therefore the AR-quiver $\Gamma_{\Lambda_E}$ of $\Lambda_E\cong\Lambda_{\mathbb{Z}B}/\Pi$ is isomorphic to $(\mathbb{Z}A_n)_{\mathscr{C}}/\langle\tau^{nr}\phi\rangle$.
\end{proof}

\section{Domestic fractional Brauer graph algebras of type MS}

In this section we assume that $k$ is an algebraically closed field. For a finite dimensional self-injective $k$-algebra $A$, denote by $\prescript{}{s}{\Gamma}_{A}$ the stable AR-quiver of $A$.

\subsection{Exceptional tubes of representation-infinite  $f_{ms}$-BGAs}
\

Let $E=(E,P,L,d)$ be a finite connected $f_{ms}$-BG with $A_E$ representation-infinite. According to \cite[Section 6]{LL}, $A_E\cong kQ'_E/I'_E$ with $I'_E$ admissible, where $Q'_E$ is the subquiver of $Q_E$ given by $(Q'_E)_0=(Q_E)_0$ and $(Q'_E)_1=\{L(e)\mid e\in E $ with $d(e)>1\}$, and $I'_E$ is generated by the following three types of relations
\begin{itemize}
\item[$(fR1')$] $L(g^{d(e)-1}\cdot e)\cdots L(g\cdot e)L(e)-L(g^{d(h)-1}\cdot h)\cdots L(g\cdot h)L(h)$, where $e,h\in E$ and $d(e),d(h)>1$;
\item[$(fR2')$] $L(e_1)L(e_2)$, where $e_1,e_2\in E$, $d(e_1),d(e_2)>1$, and $e_1\neq g\cdot e_2$;
\item[$(fR3')$] $L(g^{d(e)}\cdot e)\cdots L(g\cdot e)L(e)$, where $e\in E$ and $d(e)>1$.
\end{itemize}

We will consider a module of $A_E$ as a representation of the quiver with relations $(Q'_E,I'_E)$. We call an $A_E$-module $M$ a string module if it can be seen as a string module for the quotient algebra $A_E/\mathrm{soc}(A_E)$ (for the definition of string and string modules, see \cite[Sections II.2 and II.3]{E}). Note that $A_E/\mathrm{soc}(A_E)$ is a string algebra, which is given by the quiver $Q'_E$ and the admissible ideal $I''_E$ of $kQ'_E$ generated by the following two types of relations
\begin{itemize}
\item[$(a)$] $L(g^{d(e)-1}\cdot e)\cdots L(g\cdot e)L(e)$, where $e\in E$ and $d(e)>1$;
\item[$(b)$] $L(e_1)L(e_2)$, where $e_1,e_2\in E$, $d(e),d(h)>1$, and $e_1\neq g\cdot e_2$.
\end{itemize}
We will use the results about the AR-sequences of a special biserial algebra described in \cite[Section II.6]{E} freely.

For every $e\in E$, denote by $M_e$ the uniserial $A_E$-module given by the direct string $L(g^{d(e)-2}\cdot e)\cdots L(g\cdot e)L(e)$ (define $M_e$ as the simple $A_E$-module at $P(e)$ if $d(e)=1$). Since $A_E$ representation-infinite, it is not a Nakayama algebra. Then it is straightforward to show that for each edge $P(e)=\{e,e'\}$ of $E$, either $d(e)>1$ or $d(e')>1$. Therefore $M_{e_1}$ is not isomorphic to $M_{e_2}$ for every $e_1\neq e_2$.

Recall that a connected component of $\prescript{}{s}{\Gamma}_{A_E}$ consisting of string modules of the form $\mathbb{Z}A_{\infty}/\langle\tau^n\rangle$ ($n\geq 1$) is called an {\it exceptional tube}. The following result should be compared with \cite[Lemma 4.4]{D} for BGAs.

\begin{Lem}\label{modules-at-the-mouth-of-exceptional-tubes}
Let $E=(E,P,L,d)$ be a finite connected $f_{ms}$-BG with $A_E$ representation-infinite. Then a string module $M$ of $A_E$ is at the mouth of an exceptional tube in the stable AR-quiver of $A_E$ if and only if $M\cong M_e$ for some $e\in E$.
\end{Lem}

\begin{proof}
Denote by $C=A_E/\mathrm{soc}(A_E)$. Note that for each edge $P(e)=\{e,e'\}$ of $E$, either $d(e)>1$ or $d(e')>1$. By \cite[Theorem 2.1 and Theorem 2.2]{ES}, an indecomposable string module $M$ of $A_E$ is at the mouth of an exceptional tube in the stable AR-quiver of $A_E$ if and only if there exists precisely one irreducible morphism $N\rightarrow M$ in $\prescript{}{s}{\Gamma}_{A_E}$ for some indecomposable $A_E$-module $N$. According to the description of AR-sequences for string algebras in \cite[Section II.6]{E}, this happens if and only if $M$ is a uniserial projective module for $C$ or $M$ is of the form $Cu/C\beta$ for some arrow $\beta$ in $Q'_E$ with $u$ the idempotent of $C$ corresponding to $s(\beta)$. This is equivalent to the fact that $M\cong M_e$ for some $e\in E$.
\end{proof}

\begin{Lem}\label{DTr}
Let $E=(E,P,L,d)$ be a finite connected $f_{ms}$-BG with $A_E$ representation-infinite. Denote by $\tau$ the involution of $E$ as a Brauer $G$-set. Then for every $e\in E$, $\mathrm{DTr}(M_e)\cong M_{\sigma^{-1}(g\tau)^2(e)}$, where $\mathrm{DTr}$ denotes the AR-translation of $A_E$ and  $\sigma$ denotes the Nakayama automorphism of $E$.
\end{Lem}

\begin{proof}
If $d(e),d(\tau(e))>1$, then there are two arrows $L(e), L(\tau(e))$ of $Q'_E$ starting at $P(e)$. There exists an AR-sequence $0\rightarrow M_{\sigma^{-1}(g\tau)^2(e)}\rightarrow N\rightarrow M_e\rightarrow 0$, where $N$ is the string module given by the string
$$L(g^{d(e)-2}\cdot e)\cdots L(g\cdot e)L(e)L(\tau(e))^{-1}L(g^{d(h)-2}\cdot h)\cdots L(g\cdot h)L(h)$$
with $h=\sigma^{-1}(g\tau)^2(e)$.

If $d(e)=1$ and $d(\tau(e))>1$, then $M_e$ is the simple $A_E$-module corresponding to the vertex $P(e)$ of $Q'_E$. There exists an AR-sequence $0\rightarrow M_{\sigma^{-1}(g\tau)^2(e)}\rightarrow N\rightarrow M_e\rightarrow 0$, where $N$ is the string module given by the string
$$L(\tau(e))^{-1}L(g^{d(h)-2}\cdot h)\cdots L(g\cdot h)L(h)$$
with $h=\sigma^{-1}(g\tau)^2(e)$.

If $d(\tau(e))=1$, then the projective cover $P$ of $M_e$ is uniserial and $M_e\cong P/\mathrm{soc}(P)$. So there exists an AR-sequence $0\rightarrow \mathrm{rad}(P)\rightarrow P\oplus(\mathrm{rad}(P)/\mathrm{soc}(P))\rightarrow M_e\rightarrow 0$, where rad$(P)\cong M_{g\cdot e}$. Since $d(\tau(e))=1$, $g\tau(e)=\sigma\tau(e)=\tau\sigma(e)$, and $\sigma^{-1}(g\tau)^2(e)=\sigma^{-1}(g\tau)(\tau\sigma)(e)=\sigma^{-1}g\sigma(e)=g\cdot e$. Then $\mathrm{DTr}(M_e)\cong \mathrm{rad}(P)\cong M_{\sigma^{-1}(g\tau)^2(e)}$.
\end{proof}

The following result should be compared with \cite[Theorem 4.5]{D} for BGAs.

\begin{Prop}\label{exceptional-tubes}
Let $E=(E,P,L,d)$ be a finite connected $f_{ms}$-BG with $A_E$ representation-infinite. Denote by $\tau$ the involution of $E$ as a Brauer $G$-set, $\sigma$ the Nakayama automorphism of $E$, and $\sigma^{-1}(g\tau)^2:E\rightarrow E$, $e\mapsto\sigma^{-1}(g\tau)^2(e)$ a permutation of $E$. Then
\begin{itemize}
\item[(1)] There is a bijection between the set of exceptional tubes in the stable AR-quiver $\prescript{}{s}{\Gamma}_{A_E}$ of $A_E$ and the set of $\langle\sigma^{-1}(g\tau)^2\rangle$-orbits of $E$.
\item[(2)] The rank of an exceptional tube of $\prescript{}{s}{\Gamma}_{A_E}$ is equal to the length of the associated $\langle\sigma^{-1}(g\tau)^2\rangle$-orbit of $E$.
\end{itemize}
\end{Prop}

\begin{proof}
According to Lemma \ref{modules-at-the-mouth-of-exceptional-tubes}, there is a bijection between $E$ and the set of string modules of $A_E$ at the mouth of an exceptional tube of $\prescript{}{s}{\Gamma}_{A_E}$. Moreover, by Lemma \ref{DTr}, the action of the AR-translation DTr on this set of modules corresponds to the permutation $\sigma^{-1}(g\tau)^2$ on $E$.
\end{proof}

\subsection{Construction of domestic  $f_{ms}$-BGAs}
\

Let $E$ be a finite connected $f_{ms}$-BG with $A_E$ domestic. By Theorem \ref{rep type of f-BCA and its reduced form}, $A_{R_E}$ is also domestic, where
\begin{equation*}
R_E= \begin{cases}
E/\langle\sigma\rangle, & \text{if } \langle\sigma\rangle \text{ is admissible}; \\
\widehat{E/\langle\sigma\rangle}, & \text{otherwise}.
\end{cases}
\end{equation*}
Suppose that the modified BG $E/\langle \sigma\rangle$ has $k$-edges, $l$ double half-edges, and $n$ vertices $v_1$, $\cdots$, $v_n$ of f-degree $d_1$, $\cdots$, $d_n$, respectively. By Lemma \ref{B}, there are three possible cases:
\begin{itemize}
\item[$(1)$] $l=2$, $k-n+1=0$, $d_i=1$ for $1\leq i\leq n$;
\item[$(2)$] $l=0$, $k-n+1=0$, $d_i=2$ for exactly two numbers $i=i_0$, $i_1$, and $d_i=1$ for $i\neq i_0$, $i_1$;
\item[$(3)$] $l=0$, $k-n+1=1$, $d_i=1$ for $1\leq i\leq n$.
\end{itemize}

Note that in case $(1)$ $E/\langle \sigma\rangle$ is an f-degree-free modified BG whose diagram is a tree after deleting two half-edges, in case $(2)$ $E/\langle \sigma\rangle$ is a BG whose diagram is a tree with two vertices multiplicity $2$ and others multiplicity $1$, and in case $(3)$ $E/\langle \sigma\rangle$ is a multiplicity-free BG whose diagram contains a unique cycle.

\begin{Lem}\label{determined-up-to-isomorphism-1}
In case $(1)$, $E$ is determined by $E/\langle\sigma\rangle$ and the order of the Nakayama automorphism $\sigma$ of $E$ up to isomorphism.
\end{Lem}

\begin{proof}
Suppose that $E$ and $E'$ are two $f_{ms}$-BGs such that $E/\langle\sigma\rangle\cong E'/\langle\sigma\rangle$ is a modified BG as in case $(1)$ and the Nakayama automorphisms of $E$ and $E'$ have the same order. According to Lemma \ref{odd degree} (2), the order of the Nakayama automorphisms of $E$ is even, say $2r$.

By Proposition \ref{modified fundamental group of modified f-BG}, the fundamental group of $E/\langle\sigma\rangle$ is isomorphic to $$F\langle a,c_1,c_2\rangle/\langle a c_1=c_1 a, a c_2=c_2 a, c_{1}^{2}=c_{2}^{2}=1\rangle,$$ and by the proof of Proposition \ref{rep type to fundamental group}, the image of the fundamental group of $E$ in $F\langle a,c_1,c_2\rangle/\langle a c_1=c_1 a$, $a c_2=c_2 a$, $c_{1}^{2}=c_{2}^{2}=1\rangle$ is the subgroup of $F\langle a,c_1,c_2\rangle/\langle a c_1=c_1 a$, $a c_2=c_2 a$, $c_{1}^{2}=c_{2}^{2}=1\rangle$ formed by elements $x$ which satisfies $\rho(x)(1)=1$, where $\rho:F\langle a,c_1,c_2\rangle/\langle a c_1=c_1 a$, $a c_2=c_2 a$, $c_{1}^{2}=c_{2}^{2}=1\rangle\rightarrow S_{2r}$ is the group homomorphism given by $\rho(\overline{a})=(1$ $2\cdots 2r)$ and $\rho(\overline {c_1})=\rho(\overline{c_2})=(1$ $r+1)(2$ $r+2)\cdots (r$ $2r)$. Since the same things also hold for $E'$, the images of the fundamental groups of $E$ and $E'$ in $F\langle a,c_1,c_2\rangle/\langle a c_1=c_1 a$, $a c_2=c_2 a$, $c_{1}^{2}=c_{2}^{2}=1\rangle$ are equal. By Proposition \ref{modified existence of morphism}, $E$ and $E'$ are isomorphic.
\end{proof}

Now we construct the $f_{ms}$-BG $E$ with $E/\langle\sigma\rangle$ belonging to case $(1)$. Suppose that the order of the Nakayama automorphism of $E$ is $2r$, where $r$ is a positive integer. Let $B=E/\langle\sigma\rangle=(B,B,\tau,d)$ and fix some $b\in B$. Since the diagram obtained by deleting the two double half-edges of the diagram of $B$ is a tree, each element $c$ of $B$ can be expressed uniquely as the form $(g\tau)^{j_c}(b)$, where $0\leq j_c\leq 2n-1$. For every vertex $v$ of $B$, let $b_v$ be the half-edge in $v$ with $j_{b_v}$ smallest. Define an $f_{ms}$-BG $E_r=(E_r,E_r,\tau,d)$ as follows: $E_r=\{(c,j)\mid c\in B,j\in\{1,2,\cdots,2r\}=\mathbb{Z}/2r\mathbb{Z}\}$; for every $(c,j)\in B$, define
$$g\cdot (c,j)=\begin{cases}
(g\cdot c,j), &\text{ if } g\cdot c\neq b_v \text{ for any vertex } v \text{ of } B; \\
(g\cdot c,j+1), &\text{ if } g\cdot c= b_v \text{ for some vertex } v \text{ of } B,
\end{cases}$$
$$\tau(c,j)=\begin{cases}
(\tau(c),j), &\text{ if } c \text{ is not a double half-edge of } B; \\
(c,j+r), &\text{ if } c \text{ is a double half-edge of } B,
\end{cases}$$
and $d(c,j)=d(c)$.

\begin{Prop}\label{construction-of-E-in-case-(1)}
In case $(1)$, $E$ is isomorphic to $E_r$.
\end{Prop}

\begin{proof}
Since $B=E/\langle\sigma\rangle$ is f-degree-free, $d(c)=|G\cdot c|$ for each $c\in B$. For every $(c,j)\in E_r$, since $d(c,j)=d(c)$, it can be shown that the Nakayama automorphism $\sigma$ of $E_r$ is given by $\sigma(c,j)=(c,j+1)$. So the covering $E_r\rightarrow B$, $(c,j)\mapsto c$ induces an isomorphism $E_r/\langle\sigma\rangle\cong B$. Since the order of the Nakayama automorphism of $E_r$ is $2r$, by Lemma \ref{determined-up-to-isomorphism-1}, $E_r$ is isomorphic to $E$.
\end{proof}

\begin{Lem}\label{determined-up-to-isomorphism-2}
In case $(2)$, $E$ is determined by $E/\langle\sigma\rangle$ and the order of the Nakayama automorphism $\sigma$ of $E$ up to isomorphism.
\end{Lem}

\begin{proof}
The proof is similar to that of Lemma \ref{determined-up-to-isomorphism-1}. Suppose that $E$ and $E'$ are two $f_{ms}$-BGs such that $E/\langle\sigma\rangle\cong E'/\langle\sigma\rangle$ is an $f_{ms}$-BG as in case $(2)$ and the Nakayama automorphisms of $E$ and $E'$ have the same order. According to the proof of Proposition \ref{rep type to fundamental group}, the order of the Nakayama automorphisms of $E$ is odd, say $2r-1$.

By \cite[Proposition 6.9]{LL2}, the fundamental group of $E/\langle\sigma\rangle$ is isomorphic to $F\langle a,b\rangle/\langle a^2=b^2\rangle$, and by the proof of Proposition \ref{rep type to fundamental group}, the image of the fundamental group of $E$ in $F\langle a,b\rangle/\langle a^2=b^2\rangle$ is the subgroup of $F\langle a,b\rangle/\langle a^2=b^2\rangle$ formed by elements $x$ which satisfies $\rho(x)(1)=1$, where $\rho: F\langle a,b\rangle/\langle a^2=b^2\rangle\rightarrow S_{2r-1}$ is the group homomorphism given by $\rho(\overline{a})=\rho(\overline{b})=(1$ $2\cdots 2r-1)$. Since the same things also hold for $E'$, the images of the fundamental groups of $E$ and $E'$ in $F\langle a,b\rangle/\langle a^2=b^2\rangle$ are equal. By \cite[Proposition 3.29]{LL2}, $E$ and $E'$ are isomorphic.
\end{proof}

Now we construct the $f_{ms}$-BG $E$ with $E/\langle\sigma\rangle$ belonging to case $(2)$. This construction is similar to that in case $(1)$. Suppose that the order of the Nakayama automorphism of $E$ is $2r-1$, where $r$ is a positive integer. Let $B=E/\langle\sigma\rangle=(B,B,\tau,d)$ and fix some $b\in B$. Since the diagram of $B$ is a tree, each element $c$ of $B$ can be expressed uniquely as the form $(g\tau)^{j_c}(b)$, where $0\leq j_c\leq 2n-3$. For every vertex $v$ of $B$, let $b_v$ be the half-edge in $v$ such that $j_{b_v}$ is smallest. Define an $f_{ms}$-BG $E'_r=(E'_r,E'_r,\tau,d)$ as follows: $E'_r=\{(c,j)\mid c\in B,j\in\{1,2,\cdots,2r-1\}=\mathbb{Z}/(2r-1)\mathbb{Z}\}$; for every $(c,j)\in E'_r$, define
$$g\cdot (c,j)=\begin{cases}
(g\cdot c,j), &\text{ if } g\cdot c\neq b_v \text{ for any vertex } v \text{ of } B; \\
(g\cdot c,j+1), &\text{ if } g\cdot c= b_{v_i} \text{ for some vertex } v_i \text{ of } B \text{ with } i\neq i_0,i_1; \\
(g\cdot c,j+r), &\text{ if } g\cdot c= b_{v_i} \text{ for some vertex } v_i \text{ of } B \text{ with } i=i_0 \text{ or } i=i_1,
\end{cases}$$
$\tau(c,j)=(\tau(c),j)$,
and $d(c,j)=d(c)$.

\begin{Prop}\label{construction-of-E-in-case-(2)}
In case $(2)$, $E$ is isomorphic to $E'_r$.
\end{Prop}
\begin{proof}
For every $(c,j)\in E'_r$, since
$$d(c,j)=\begin{cases}
| G\cdot c |, &\text{ if } c\notin v_{i_0}, v_{i_1}; \\
2 | G\cdot c |, &\text{ if } c\in v_{i_0} \text{ or } c\in v_{i_1},
\end{cases}$$
we have $\sigma(c,j)=(c,j+1)$, where $\sigma$ denotes the Nakayama automorphism of $E'_r$. Therefore $E'_r/\langle\sigma\rangle\cong B$. Since the order of the Nakayama automorphism of $E'_r$ is $2r-1$, by Lemma \ref{determined-up-to-isomorphism-2}, $E'_r$ is isomorphic to $E$.
\end{proof}

Let $B=E/\langle\sigma\rangle=(B,B,\tau,d)$ be the $f_{ms}$-BG in case $(3)$. Then the diagram of $B$ is a graph with a unique cycle. Suppose the length of this cycle is $m$, and there are $p$ edges outside this cycle and $q$ edges inside this cycle. Then we have $m+p+q=n$. Denote by $g\tau$ the permutation of $B$ mapping each $c\in B$ to $(g\tau)(c)$. It is straightforward to show that $B$ has exactly two $\langle g\tau\rangle$-orbits, one of length $m+2p$ containing every half-edge outside the unique cycle of $B$, and the other of length $m+2q$ containing every half-edge inside the unique cycle of $B$. We call a half-edge of $B$ {\it outer} (resp. {\it inner}) if it belongs to the first (resp. second) $\langle g\tau\rangle$-orbit.

Now we construct the $f_{ms}$-BG $E$ with $E/\langle\sigma\rangle$ belonging to case $(3)$. Fix an outer half-edge $b\in B=E/\langle\sigma\rangle$ belonging to the unique cycle of $B$. Then $\tau(b)$ is an inner half-edge of $B$. For every outer (resp. inner) half-edge $c$ of $B$, there exists a unique integer $0\leq j_c\leq m+2p-1$ (resp. $0\leq j'_c\leq m+2q-1$) such that $c=(g\tau)^{j_c}(b)$ (resp. $c=(g\tau)^{j'_c}(\tau(b))$). For every vertex $v$ of $B$, define a half-edge $b_v\in v$ as follows: if $v$ contains an outer half-edge, define $b_v$ as the outer half-edge in $v$ with $j_{b_v}$ smallest; if $v$ does not contain any outer half-edge, define $b_v$ as the inner half-edge in $v$ with $j'_{b_v}$ smallest.

Suppose that the order of the Nakayama automorphism $\sigma$ of $E$ is $r$. For every integer $1\leq l\leq r$, define an $f_{ms}$-BG $E_{rl}=(E_{rl},E_{rl},\tau,d)$ as follows: $E_{rl}=\{(c,j)\mid c\in B, j\in\{1,\cdots,r\}=\mathbb{Z}/r\mathbb{Z}\}$; for every $(c,j)\in E_{rl}$, define
$$g\cdot (c,j)=\begin{cases}
(g\cdot c,j), &\text{ if } g\cdot c\neq b_v \text{ for any vertex } v \text{ of } B; \\
(g\cdot c,j+1), &\text{ if } g\cdot c= b_v \text{ for some vertex } v \text{ of } B,
\end{cases}$$
$$\tau(c,j)=\begin{cases}
(\tau(c),j), &\text{ if } c\neq b \text{ and } c\neq\tau(b); \\
(\tau(c),j+l), &\text{ if } c=b; \\
(\tau(c),j-l), &\text{ if } c=\tau(b),
\end{cases}$$
and $d(c,j)=d(c)$.

\begin{Prop}\label{construction-of-E-in-case-(3)}
In case $(3)$, $E$ is isomorphic to some $E_{rl}$ with $1\leq l\leq r$.
\end{Prop}

We need some preparations before we prove Proposition \ref{construction-of-E-in-case-(3)}.

According to the proof of \cite[Proposition 6.9]{LL2}, we have an isomorphism $u:\Pi(B,b)\rightarrow\mathbb{Z}\oplus\mathbb{Z}$ with $u(\overline{(b|g^{d(b)}|b)})=(1,0)$ and $u(\overline{(b|(g\tau)^{m+2p}|b)})=(0,1)$. Let $H$ be the image of the composition map $\Pi(E,e)\rightarrow\Pi(B,b)\xrightarrow{u}\mathbb{Z}\oplus\mathbb{Z}$, where $e\in E$ is a preimage of $b$ in $E$ (the subgroup $H$ of $\mathbb{Z}\oplus\mathbb{Z}$ does not depend on the choice of $e$). Let $H_1=\{(a,0)\mid a\in\mathbb{Z}\}$, $H_2=\{(0,a)\mid a\in\mathbb{Z}\}$ be two subgroups of $\mathbb{Z}\oplus\mathbb{Z}$. Since $E$ is connected and since the order of the Nakayama automorphism $\sigma$ of $E$ is $r$, the closed walk $(b|g^{kd(b)}|b)$ of $B$ at $b$ lifts to a closed walk of $E$ at $e$ if and only if $r\mid k$. Therefore $H\cap H_1=rH_1$. Moreover, since the covering $E\rightarrow B=E/\langle\sigma\rangle$ is $r$-sheeted, $[\mathbb{Z}\oplus\mathbb{Z}:H]=r$.

\begin{Lem}\label{generators-of-H}
$H$ is a free abelian subgroup of $\mathbb{Z}\oplus\mathbb{Z}$ generated by $(r,0)$ and $(i,1)$, where $i$ is some integer with $0\leq i\leq r-1$.
\end{Lem}

\begin{proof}
We have $(H_1+H)/H\cong H_1/(H_1\cap H)\cong \mathbb{Z}/r\mathbb{Z}$, and $[H_1+H:H]=r$. Since $[\mathbb{Z}\oplus\mathbb{Z}:H]=r$, $H_1+H=\mathbb{Z}\oplus\mathbb{Z}$. Since $(0,1)\in H_1+H$, there exists some $h\in H$ such that $(0,1)\in H_1+h$. So $h=(i,1)$ for some $i\in\mathbb{Z}$. Let $H'$ be the subgroup of $H$ generated by $(r,0)$ and $(i,1)$. Since $[\mathbb{Z}\oplus\mathbb{Z}:H']=r$, $H'=H$. Then $H$ is generated by $(r,0)$ and $(i,1)$, and we may choose $i$ to be an integer such that $0\leq i\leq r-1$.
\end{proof}

\begin{proof}[{\bf Proof of Proposition \ref{construction-of-E-in-case-(3)}}]
Since the f-degree of $B$ is free, the Nakayama automorphism $\sigma$ of $E_{rl}$ is given by $\sigma(c,j)=(c,j+1)$. Then the covering $p:E_{rl}\rightarrow B$, $(c,j)\mapsto c$ induces an isomorphism $E_{rl}/\langle\sigma\rangle\cong B$. Let $u:\Pi(B,b)\rightarrow\mathbb{Z}\oplus\mathbb{Z}$ be the isomorphism with $u(\overline{(b|g^{d(b)}|b)})=(1,0)$ and $u(\overline{(b|(g\tau)^{m+2p}|b)})=(0,1)$. In order to calculate the image of the composition map $\Pi(E_{rl},(b,1))\xrightarrow{p_{*}}\Pi(B,b)\xrightarrow{u}\mathbb{Z}\oplus\mathbb{Z}$, we first consider the action of $\mathbb{Z}\oplus\mathbb{Z}$ on $p^{-1}(b)=\{(b,j)\mid j\in\{1,\cdots,r\}=\mathbb{Z}/r\mathbb{Z}\}$ via the isomorphism $u:\Pi(B,b)\rightarrow\mathbb{Z}\oplus\mathbb{Z}$.

We have $(1,0)\cdot (b,j)=g^{d(b)}\cdot (b,j)=\sigma(b,j)=(b,j+1)$. Since the number of outer half-edges of $B$ of the form $b_v$ with $v$ a vertex of $B$ is $m+p$, we have $(0,1)\cdot (b,j)=(g\tau)^{m+2p}(b,j)=(b,j+l+m+p)$. Let $K_l$ be the image of the composition map $\Pi(E_{rl},(b,1))\xrightarrow{p_{*}}\Pi(B,e)\xrightarrow{u}\mathbb{Z}\oplus\mathbb{Z}$. Then $K_l=\{x\in\mathbb{Z}\oplus\mathbb{Z}\mid x\cdot (b,1)=(b,1)\}$, which is the subgroup of $\mathbb{Z}\oplus\mathbb{Z}$ generated by $(r,0)$ and $(-(l+m+p),1)$. According to Lemma \ref{generators-of-H}, there exists some $1\leq l\leq r$ such that $K_l=H$. By \cite[Proposition 3.29]{LL2}, $E_{rl}$ is isomorphic to $E$.
\end{proof}

We remark that in Case (3) the $f_{ms}$-BGs $E_{r1},\cdots,E_{rr}$ that we have constructed may not be pairwise non-isomorphic, although the $f_{ms}$-BG $E$ is isomorphic to one of the $f_{ms}$-BGs $E_{rl}$'s. For example, if $B=E/\langle\sigma\rangle$ is a BG with free f-degree given by the diagram
$$\begin{tikzpicture}
\draw (0,0) circle (0.5);
\fill (0.5,0) circle (0.5ex);
\fill (-0.5,0) circle (0.5ex);
\node at(-0.9,0.5) {$b=c_1$};
\node at(1.1,0.5) {$\tau(b)=c_4$};
\node at(-0.5,-0.5) {$c_2$};
\node at(0.5,-0.5) {$c_3$};
\end{tikzpicture}.$$
Then there exists an isomorphism $f:E_{3,1}\rightarrow E_{3,3}$, which is given by $f(c_1,i)=(c_2,i-1)$, $f(c_2,i)=(c_1,i)$, $f(c_3,i)=(c_4,i)$, $f(c_4,i)=(c_3,i+1)$ for each $i=\{1,2,3\}=\mathbb{Z}/3\mathbb{Z}$.

\subsection{Stable AR-components of domestic  $f_{ms}$-BGAs}
\

In this subsection, we always denote by $\tau$ the involution of a Brauer $G$-set or the automorphism of a stable translation quiver induced by its translation according to its context, and we always denote by $\sigma$ the Nakayama automorphism of a Brauer $G$-set. We fix the enumeration on the vertices of $A_{\infty}^{\infty}$ as follows:

$$\xymatrix{A_{\infty}^{\infty}: & & \cdots \ar[r]& -2\ar[r] & -1\ar[r] & 0 \ar[r]&
               1 \ar[r] & 2\ar[r]  & \cdots
	}$$

To determine the stable AR-components for a domestic $f_{ms}$-BGA, our original ideal is to determine all its exceptional tubes, and using the result about the stable AR-components of domestic self-injective special biserial algebras in \cite[Theorem 2.1]{ES}. However, for a domestic self-injective special biserial algebra $A$, it is not clear whether the $m$ components of the form $\mathbb{Z}A_{\infty}/\langle\tau^{p}\rangle$ and the $m$ components of the form $\mathbb{Z}A_{\infty}/\langle\tau^{q}\rangle$ in \cite[Theorem 2.1(iii)]{ES} are just all the exceptional tubes of $A$ (when $A$ is symmetric this assertion is true, which was pointed out by Duffield). Therefore we need some additional information about the $\mathbb{Z}\widetilde{A}_{p,q}$ components of domestic $f_{ms}$-BGAs.

\begin{Lem}\label{ZApq-components}
If $E$, $E'$ are finite connected $f_{ms}$-BGs with $A_{E'}$ domestic, and suppose that there exists a regular covering $f:E\rightarrow E'$, then $A_E$ is domestic. Moreover, if the stable AR-quiver $\prescript{}{s}{\Gamma}_{A_{E'}}$ of $A_{E'}$ contains a component of the form $\mathbb{Z}\widetilde{A}_{p',q'}$, where $p'$, $q'$ are positive integers, then $\prescript{}{s}{\Gamma}_{A_{E}}$ contains a component of the form $\mathbb{Z}\widetilde{A}_{np',nq'}$ for some positive integer $n$.
\end{Lem}

\begin{proof}
Let $p:\widetilde{E}\rightarrow E$ be a universal cover of $E$. Since $\widetilde{E}$ is simply connected, $p$ is regular. According to \cite[Proposition 5.15]{LL2}, $p$ induces a Galois covering $F:\Lambda_{\widetilde{E}}\rightarrow\Lambda_E$ with group $\Pi\cong\Pi(E)$. Since $A_{E'}$ is domestic, according to Proposition \ref{rep type to fundamental group}, $\Pi(E')$ is torsion-free. Since $\Pi\cong\Pi(E)$ is isomorphic to a subgroup of $\Pi(E')$, $\Pi$ is also torsion-free, so it acts freely on ind$\Lambda_{\widetilde{E}}$. According to \cite[Theorem 3.6]{G}, $F_{\lambda}$ induces an isomorphism of the quotient $\Gamma_{\Lambda_{\widetilde{E}}}/\Pi$ onto the union of some components of $\Gamma_{\Lambda_E}$, where $F_{\lambda}:\mathrm{mod}\Lambda_{\widetilde{E}}\rightarrow\mathrm{mod}\Lambda_{E}$ is the push-down functor associated with $F$. Since $F_{\lambda}$ maps projectives to projectives, it induces an isomorphism of the quotient $\prescript{}{s}{\Gamma}_{\Lambda_{\widetilde{E}}}/\Pi$ onto the union of some components of $\prescript{}{s}{\Gamma}_{\Lambda_E}$. Since $fp:\widetilde{E}\rightarrow E'$ is a universal cover of $E'$, it induces a Galois covering $F':\Lambda_{\widetilde{E}}\rightarrow\Lambda_{E'}$ with group $\Pi'\cong\Pi(E')$. Similarly, the push-down functor $F'_{\lambda}:\mathrm{mod}\Lambda_{\widetilde{E}}\rightarrow\mathrm{mod}\Lambda_{E'}$ induces an isomorphism of the quotient $\prescript{}{s}{\Gamma}_{\Lambda_{\widetilde{E}}}/\Pi'$ onto the union of some components of $\prescript{}{s}{\Gamma}_{\Lambda_{E'}}$. 

Denote by $(\prescript{}{s}{\Gamma}_{\Lambda_{E'}})^{0}$ (resp. $(\prescript{}{s}{\Gamma}_{\Lambda_{E}})^{0}$) the image of $\prescript{}{s}{\Gamma}_{\Lambda_{\widetilde{E}}}/\Pi'\hookrightarrow\prescript{}{s}{\Gamma}_{\Lambda_{E'}}$ (resp. $\prescript{}{s}{\Gamma}_{\Lambda_{\widetilde{E}}}/\Pi\hookrightarrow\prescript{}{s}{\Gamma}_{\Lambda_{E}}$). Since $\Pi\cong\Pi(E)$ is a normal subgroup of $\Pi'\cong\Pi(E')$, there is a Galois covering $\prescript{}{s}{\Gamma}_{\Lambda_{\widetilde{E}}}/\Pi\rightarrow\prescript{}{s}{\Gamma}_{\Lambda_{\widetilde{E}}}/\Pi'$ with group $\Pi'/\Pi$, which induces a Galois covering $\psi:(\prescript{}{s}{\Gamma}_{\Lambda_{E}})^{0}\rightarrow(\prescript{}{s}{\Gamma}_{\Lambda_{E'}})^{0}$ with group $\Pi'/\Pi$. Note that each string in $\Lambda_{E'}$ induces a walk of $E'$, which lifts to a walk of $\widetilde{E}$ and induces a string in $\Lambda_{\widetilde{E}}$. So for each string module $M'$ of $\Lambda_{E'}$, there exists some string module $\widetilde{M}$ over $\Lambda_{\widetilde{E}}$ such that $F'_{\lambda}\widetilde{M}\cong M'$. Therefore $(\prescript{}{s}{\Gamma}_{\Lambda_{E'}})^{0}$ contains each component of $\prescript{}{s}{\Gamma}_{\Lambda_{E'}}$ consisting of string modules. Similarly, $(\prescript{}{s}{\Gamma}_{\Lambda_{E}})^{0}$ contains each component of $\prescript{}{s}{\Gamma}_{\Lambda_{E}}$ consisting of string modules.

Since $A_{E'}$ is domestic, by \cite[Theorem 2.1]{ES}, $\prescript{}{s}{\Gamma}_{\Lambda_{E'}}$ contains a component $\Gamma'$ of the form $\mathbb{Z}\widetilde{A}_{p',q'}$. Then $\Gamma'$ consists of string modules, so $\Gamma'$ is contained in $(\prescript{}{s}{\Gamma}_{\Lambda_{E'}})^{0}$. Choose a vertex $x$ of $(\prescript{}{s}{\Gamma}_{\Lambda_{E}})^{0}$ such that the image $x'=\psi(x)$ of $x$ in $(\prescript{}{s}{\Gamma}_{\Lambda_{E'}})^{0}$ belongs to $\Gamma'$, and let $\Gamma$ be the component of $(\prescript{}{s}{\Gamma}_{\Lambda_{E}})^{0}$ containing $x$. Then there is a covering $\psi':\Gamma\rightarrow\Gamma'$ of translation quivers which is induced from $\psi$. Since the fundamental group of $\Gamma'$ is isomorphic to $\mathbb{Z}$, which is abelian, the covering $\psi':\Gamma\rightarrow\Gamma'$ is Galois. Since $\psi:\prescript{}{s}{\Gamma}_{\Lambda_{\widetilde{E}}}/\Pi\rightarrow\prescript{}{s}{\Gamma}_{\Lambda_{\widetilde{E}}}/\Pi'$ is a Galois covering with group $\Pi'/\Pi$, where $\Pi'/\Pi\cong\Pi(E')/f_{*}(\Pi(E))\cong\mathrm{Aut}(f)$ is finite, $\psi^{-1}(x')$ is finite. Then $(\psi')^{-1}(x')$ is finite, and the group of the Galois covering $\psi'$ is finite. It is straightforward to show that $\mathbb{Z}A_{\infty}^{\infty}$ is a universal cover of $\mathbb{Z}\widetilde{A}_{p',q'}$, and $\mathbb{Z}\widetilde{A}_{p',q'}\cong\mathbb{Z}A_{\infty}^{\infty}/\langle\phi\rangle$, where $\phi$ is the automorphism of $\mathbb{Z}A_{\infty}^{\infty}$ given by $(i,j)\mapsto(i-q',j+p'+q')$. Since $\Gamma$ is connected and since there exists a covering $\psi':\Gamma\rightarrow\Gamma'$, $\Gamma$ is isomorphic to $\mathbb{Z}A_{\infty}^{\infty}/H$ for some subgroup $H$ of $\langle\phi\rangle$. Since the group of the Galois covering $\psi':\Gamma\rightarrow\Gamma'$ is finite, $H$ is of the form $\langle\phi^n\rangle$ for some positive integer $n$, and $\Gamma\cong\mathbb{Z}A_{\infty}^{\infty}/\langle\phi^n\rangle\cong\mathbb{Z}\widetilde{A}_{np',nq'}$. Since $\prescript{}{s}{\Gamma}_{A_{E}}$ contains a component $\Gamma$ of the form $\mathbb{Z}\widetilde{A}_{p,q}$, according to \cite[Theorem 2.1]{ES}, $A_E$ is domestic.
\end{proof}

Let $E$ be a finite connected $f_{ms}$-BG such that $E/\langle\sigma\rangle$ is a Brauer $G$-set in case $(1)$ of Lemma \ref{B}, that is, $B=E/\langle\sigma\rangle$ is an f-degree-free modified BG with $n$ vertices, $n-1$ edges and $2$ double half-edges. Denote by $\sigma^{-1}(g\tau)^2$ the permutation of $E$ mapping each $e\in E$ to $\sigma^{-1}(g\tau)^2(e)$. Suppose that the order of the Nakayama automorphism $\sigma$ of $E$ is $2r$.

\begin{Lem}\label{length-of-DTr-orbit-case-1}
For every $e\in E$, the length of the $\langle\sigma^{-1}(g\tau)^2\rangle$-orbit of $e$ is $n$.
\end{Lem}

\begin{proof}
According to Proposition \ref{construction-of-E-in-case-(1)}, $E$ is isomorphic to $E_r$ (the definition of $E_r$ is given before Proposition \ref{construction-of-E-in-case-(1)}). For every $(c,j)\in E_r$, let $N$ be the minimal positive integer such that $(\sigma^{-1}(g\tau)^2)^N(c,j)=(c,j)$. Note that the Nakayama automorphism $\sigma$ of $B$ is identity. So $n$ is the minimal positive integer such that $(\sigma^{-1}(g\tau)^2)^n(c)=c$. Since there exists a covering of Brauer $G$-sets $p:E_r\rightarrow B$ which maps each $(c',j')\in E_r$ to $c'\in B$, we have
$$(\sigma^{-1}(g\tau)^2)^N(c)=(\sigma^{-1}(g\tau)^2)^N(p(c,j))=p((\sigma^{-1}(g\tau)^2)^N(c,j))=p(c,j)=c.$$
Therefore $n\mid N$.

Since $\{(g\tau)(c),(g\tau)^2(c),\cdots,(g\tau)^{2n}(c)\}=B$, there are exactly $n$ numbers $i\in\{1,2,\cdots, 2n\}$ such that $(g\tau)^i(c)$ is of the form $b_v$ for some vertex $v$ of $B$. Moreover, since $B$ contains exactly two double half-edges and since $\{c,(g\tau)(c),\cdots,(g\tau)^{2n-1}(c)\}=B$, there are exactly two numbers $i\in\{0,1,\cdots, 2n-1\}$ such that $(g\tau)^i(c)$ is a double half-edge of $B$. So we have $(g\tau)^{2n}(c,j)=((g\tau)^{2n}(c),j+n+2r)=(c,j+n)$. Therefore $(\sigma^{-1}(g\tau)^2)^n(c,j)=\sigma^{-n}(g\tau)^{2n}(c,j)=\sigma^{-n}(c,j+n)=(c,j)$ and $N\mid n$. Then we have $N=n$.
\end{proof}

\begin{Prop} \label{stable-AR-component-case-(1)}
Let $E$ be a finite connected $f_{ms}$-BG such that $E/\langle\sigma\rangle$ is a Brauer $G$-set in case $(1)$ of Lemma \ref{B}. Suppose that the order of the Nakayama automorphism of $E$ is $2r$. Then $\prescript{}{s}{\Gamma}_{A_E}$ is a disjoint union of $4r$ components of the form $\mathbb{Z}A_{\infty}/\langle\tau^{n}\rangle$, $2r$ components of the form $\mathbb{Z}\widetilde{A}_{n,n}$, and infinitely many
components of the form $\mathbb{Z}A_{\infty}/\langle\tau\rangle$.
\end{Prop}

\begin{proof}
Since $E$ has $4nr$ half-edges, by Lemma \ref{length-of-DTr-orbit-case-1}, $E$ contains $4r$ $\langle\sigma^{-1}(g\tau)^2\rangle$-orbits, each of length $n$. By Proposition \ref{exceptional-tubes}, $\prescript{}{s}{\Gamma}_{A_E}$ contains $4r$ exceptional tubes, where the rank of each tube is $n$.

Denote by $f:E\rightarrow B=E/\langle\sigma\rangle$ and $\pi:R_E\rightarrow B$ the natural projections, and let $p:\widetilde{E}\rightarrow E$ be a universal cover of $E$. Fix $\widetilde{e}\in\widetilde{E}$ and let $e=p(\widetilde{e})\in E$, $b=fp(\widetilde{e})\in B$. Choose $b'\in R_E$ such that $\pi(b')=b$, then there exists a unique covering $q:\widetilde{E}\rightarrow R_E$ such that $fp=\pi q$ and $q(\widetilde{e})=b'$. By an analogue of \cite[Proposition 3.32]{LL2}, there exist isomorphisms
$$\mu_1:\mathrm{Aut}(p)\xrightarrow{\sim}\Pi_m(E,e)$$
$$\mu_2:\mathrm{Aut}(q)\xrightarrow{\sim}\Pi_m(R_E,b')$$
$$\mu_3:\mathrm{Aut}(fp)\xrightarrow{\sim}\Pi_m(B,b)$$
such that the diagrams
$$\xymatrix@C=0.2pc{
		 \mathrm{Aut}(p)\ar[d]_{\mu_1}^{\wr} & \leq & \mathrm{Aut}(fp)\ar[d]_{\mu_3}^{\wr} \\
		\Pi_m(E,e)\ar[rr]^{f_{*}} & & \Pi_m(B,b)
	}$$
and
$$\xymatrix@C=0.2pc{
		 \mathrm{Aut}(q)\ar[d]_{\mu_2}^{\wr} & \leq & \mathrm{Aut}(fp)\ar[d]_{\mu_3}^{\wr} \\
		\Pi_m(R_E,b')\ar[rr]^{\pi_{*}} & & \Pi_m(B,b)
	}$$
commute. Since $f$ and $\pi$ are regular coverings, $f_{*}(\Pi_m(E,e))$ and $\pi_{*}(\Pi_m(R_E,b'))$ are normal subgroups of $\Pi_m(B,b)$. Therefore Aut$(p)$ and Aut$(q)$ are normal subgroups of Aut$(fp)$. Denote by $H=\mathrm{Aut}(p)\cap\mathrm{Aut}(q)$, which is both a normal subgroup of Aut$(p)$ and a normal subgroup of Aut$(q)$. Denote by $E'=\widetilde{E}/H$, then the coverings $p:\widetilde{E}\rightarrow E$ and $q:\widetilde{E}\rightarrow R_E$ induce coverings $p':E'\rightarrow E$ and $q':E'\rightarrow R_E$. By \cite[Theorem 3.37]{LL2}, Aut$(p)$ acts admissibly on $\widetilde{E}$, so $H$ also acts admissibly on $\widetilde{E}$ and $E'=\widetilde{E}/H$ is an $f_{ms}$-BG. Since $H$ is both a normal subgroup of Aut$(p)$ and a normal subgroup of Aut$(q)$, the coverings $p':E'\rightarrow E$ and $q':E'\rightarrow R_E$ are regular. Moreover, \begin{multline*}\mathrm{Aut}(p')\cong\mathrm{Aut}(p)/H\cong f_{*}(\Pi_m(E,e))/(f_{*}(\Pi_m(E,e))\cap \pi_{*}(\Pi_m(R_E,b'))) \\ \cong f_{*}(\Pi_m(E,e))\pi_{*}(\Pi_m(R_E,b'))/\pi_{*}(\Pi_m(R_E,b')) \\ \leq\Pi_m(B,b)/\pi_{*}(\Pi_m(R_E,b'))\cong\mathrm{Aut}(\pi).\end{multline*}
Since $R_E$ is a finite $f_{ms}$-BG, Aut$(\pi)$ is finite. So Aut$(p')$ is also finite. Since $E\cong E'/\mathrm{Aut}(p')$ and since $E$ is finite, $E'$ is also finite. Since $\widetilde{E}$ is connected and since $E'=\widetilde{E}/H$, $E'$ is connected. Since $p':E'\rightarrow E$ is a regular covering of finite connected $f_{ms}$-BGs and since $A_E$ is domestic, by Lemma \ref{ZApq-components}, $A_{E'}$ is also domestic.

Since $R_E$ is an f-degree-free BG with $2n$ edges whose underlying graph has a unique cycle of even length such that the number of edges inside this cycle and the number of edges outside this cycle are equal, according to \cite[Corollary 4.7]{D}, the $\mathbb{Z}\widetilde{A}_{a,b}$ components of the stable AR-quiver of $A_{R_E}$ are given by $a=b=n$. Since there are regular coverings $p':E'\rightarrow E$ and $E'\rightarrow R_E$, where $E'$, $E$, $R_E$ are finite connected $f_{ms}$-BGs with $A_{E'}$, $A_E$, $A_{R_E}$ domestic, according to Lemma \ref{ZApq-components}, the $\mathbb{Z}\widetilde{A}_{a,b}$ components of $\prescript{}{s}{\Gamma}_{A_E}$ satisfy $a=b$.

According to \cite[Theorem 2.1]{ES}, there are positive integers $m$, $p$, $q$ such that $\prescript{}{s}{\Gamma}_{A_E}$ is a disjoint union of $m$ components of the form $\mathbb{Z}\widetilde{A}_{p,q}$, $m$ components of the form $\mathbb{Z}A_{\infty}/\langle\tau^p\rangle$, $m$ components of the form $\mathbb{Z}A_{\infty}/\langle\tau^q\rangle$, and infinitely many components of the form $\mathbb{Z}A_{\infty}/\langle\tau\rangle$. When $n>1$, since $\prescript{}{s}{\Gamma}_{A_E}$ contains $4r$ exceptional tubes, where the rank of each tube is $n$, there are two possible cases:
\begin{itemize}
\item[$(i)$] $m=2r$, $p=q=n$;
\item[$(ii)$] $m=4r$, $p=n$, $q=1$.
\end{itemize}
Since $p=q$, case $(ii)$ can not occur. Therefore $\prescript{}{s}{\Gamma}_{A_E}$ is a disjoint union of $4r$ components of the form $\mathbb{Z}A_{\infty}/\langle\tau^{n}\rangle$, $2r$ components of the form $\mathbb{Z}\widetilde{A}_{n,n}$, and infinitely many components of the form $\mathbb{Z}A_{\infty}/\langle\tau\rangle$. When $n=1$, it is straightforward to show that $\prescript{}{s}{\Gamma}_{A_E}$ is a disjoint union of $2r$ components of the form $\mathbb{Z}\widetilde{A}_{1,1}$ and infinitely many components of the form $\mathbb{Z}A_{\infty}/\langle\tau\rangle$.
\end{proof}

Let $E$ be a finite connected $f_{ms}$-BG such that $E/\langle\sigma\rangle$ is a Brauer $G$-set in case $(2)$ of Lemma \ref{B}, that is, $B=E/\langle\sigma\rangle$ is a Brauer graph whose underlying diagram is a tree with $n$ vertices $v_1$, $\cdots$, $v_n$ of f-degree $d_1$, $\cdots$, $d_n$, respectively, such that $d_i=2$ for exactly two numbers $i=i_0$, $i_1$ and $d_i=1$ for $i\neq i_0$, $i_1$. Denote by $\sigma^{-1}(g\tau)^2$ the permutation of $E$ mapping each $e\in E$ to $\sigma^{-1}(g\tau)^2(e)$. Suppose that the order of the Nakayama automorphism $\sigma$ of $E$ is $2r-1$.

\begin{Lem}\label{length-of-DTr-orbit-case-2}
For every $e\in E$, the length of the $\langle\sigma^{-1}(g\tau)^2\rangle$-orbit of $e$ is $n-1$.
\end{Lem}

\begin{proof}
According to Proposition \ref{construction-of-E-in-case-(2)}, $E$ is isomorphic to $E'_r$ (the definition of $E'_r$ is given before Proposition \ref{construction-of-E-in-case-(2)}). For every $(c,j)\in E'_r$, let $N$ be the minimal positive integer such that $(\sigma^{-1}(g\tau)^2)^N(c,j)=(c,j)$. Note that the Nakayama automorphism $\sigma$ of $B$ is identity. So $n-1$ is the minimal positive integer such that $(\sigma^{-1}(g\tau)^2)^{n-1}(c)=c$. Since there exists a covering of Brauer $G$-sets $p:E'_r\rightarrow B$ which maps each $(c',j')\in E'_r$ to $c'\in B$, we have
$$(\sigma^{-1}(g\tau)^2)^N(c)=(\sigma^{-1}(g\tau)^2)^N(p(c,j))=p((\sigma^{-1}(g\tau)^2)^N(c,j))=p(c,j)=c.$$
Therefore $(n-1)\mid N$.

Since $\{(g\tau)(c),(g\tau)^2(c),\cdots,(g\tau)^{2n-2}(c)\}=B$, for each vertex $v$ of $B$, there is exactly one number $i\in\{1,2,\cdots, 2n-2\}$ such that $(g\tau)^i(c)=b_v$. So we have $(g\tau)^{2n-2}(c,j)=((g\tau)^{2n-2}(c),j+n-2+2r)=(c,j+n-2+2r)$. Since $\sigma(c',j')=(c,j'+1)$ for every $(c',j')\in E'_r$, $(\sigma^{-1}(g\tau)^2)^{n-1}(c,j)=\sigma^{-n+1}(g\tau)^{2n-2}(c,j)=\sigma^{-n+1}(c,j+n-2+2r)=(c,j+2r-1)=(c,j)$ and $N\mid (n-1)$. Then we have $N=n-1$.
\end{proof}

\begin{Prop} \label{stable-AR-component-case-(2)}
Let $E$ be a finite connected $f_{ms}$-BG such that $E/\langle\sigma\rangle$ is a Brauer $G$-set in case $(2)$ of Lemma \ref{B} and suppose that the order of the Nakayama automorphism of $E$ is $2r-1$. Then $\prescript{}{s}{\Gamma}_{A_E}$ is a disjoint union of $4r-2$ components of the form $\mathbb{Z}A_{\infty}/\langle\tau^{n-1}\rangle$, $2r-1$ components of the form $\mathbb{Z}\widetilde{A}_{n-1,n-1}$, and infinitely many
components of the form $\mathbb{Z}A_{\infty}/\langle\tau\rangle$.
\end{Prop}

\begin{proof}
Since $E$ has $(4r-2)(n-1)$ half-edges, by Lemma \ref{length-of-DTr-orbit-case-2}, $E$ contains $4r-2$ $\langle\sigma^{-1}(g\tau)^2\rangle$-orbits, each of length $n-1$. By Proposition \ref{exceptional-tubes}, $\prescript{}{s}{\Gamma}_{A_E}$ contains $4r-2$ exceptional tubes, where the rank of each tube is $n-1$. Since $B=E/\langle\sigma\rangle$ is a Brauer graph whose underlying diagram is a tree with $n$ vertices and since $A_B$ is domestic, according to \cite[Theorem 4.6]{D}, the $\mathbb{Z}\widetilde{A}_{a,b}$ component of the stable AR-quiver of $A_{B}$ is given by $a=b=n-1$. Then according to Lemma \ref{ZApq-components}, every $\mathbb{Z}\widetilde{A}_{a,b}$ component of $\prescript{}{s}{\Gamma}_{A_E}$ satisfies $a=b$.

According to \cite[Theorem 2.1]{ES}, there are positive integers $m$, $p$, $q$ such that $\prescript{}{s}{\Gamma}_{A_E}$ is a disjoint union of $m$ components of the form $\mathbb{Z}\widetilde{A}_{p,q}$, $m$ components of the form $\mathbb{Z}A_{\infty}/\langle\tau^p\rangle$, $m$ components of the form $\mathbb{Z}A_{\infty}/\langle\tau^q\rangle$, and infinitely many components of the form $\mathbb{Z}A_{\infty}/\langle\tau\rangle$. When $n>2$, since $\prescript{}{s}{\Gamma}_{A_E}$ contains $4r-2$ exceptional tubes, where the rank of each tube is $n-1$, there are two possible cases:
\begin{itemize}
\item[$(i)$] $m=2r-1$, $p=q=n-1$;
\item[$(ii)$] $m=4r-2$, $p=n-1$, $q=1$.
\end{itemize}
Since $p=q$, case $(ii)$ can not occur. Therefore $\prescript{}{s}{\Gamma}_{A_E}$ is a disjoint union of $4r-2$ components of the form $\mathbb{Z}A_{\infty}/\langle\tau^{n-1}\rangle$, $2r-1$ components of the form $\mathbb{Z}\widetilde{A}_{n-1,n-1}$, and infinitely many components of the form $\mathbb{Z}A_{\infty}/\langle\tau\rangle$. When $n=2$, it is straightforward to show that $\prescript{}{s}{\Gamma}_{A_E}$ is a disjoint union of $2r-1$ components of the form $\mathbb{Z}\widetilde{A}_{1,1}$ and infinitely many components of the form $\mathbb{Z}A_{\infty}/\langle\tau\rangle$.
\end{proof}

Let $E$ be a finite connected $f_{ms}$-BG such that $E/\langle\sigma\rangle$ is a Brauer $G$-set in case $(3)$ of Lemma \ref{B}, that is, $B$ is a Brauer graph with free f-degree whose diagram contains a unique cycle. Suppose that the length of this cycle is $m$, and suppose that there are $p$ edges outside this cycle and $q$ edges inside this cycle, where $n=m+p+q$. Fix an outer half-edge $b$ of $B$ which belongs to the unique cycle of $B$, and suppose that the order of the Nakayama automorphism of $E$ is $r$. According to Proposition \ref{construction-of-E-in-case-(3)}, $E$ is isomorphic to some $E_{rl}$ with $1\leq l\leq r$, where the definition of $E_{rl}$'s are given before Proposition \ref{construction-of-E-in-case-(3)}.

Denote by $\sigma^{-1}(g\tau)^2$ the permutation of $E$ mapping each $e\in E$ to $\sigma^{-1}(g\tau)^2(e)$.

\begin{Lem}\label{length-of-DTr-orbit-case-3}
Under the assumptions above, when $m$ is odd, $E\cong E_{rl}$ contains $(r,m+2l)$ $\langle \sigma^{-1}(g\tau)^2\rangle$-orbits of length $\frac{r(m+2p)}{(r,m+2l)}$ and $(r,m+2l)$ $\langle \sigma^{-1}(g\tau)^2\rangle$-orbits of length $\frac{r(m+2q)}{(r,m+2l)}$, and when $m$ is even, $E\cong E_{rl}$ contains $(2r,m+2l)$ $\langle \sigma^{-1}(g\tau)^2\rangle$-orbits of length $\frac{r(m+2p)}{(2r,m+2l)}$ and $(2r,m+2l)$ $\langle \sigma^{-1}(g\tau)^2\rangle$-orbits of length $\frac{r(m+2q)}{(2r,m+2l)}$, where $(a,b)$ denotes the greatest common divisor of $a$ and $b$.
\end{Lem}

\begin{proof}
Denote by $f:E_{rl}\rightarrow B$ the covering of Brauer $G$-sets which maps each $(c,j)\in E_{rl}$ to $c\in B$. Note that $B$ is a disjoint union of two $\langle g\tau\rangle$-orbits $b^{\langle g\tau\rangle}$ and $\tau(b)^{\langle g\tau\rangle}$, where the length of $b^{\langle g\tau\rangle}$ is $m+2p$ and the length of $\tau(b)^{\langle g\tau\rangle}$ is $m+2q$ (here $g\tau$ denotes the permutation of $B$ mapping each $c\in B$ to $g\tau(c)$).

When $m$ is odd, since both $m+2p$ and $m+2q$ are odd, the $\langle g\tau\rangle$-orbits $b^{\langle g\tau\rangle}$ and $\tau(b)^{\langle g\tau\rangle}$ of $B$ are also $\langle \sigma^{-1}(g\tau)^2\rangle$-orbits of $B$ (note that the Nakayama automorphism $\sigma$ of $B$ is identity). Since the image of each $\langle \sigma^{-1}(g\tau)^2\rangle$-orbit of $E_{rl}$ under $f$ is a $\langle \sigma^{-1}(g\tau)^2\rangle$-orbit of $B$, each $\langle \sigma^{-1}(g\tau)^2\rangle$-orbit of $E_{rl}$ is of the form $(b,j)^{\langle \sigma^{-1}(g\tau)^2\rangle}$ or of the form $(\tau(b),j)^{\langle \sigma^{-1}(g\tau)^2\rangle}$, where $j\in\{1,\cdots,r\}=\mathbb{Z}/r\mathbb{Z}$.

Let $N$ be the minimal positive integer such that $(\sigma^{-1}(g\tau)^2)^N(b,j)=(b,j)$. Since $m+2p$ is the minimal positive integer such that $(\sigma^{-1}(g\tau)^2)^{m+2p}(b)=b$, we have
$$(\sigma^{-1}(g\tau)^2)^N(b)=(\sigma^{-1}(g\tau)^2)^N(f(b,j))=f((\sigma^{-1}(g\tau)^2)^N(b,j))=f(b,j)=b.$$
Therefore $(m+2p)\mid N$. Note that $b^{\langle g\tau\rangle}=\{b,g\tau(b),\cdots,(g\tau)^{m+2p-1}(b)\}$ contains $m+p$ half-edges of the form $b_v$, and $b\in b^{\langle g\tau\rangle}$, $\tau(b)\notin b^{\langle g\tau\rangle}$. Therefore
$$(g\tau)^{m+2p}(b,j)=((g\tau)^{m+2p}(b),j+m+p+l)=(b,j+m+p+l),$$ and
$$(\sigma^{-1}(g\tau)^2)^{m+2p}(b,j)=\sigma^{-(m+2p)}(g\tau)^{2(m+2p)}(b,j)=(b,j+2(m+p+l)-(m+2p))=(b,j+m+2l).$$
Since $N'=\frac{r}{(r,m+2l)}$ is the minimal positive integer such that $((\sigma^{-1}(g\tau)^2)^{m+2p})^{N'}(b,j)=(b,j)$, $N=(m+2p)N'=\frac{r(m+2p)}{(r,m+2l)}$. Moreover, two half-edges $(b,j_1)$, $(b,j_2)$ of $E_{rl}$ belong to the same $\langle \sigma^{-1}(g\tau)^2\rangle$-orbit if and only if $(r,m+2l)$ divides $j_1-j_2$. Therefore there are $(r,m+2l)$ $\langle \sigma^{-1}(g\tau)^2\rangle$-orbits of $E_{rl}$ of the form $(b,j)^{\langle \sigma^{-1}(g\tau)^2\rangle}$, each of length $\frac{r(m+2p)}{(r,m+2l)}$.

Let $M$ be the minimal positive integer such that $(\sigma^{-1}(g\tau)^2)^M(\tau(b),j)=(\tau(b),j)$. Since $m+2q$ is the minimal positive integer such that $(\sigma^{-1}(g\tau)^2)^{m+2q}(\tau(b))=\tau(b)$, $m+2q$ divides $M$. Since $\tau(b)^{\langle g\tau\rangle}=\{\tau(b),g\tau(\tau(b)),\cdots,(g\tau)^{m+2q-1}(\tau(b))\}$ contains $q$ half-edges of the form $b_v$, and since $b\notin \tau(b)^{\langle g\tau\rangle}$, $\tau(b)\in \tau(b)^{\langle g\tau\rangle}$, we have $$(g\tau)^{m+2q}(\tau(b),j)=((g\tau)^{m+2q}(\tau(b)),j+q-l)=(\tau(b),j+q-l).$$
So
\begin{multline*}(\sigma^{-1}(g\tau)^2)^{m+2q}(\tau(b),j)=\sigma^{-(m+2q)}(g\tau)^{2(m+2q)}(\tau(b),j) \\ =(\tau(b),j+2(q-l)-(m+2q))=(\tau(b),j-(m+2l)).\end{multline*}
Then $M'=\frac{r}{(r,m+2l)}$ is the minimal positive integer such that $((\sigma^{-1}(g\tau)^2)^{m+2q})^{M'}(\tau(b),j)=(\tau(b),j)$, and $M=(m+2q)M'=\frac{r(m+2q)}{(r,m+2l)}$. Moreover, two half-edges $(\tau(b),j_1)$, $(\tau(b),j_2)$ of $E_{rl}$ belong to the same $\langle \sigma^{-1}(g\tau)^2\rangle$-orbit if and only if $(r,m+2l)$ divides $j_1-j_2$. Therefore there are $(r,m+2l)$ $\langle \sigma^{-1}(g\tau)^2\rangle$-orbits of $E_{rl}$ of the form $(\tau(b),j)^{\langle \sigma^{-1}(g\tau)^2\rangle}$, each of length $\frac{r(m+2q)}{(r,m+2l)}$.

When $m$ is even, since both $m+2p$ and $m+2q$ are even, the $\langle g\tau\rangle$-orbits $b^{\langle g\tau\rangle}$ of $B$ splits into two $\langle \sigma^{-1}(g\tau)^2\rangle$-orbits $b^{\langle \sigma^{-1}(g\tau)^2\rangle}$ and $g\tau(b)^{\langle \sigma^{-1}(g\tau)^2\rangle}$, each of length $\frac{m}{2}+p$, and the $\langle g\tau\rangle$-orbits $\tau(b)^{\langle g\tau\rangle}$ of $B$ splits into two $\langle \sigma^{-1}(g\tau)^2\rangle$-orbits $\tau(b)^{\langle \sigma^{-1}(g\tau)^2\rangle}$ and $(g\cdot b)^{\langle \sigma^{-1}(g\tau)^2\rangle}$, each of length $\frac{m}{2}+q$ (note that the Nakayama automorphism $\sigma$ of $B$ is identity). Since the image of each $\langle \sigma^{-1}(g\tau)^2\rangle$-orbit of $E_{rl}$ under $f$ is a $\langle \sigma^{-1}(g\tau)^2\rangle$-orbit of $B$, each $\langle \sigma^{-1}(g\tau)^2\rangle$-orbit of $E_{rl}$ is equal to one of the following $\langle \sigma^{-1}(g\tau)^2\rangle$-orbits of $E_{rl}$: $(b,j)^{\langle \sigma^{-1}(g\tau)^2\rangle}$, $(g\tau(b),j)^{\langle \sigma^{-1}(g\tau)^2\rangle}$, $(\tau(b),j)^{\langle \sigma^{-1}(g\tau)^2\rangle}$, $(g\cdot b,j)^{\langle \sigma^{-1}(g\tau)^2\rangle}$, where $j\in\{1,\cdots,r\}=\mathbb{Z}/r\mathbb{Z}$.

Let $N$ be the minimal positive integer such that $(\sigma^{-1}(g\tau)^2)^N(b,j)=(b,j)$. Since $\frac{m}{2}+p$ is the minimal positive integer such that $(\sigma^{-1}(g\tau)^2)^{\frac{m}{2}+p}(b)=b$, we have
$$(\sigma^{-1}(g\tau)^2)^N(b)=(\sigma^{-1}(g\tau)^2)^N(f(b,j))=f((\sigma^{-1}(g\tau)^2)^N(b,j))=f(b,j)=b.$$
Therefore $(\frac{m}{2}+p)\mid N$. Similar to the case where $m$ is odd, we have
$(g\tau)^{m+2p}(b,j)=(b,j+m+p+l)$. Then
$$(\sigma^{-1}(g\tau)^2)^{\frac{m}{2}+p}(b,j)=\sigma^{-(\frac{m}{2}+p)}(g\tau)^{m+2p}(b,j)=(b,j+m+p+l-(\frac{m}{2}+p))=(b,j+\frac{m}{2}+l).$$
Since $N'=\frac{r}{(r,\frac{m}{2}+l)}$ is the minimal positive integer such that $((\sigma^{-1}(g\tau)^2)^{\frac{m}{2}+p})^{N'}(b,j)=(b,j)$, $$N=(\frac{m}{2}+p)N'=\frac{r(m+2p)}{(2r,m+2l)}.$$
Moreover, two half-edges $(b,j_1)$, $(b,j_2)$ of $E_{rl}$ belong to the same $\langle \sigma^{-1}(g\tau)^2\rangle$-orbit if and only if $(r,\frac{m}{2}+l)$ divides $j_1-j_2$. Therefore there are $(r,\frac{m}{2}+l)$ $\langle \sigma^{-1}(g\tau)^2\rangle$-orbits of $E_{rl}$ of the form $(b,j)^{\langle \sigma^{-1}(g\tau)^2\rangle}$, each of length $\frac{r(m+2p)}{(2r,m+2l)}$. Similarly it can be shown that there are $(r,\frac{m}{2}+l)$ $\langle \sigma^{-1}(g\tau)^2\rangle$-orbits of $E_{rl}$ of the form $(g\tau(b),j)^{\langle \sigma^{-1}(g\tau)^2\rangle}$, each of length $\frac{r(m+2p)}{(2r,m+2l)}$.

Let $M$ be the minimal positive integer such that $(\sigma^{-1}(g\tau)^2)^M(\tau(b),j)=(\tau(b),j)$. Since $\frac{m}{2}+q$ is the minimal positive integer such that $(\sigma^{-1}(g\tau)^2)^{\frac{m}{2}+q}(\tau(b))=\tau(b)$, we have $(\frac{m}{2}+q)\mid M$. Similar to the case where $m$ is odd, we have
$(g\tau)^{m+2q}(\tau(b),j)=(\tau(b),j+q-l)$. Then
$$(\sigma^{-1}(g\tau)^2)^{\frac{m}{2}+q}(\tau(b),j)=\sigma^{-(\frac{m}{2}+q)}(g\tau)^{m+2q}(\tau(b),j)=(\tau(b),j+q-l-(\frac{m}{2}+q))=(\tau(b),j-(\frac{m}{2}+l)).$$
Since $M'=\frac{r}{(r,\frac{m}{2}+l)}$ is the minimal positive integer such that $((\sigma^{-1}(g\tau)^2)^{\frac{m}{2}+q})^{M'}(\tau(b),j)=(\tau(b),j)$, $$M=(\frac{m}{2}+q)M'=\frac{r(m+2q)}{(2r,m+2l)}.$$
Moreover, two half-edges $(\tau(b),j_1)$, $(\tau(b),j_2)$ of $E_{rl}$ belong to the same $\langle \sigma^{-1}(g\tau)^2\rangle$-orbit if and only if $(r,\frac{m}{2}+l)$ divides $j_1-j_2$. Therefore there are $(r,\frac{m}{2}+l)$ $\langle \sigma^{-1}(g\tau)^2\rangle$-orbits of $E_{rl}$ of the form $(\tau(b),j)^{\langle \sigma^{-1}(g\tau)^2\rangle}$, each of length $\frac{r(m+2q)}{(2r,m+2l)}$. Similarly it can be shown that there are $(r,\frac{m}{2}+l)$ $\langle \sigma^{-1}(g\tau)^2\rangle$-orbits of $E_{rl}$ of the form $(g\cdot b,j)^{\langle \sigma^{-1}(g\tau)^2\rangle}$, each of length $\frac{r(m+2q)}{(2r,m+2l)}$.
\end{proof}

\begin{Prop} \label{stable-AR-component-case-(3)}
Let $E\cong E_{rl}$ be a finite connected $f_{ms}$-BG such that $E/\langle\sigma\rangle$ is a Brauer $G$-set in case $(3)$ of Lemma \ref{B}, where the length of the unique cycle of $B=E/\langle\sigma\rangle$ is $m$ and the number of edges of $B$ outside (resp. inside) this cycle is $p$ (resp. $q$). If $m$ is odd, then $\prescript{}{s}{\Gamma}_{A_E}$ is a disjoint union of $(r,m+2l)$ components of the form $\mathbb{Z}A_{\infty}/\langle\tau^{\frac{r(m+2p)}{(r,m+2l)}}\rangle$, $(r,m+2l)$ components of the form $\mathbb{Z}A_{\infty}/\langle\tau^{\frac{r(m+2q)}{(r,m+2l)}}\rangle$, $(r,m+2l)$ components of the form $\mathbb{Z}\widetilde{A}_{\frac{r(m+2p)}{(r,m+2l)},\frac{r(m+2q)}{(r,m+2l)}}$, and infinitely many
components of the form $\mathbb{Z}A_{\infty}/\langle\tau\rangle$. If $m$ is even, then $\prescript{}{s}{\Gamma}_{A_E}$ is a disjoint union of $(2r,m+2l)$ components of the form $\mathbb{Z}A_{\infty}/\langle\tau^{\frac{r(m+2p)}{(2r,m+2l)}}\rangle$, $(2r,m+2l)$ components of the form $\mathbb{Z}A_{\infty}/\langle\tau^{\frac{r(m+2q)}{(2r,m+2l)}}\rangle$, $(2r,m+2l)$ components of the form $\mathbb{Z}\widetilde{A}_{\frac{r(m+2p)}{(2r,m+2l)},\frac{r(m+2q)}{(2r,m+2l)}}$, and infinitely many
components of the form $\mathbb{Z}A_{\infty}/\langle\tau\rangle$.
\end{Prop}

\begin{proof}
We may assume that $m$ is odd, since when $m$ is even the proof is similar. According to Lemma \ref{length-of-DTr-orbit-case-3}, $E\cong E_{rl}$ contains $(r,m+2l)$ $\langle \sigma^{-1}(g\tau)^2\rangle$-orbits of length $\frac{r(m+2p)}{(r,m+2l)}$ and $(r,m+2l)$ $\langle \sigma^{-1}(g\tau)^2\rangle$-orbits of length $\frac{r(m+2q)}{(r,m+2l)}$. By Proposition \ref{exceptional-tubes}, $\prescript{}{s}{\Gamma}_{A_E}$ contains $(r,m+2l)$ exceptional tubes of the form $\mathbb{Z}A_{\infty}/\langle\tau^{\frac{r(m+2p)}{(r,m+2l)}}\rangle$ and $(r,m+2l)$ exceptional tubes of the form $\mathbb{Z}A_{\infty}/\langle\tau^{\frac{r(m+2q)}{(r,m+2l)}}\rangle$. Since $B=E/\langle\sigma\rangle$ is a Brauer graph with a unique cycle of length $m$ and the number of edges of $B$ outside (resp. inside) this cycle is $p$ (resp. $q$), according to \cite[Corollary 4.7]{D}, the $\mathbb{Z}\widetilde{A}_{a,b}$ component of the stable AR-quiver of $A_B$ is given by $a=m+2p$, $b=m+2q$.

According to \cite[Theorem 2.1]{ES}, there are positive integers $M$, $c$, $d$ such that $\prescript{}{s}{\Gamma}_{A_E}$ is a disjoint union of $M$ components of the form $\mathbb{Z}\widetilde{A}_{c,d}$, $M$ components of the form $\mathbb{Z}A_{\infty}/\langle\tau^c\rangle$, $M$ components of the form $\mathbb{Z}A_{\infty}/\langle\tau^d\rangle$, and infinitely many components of the form $\mathbb{Z}A_{\infty}/\langle\tau\rangle$. We need to show that $M=(r,m+2l)$, $c=\frac{r(m+2p)}{(r,m+2l)}$, $d=\frac{r(m+2q)}{(r,m+2l)}$. If $\frac{r(m+2p)}{(r,m+2l)}>1$, $\frac{r(m+2q)}{(r,m+2l)}>1$, and $\frac{r(m+2p)}{(r,m+2l)}\neq\frac{r(m+2q)}{(r,m+2l)}$, then clearly we have $M=(r,m+2l)$, $c=\frac{r(m+2p)}{(r,m+2l)}$, $d=\frac{r(m+2q)}{(r,m+2l)}$. If $\frac{r(m+2p)}{(r,m+2l)}=\frac{r(m+2q)}{(r,m+2l)}>1$, then there are two possible cases:
\begin{itemize}
\item[$(i)$] $M=(r,m+2l)$, $c=d=\frac{r(m+2p)}{(r,m+2l)}$;
\item[$(ii)$] $M=2(r,m+2l)$, $c=\frac{r(m+2p)}{(r,m+2l)}$, $d=1$.
\end{itemize}
Since $m+2p=m+2q$ and since the $\mathbb{Z}\widetilde{A}_{a,b}$ component of the stable AR-quiver of $A_B$ is given by $a=m+2p$, $b=m+2q$, by Lemma \ref{ZApq-components} we have $c=d$, so case $(ii)$ can not occur. If one (but not both) of the numbers $\frac{r(m+2p)}{(r,m+2l)}$, $\frac{r(m+2q)}{(r,m+2l)}$ is equal to $1$, then we may assume that $\frac{r(m+2p)}{(r,m+2l)}>1$ and $\frac{r(m+2q)}{(r,m+2l)}=1$. There are two possible cases:
\begin{itemize}
\item[$(i)$] $M=(r,m+2l)$, $c=\frac{r(m+2p)}{(r,m+2l)}$, $d=1$;
\item[$(ii)$] $M=\frac{(r,m+2l)}{2}$, $c=d=\frac{r(m+2p)}{(r,m+2l)}$.
\end{itemize}
Since $m+2p\neq m+2q$ and since the $\mathbb{Z}\widetilde{A}_{a,b}$ component of the stable AR-quiver of $A_B$ is given by $a=m+2p$, $b=m+2q$, by Lemma \ref{ZApq-components} we have $c\neq d$, so case $(ii)$ can not occur. If $\frac{r(m+2p)}{(r,m+2l)}=\frac{r(m+2q)}{(r,m+2l)}=1$, then we have $m=1$, $p=q=0$, and $r|(1+2l)$. It is straightforward to show that $M=r$ and $c=d=1$.
\end{proof}

\end{document}